\documentclass[12pt,a4paper]{amsart}

\usepackage{amsmath}
\usepackage{mathtools}
\usepackage{amsthm}
\usepackage{amssymb}

\newcommand*\fullref[3][\relax]{%
  \ifdefined\hyperref%
    {\hyperref[#3]{#2\penalty 200\ \ref*{#3}#1}}%
  \else%
    {#2\penalty 200\ \relax\ref{#3}#1}%
  \fi%
}

\usepackage{tikz}
\usetikzlibrary{calc}

\usetikzlibrary{matrix}

\tikzset{
  pretableaumatrix/.style={
    ampersand replacement=\&,
    matrix of math nodes,
    outer sep=1mm,
    inner sep=0mm,
    anchor=center,
    row sep={between borders,-\pgflinewidth},
    column sep={between borders,-\pgflinewidth},
    dottedentry/.style={densely dotted},
  },
  pretableaunode/.style={
    font=\small,
    draw=gray,
    anchor=base,
    text height=3.75mm,
    text depth=1.25mm,
    minimum height=5mm,
    minimum width=5mm,
    inner sep=0mm,
    outer sep=0mm,
  },
  tableaumatrix/.style={
    pretableaumatrix,
    every node/.append style={
      pretableaunode,
    },
  },
  medtableaumatrix/.style={
    pretableaumatrix,
    every node/.append style={
      pretableaunode,
      font=\footnotesize,
      text height=2.75mm,
      text depth=.75mm,
      minimum height=3.5mm,
      minimum width=3.5mm
    },
  },
  smalltableaumatrix/.style={
    pretableaumatrix,
    every node/.append style={
      pretableaunode,
      font=\scriptsize,
      text height=1.85mm,
      text depth=.15mm,
      minimum height=2.5mm,
      minimum width=2.5mm,
    },
  },
  tinytableaumatrix/.style={
    pretableaumatrix,
    every node/.append style={
      pretableaunode,
      font=\tiny,
      text height=1.25mm,
      text depth=.15mm,
      minimum height=1.75mm,
      minimum width=1.75mm
    },
  },
  tableau/.style={
    baseline=-1.25mm,
    every matrix/.style={tableaumatrix},
  },
  preshapetableaumatrix/.style={
    pretableaumatrix,
    execute at end cell={\strut},
    every node/.append style={
      draw=black,
      anchor=base,
      inner sep=0mm,
      outer sep=0mm,
    },
    shadedentry/.style={fill=gray},
    darkshadedentry/.style={fill=darkgray},
  },
  medshapetableaumatrix/.style={
    preshapetableaumatrix,
    every node/.append style={
      text height=2.75mm,
      text depth=.75mm,
      minimum height=3.5mm,
      minimum width=3.5mm
    },
  },
  shapetableaumatrix/.style={
    ampersand replacement=\&,
    matrix of math nodes,
    outer sep=0mm,
    inner sep=0mm,
    anchor=base,
    row sep={between borders,-\pgflinewidth},
    column sep={between borders,-\pgflinewidth},
    execute at begin cell={\strut},
    every node/.append style={draw,anchor=base,text height=1mm,text depth=.5mm,minimum size=1.5mm,inner sep=0mm,outer sep=0mm},
  },
  shapetableau/.style={
    every matrix/.style={shapetableaumatrix},
  },
  topalign/.style={
    every matrix/.append style={name=maintableau,anchor=maintableau-1-1.base},
    baseline,
  },
}

\newcommand*\tableau[2][]{\tikz[tableau,#1]\matrix{#2};}

\theoremstyle{definition}
\newtheorem{definition}{Definition}[section]
\newtheorem{algorithm}[definition]{Algorithm}

\newtheorem{question}[definition]{Question}
\newtheorem{remark}[definition]{Remark}

\theoremstyle{plain}
\newtheorem{corollary}[definition]{Corollary}
\newtheorem{lemma}[definition]{Lemma}
\newtheorem{proposition}[definition]{Proposition}
\newtheorem{theorem}[definition]{Theorem}

\numberwithin{equation}{section}

\newcommand*{\defterm}[1]{\textit{#1}}

\newcommand{\textparens}[1]{\textrm{(}#1\textrm{)}}

\DeclarePairedDelimiter{\set}{\{}{\}}
\DeclarePairedDelimiter{\mset}{\langle}{\rangle}
\DeclarePairedDelimiterX{\gset}[2]{\{}{\}}{\,#1:#2\,}
\newcommand\gsetsplit[3][]{\mathopen#1\{\,#2:#3\,\mathclose#1\}}

\DeclarePairedDelimiter{\abs}{\lvert}{\rvert}
\DeclarePairedDelimiter{\brackets}{\lbrack}{\rbrack}

\DeclarePairedDelimiter{\parens}{\lparen}{\rparen}

\newcommand*{\emptyword}{\varepsilon}

\newcommand*{\nset}{\mathbb{N}}
\newcommand*{\zset}{\mathbb{Z}}

\newcommand*{\cset}{\mathbb{C}}

\DeclarePairedDelimiterX{\pres}[2]{\langle}{\rangle}{\,#1\,\delimsize\vert\,\mathopen{}#2\,}

\newcommand*{\drel}[1]{\mathcal{#1}}

\newcommand*\e{\ddot{e}}
\newcommand*\f{\ddot{f}}
\newcommand*\g{\ddot{g}}

\newcommand*\ecount{\ddot\epsilon}
\newcommand*\fcount{\ddot\phi}

\newcommand*\ke{\tilde{e}}
\newcommand*\kf{\tilde{f}}

\newcommand*\kecount{\tilde\epsilon}
\newcommand*\kfcount{\tilde\phi}

\newcommand*{\aA}{\mathcal{A}}

\newcommand*{\std}[2][]{\mathrm{std}\parens[#1]{#2}}

\newcommand*{\plen}[2][]{\ell\parens[#1]{#2}}
\newcommand*{\pwt}[2][]{\abs[#1]{#2}}

\newcommand*{\clen}[2][]{\ell\parens[#1]{#2}}
\newcommand*{\cwt}[2][]{\abs[#1]{#2}}

\newcommand*{\wt}[2][]{\wtlit\parens[#1]{#2}}
\newcommand*{\wtlit}{\mathrm{wt}}

\newcommand*{\descomp}[2][]{\mathcal{DC}\parens[#1]{#2}}

\newcommand*{\schinv}[1]{{#1}^\sharp}
\newcommand*{\schinvlit}{{}^\sharp}

\renewcommand*{\P}[2][]{\mathrm{P}\parens[#1]{#2}}
\newcommand*{\Q}[2][]{\mathrm{Q}\parens[#1]{#2}}

\newcommand*{\tabloid}[2][]{\mathrm{Toid}\parens[#1]{#2}}
\newcommand*{\colreading}[2][]{\mathrm{C}\parens[#1]{#2}}

\newcommand*{\qrtabloid}[2][]{\mathrm{QRoid}\parens[#1]{#2}}
\newcommand*{\qrtableau}[2][]{\mathrm{QR}\parens[#1]{#2}}
\newcommand*{\recribbon}[2][]{\mathrm{RR}\parens[#1]{#2}}

\newcommand*{\plac}{{\smash{\mathrm{plac}}}}
\newcommand*{\hypo}{{\smash{\mathrm{hypo}}}}

\newcommand*{\placcong}{\equiv_\plac}
\newcommand*{\hypocong}{\equiv_\hypo}

\newcommand*\ghost[2]{\mathrlap{#1}\phantom{#2}}

\begin{document}

\title[Crystallizing the hypoplactic monoid]{Crystallizing the hypoplactic monoid: from quasi-Kashiwara operators to the
  Robinson--Schensted--Knuth-type correspondence for quasi-ribbon tableaux}

\author{Alan J. Cain}
\address{%
Centro de Matem\'{a}tica e Aplica\c{c}\~{o}es\\
Faculdade de Ci\^{e}ncias e Tecnologia\\
Universidade Nova de Lisboa\\
2829--516 Caparica\\
Portugal
}
\email{%
a.cain@fct.unl.pt
}
\thanks{The first author was supported by an Investigador {\sc FCT} fellowship ({\sc IF}/01622/2013/{\sc CP}1161/{\sc
    CT}0001).}

\author{Ant\'onio Malheiro}
\address{%
Centro de Matem\'{a}tica e Aplica\c{c}\~{o}es\\
Faculdade de Ci\^{e}ncias e Tecnologia\\
Universidade Nova de Lisboa\\
2829--516 Caparica\\
Portugal
}
\address{%
Departamento de Matem\'{a}tica\\
Faculdade de Ci\^{e}ncias e Tecnologia\\
Universidade Nova de Lisboa\\
2829--516 Caparica\\
Portugal
}
\email{%
ajm@fct.unl.pt
}
\thanks{For both authors, this work was partially supported by the Funda\c{c}\~{a}o para
a Ci\^{e}ncia e a Tecnologia (Portuguese Foundation for Science and Technology) through the project {\sc UID}/{\sc
  MAT}/00297/2013 (Centro de Matem\'{a}tica e Aplica\c{c}\~{o}es).}

\begin{abstract}
  Crystal graphs, in the sense of Kashiwara, carry a natural monoid structure given by identifying words labelling
  vertices that appear in the same position of isomorphic components of the crystal. In the particular case of the
  crystal graph for the $q$-analogue of the special linear Lie algebra $\mathfrak{sl}_{n}$, this monoid is the celebrated
  plactic monoid, whose elements can be identified with Young tableaux. The crystal graph and the
  so-called Kashiwara operators interact beautifully with the combinatorics of Young tableaux and with the
  Robinson--Schensted--Knuth correspondence and so provide powerful combinatorial tools to work with them. This paper
  constructs an analogous `quasi-crystal' structure for the hypoplactic monoid, whose elements can be identified with
  quasi-ribbon tableaux and whose connection with the theory of quasi-symmetric functions echoes the connection of the
  plactic monoid with the theory of symmetric functions. This quasi-crystal structure and the associated quasi-Kashiwara
  operators are shown to interact just as neatly with the combinatorics of quasi-ribbon tableaux and with the
  hypoplactic version of the Robinson--Schensted--Knuth correspondence. A study is then made of the interaction of the crystal
  graph of the plactic monoid and the quasi-crystal graph for the hypoplactic monoid. Finally, the quasi-crystal
  structure is applied to prove some new results about the hypoplactic monoid.
\end{abstract}


\maketitle

\tableofcontents

\section{Introduction}

A crystal basis, in the sense of Kashiwara \cite{kashiwara_oncrystalbases,kashiwara_crystalizing}, is (informally) a
basis for a representation of a suitable algebra on which the generators have a particularly neat action. It gives rise,
via tensor products, to the crystal graph, which carries a natural monoid structure given by identifying words labelling
vertices that appear in the same position of isomorphic components. The ubiquitous plactic monoid, whose elements can be
viewed as semistandard Young tableaux, and which appears in such diverse contexts as symmetric functions
\cite{macdonald_symmetric}, representation theory and algebraic combinatorics \cite{fulton_young,lothaire_algebraic},
Kostka--Foulkes polynomials \cite{lascoux_plaxique,lascoux_foulkes}, Schubert polynomials
\cite{lascoux_schubert,lascoux_tableaux}, and musical theory \cite{jedrzejewski_plactic}, arises in this way from the
crystal basis for the $q$-analogue of the special linear Lie algebra $\mathfrak{sl}_{n}$. The crystal graph and the
associated Kashiwara operators interact beautifully with the combinatorics of Young tableaux and with the
Robinson--Schensted--Knuth correspondence and so provide powerful combinatorial tools to work with them.

This paper is dedicated to constructing an analogue of this crystal structure for the monoid of quasi-ribbon tableaux:
the so-called hypoplactic monoid. To explain this aim in more detail, and in particular to describe the properties such
an analogue should enjoy, it is necessary to briefly recapitulate some of the theory of crystals and of Young tableaux.

The plactic monoid of rank $n$ (where $n \in \nset$) arises by factoring the free monoid $\aA_n^*$ over the ordered
alphabet $\aA_n = \set{1 < \ldots < n}$ by a relation $\placcong$, which can be defined in various ways. Using
Schensted's algorithm \cite{schensted_longest}, which was originally intended to find longest increasing and decreasing
subsequences of a given sequence, one can compute a (semistandard) Young tableau $\P{w}$ from a word $w \in \aA_n^*$ and so define
$\placcong$ as relating those words that yield the same Young tableau. Knuth made a study of correspondences between
Young tableaux and non-negative integer matrices and gave defining relations for the plactic monoid
\cite{knuth_permutations}; the relation $\placcong$ can be viewed as the congruence generated by these defining
relations.

Lascoux \& Sch\"utzenberger \cite{lascoux_plaxique} began the systematic study of the plactic monoid, and, as remarked
above, connections have emerged with myriad areas of mathematics, which is one of the reasons Sch\"utzen\-berger
proclaimed it `one of the most fundamental monoids in algebra' \cite{schutzenberger_pour}. Of particular interest for us
is how it arises from the crystal basis for the $q$-analogue of the special linear Lie algebra $\mathfrak{sl}_{n}$ (that
is, the type $A_{n+1}$ simple Lie algebra), which links it to Kashiwara's theory of crystal bases
\cite{kashiwara_classical}. Isomorphisms between connected components of the crystal graph correspond to the relation
$\placcong$. Viewed on a purely combinatorial level, the Kashiwara operators and crystal graph are important tools for
working with the plactic monoid. (Indeed, in this context they are sometimes called `coplactic' operators
\cite[ch.~5]{lothaire_algebraic}, being in a sense `orthogonal' to $\placcong$.) Similarly, crystal theory can also be
used to analyse the analogous `plactic monoids' that index representations of the $q$-analogues of symplectic Lie
algebras $\mathfrak{sp}_n$ (the type $C_n$ simple Lie algebra), special orthogonal Lie algebras of odd and even rank
$\mathfrak{so}_{2n+1}$ and $\mathfrak{so}_{2n}$ (the type $B_n$ and $D_n$ simple Lie algebras), and the exceptional
simple Lie algebra $G_2$ (see \cite{kim_insertion,lecouvey_cn,lecouvey_bndn} and the survey \cite{lecouvey_survey}). The
present authors and Gray applied this crystal theory to construct finite complete rewriting systems and biautomatic
structures for all these plactic monoids \cite{cgm_crystal}; Hage independently constructed a finite complete rewriting
system for the plactic monoid of type $C_n$ \cite{hage_typec}.

As is described in detail later in the paper, the crystal structure meshes neatly with the Robinson--Schensted--Knuth
correspondence. This correspondence is a bijection $w \leftrightarrow (P,Q)$ where $w$ is a word over $\aA_n$, and $P$
is a semistandard Young tableau with entries in $\aA_n$ and $Q$ is a standard Young tableau of the same shape. (The
semistandard Young tableau $P$ is the tableau $\P{w}$ computed by Schensted's algorithm; the standard tableau $Q$ can be
computed in parallel.) Essentially, the standard Young tableau $Q$ corresponds to the connected component of the crystal
graph in which the word $w$ lies, and the semistandard Young tableau $P$ corresponds to the position of $w$ in that
component. By holding $Q$ fixed and varying $P$ over semistandard tableaux of the same shape, one obtains all words in a
given connected component. Consequently, all words in a given connected component correspond to tableaux of the same
shape.

In summary, there are three equivalent approaches to the plactic monoid:
\begin{itemize}
\item[P1.] \textit{Generators and relations}: the plactic monoid is defined by the presentation
  $\pres[\big]{\aA_n}{\drel{R}_\plac}$, where
\begin{align*}
\drel{R}_\plac ={}& \gset[\big]{(acb,cab)}{a \leq b < c} \\
&\cup \gset[\big]{(bac, bca)}{a < b \leq c}.
\end{align*}
Equivalently, $\placcong$ is the congruence on $\aA_n^*$ generated by $\drel{R}_\plac$.
\item[P2.] \textit{Tableaux and insertion}: the relation $\placcong$ is defined by $u \placcong v$ if and only if $\P{u} = \P{v}$,
  where $\P{\cdot}$ is the Young tableau computed using the Schensted insertion algorithm (see
  \fullref{Algorithm}{alg:placticinsert} below).
\item[P3.] \textit{Crystals}: the relation $\placcong$ is defined by $u \placcong v$ if and only if there is a crystal
  isomorphism between connected components of the crystal graph that maps $u$ onto $v$.
\end{itemize}

The defining relations in $\drel{R}_\plac$ (known as the Knuth relations) are the reverse of the ones given in
\cite{cgm_crystal}. This is because, in the context of crystal bases, the convention for tensor products gives rise to a
`plactic monoid' that is actually anti-isomorphic to the usual notion of plactic monoid. Since this paper is mainly
concerned with combinatorics, rather than representation theory, it follows Shimozono \cite{shimozono_crystals} in using
the convention that is compatible with the usual notions of Young tableaux and the Robinson--Schensted--Knuth correspondence.

Another important aspect of the plactic monoid is its connection to the theory of symmetric polynomials. The Schur
polynomials with $n$ indeterminates, which are the irreducible polynomial characters of the general linear group
$\mathrm{GL}_n(\cset)$, are indexed by shapes of Young tableaux with entries in $\aA_n$, and they form a $\zset$-basis
for the ring of symmetric polynomials in $n$ indeterminates. The plactic monoid was applied to give the first rigorous
proof of the Littlewood--Richardson rule (see \cite{littlewood_group} and \cite[Appendix]{green_polynomial}), which is a
combinatorial rule for expressing a product of two Schur polynomials as a linear combination of Schur polynomials.

In recent years, there has emerged a substantial theory of non-commutative symmetric functions and quasi-symmetric functions;
see, for example, \cite{gelfand_noncommutative1,krob_noncommutative4,krob_noncommutative5}. Of particular interest for
this paper is the notion of quasi-ribbon polynomials, which form a basis for the ring of quasi-symmetric polynomials, just
as the Schur polynomials form a basis for the ring of symmetric polynomials. The quasi-ribbon polynomials are indexed by the
so-called quasi-ribbon tableaux. These quasi-ribbon tableaux have an insertion algorithm and an associated monoid called
the hypoplactic monoid, which was first studied in depth by Novelli \cite{novelli_hypoplactic}. The hypoplactic monoid
of rank $n$ arises by factoring the free monoid $\aA_n^*$ by a relation $\hypocong$, which, like $\placcong$, can be
defined in various ways. Using the insertion algorithm one can compute a quasi-ribbon tableau from a word and so define
$\hypocong$ as relating those words that yield the same quasi-ribbon tableau. Alternatively, the relation $\hypocong$ can be
viewed as the congruence generated by certain defining relations.

Thus there are two equivalent approaches to the hypoplactic monoid:
\begin{itemize}
\item[H1.] \textit{Generators and relations}: the hypoplactic monoid is defined by the presentation
  $\pres[\big]{\aA_n}{\drel{R}_\hypo}$, where
\begin{align*}
\drel{R}_\hypo ={}& \drel{R}_\plac \\
&\cup \gset[\big]{(cadb,acbd)}{a \leq b < c \leq d} \\
&\cup \gset[\big]{(bdac,dbca)}{a < b \leq c < d}.
\end{align*}
Equivalently, $\hypocong$ is the congruence on the free monoid $\aA_n^*$ generated by $\drel{R}_\hypo$.
\item[H2.] \textit{Tableaux and insertion}: the relation $\hypocong$ is defined by $u \hypocong v$ if and only if $\qrtableau{u} = \qrtableau{v}$,
  where $\qrtableau{{\cdot}}$ is the quasi-ribbon tableau computed using the Krob--Thibon insertion algorithm (see
  \fullref{Algorithm}{alg:hypoplacticinsert} below).
\end{itemize}
Krob \& Thibon \cite{krob_noncommutative4} proved the equivalence of H1 and H2, which are the direct analogues of P1 and
P2. Owing to the previous success in detaching crystal basis theory from its representation-theoretic foundation and
using it as a combinatorial tool for working with Young tableaux and plactic monoids, it seems worthwhile to try to find
an analogue of P3 for the hypoplactic monoid. Such an analogue should have the following form:
\begin{itemize}
\item[H3.] \textit{Quasi-crystals}: the relation $\hypocong$ is defined by $u \hypocong v$ if and only if there is a quasi-crystal
  isomorphism between connected components of the quasi-crystal graph that maps $u$ onto $v$.
\end{itemize}
The aim of this paper is to define quasi-Kashiwara operators on a purely combinatorial level, and so construct a
quasi-crystal graph as required for H3. As will be shown, the interaction of the combinatorics of quasi-ribbon tableaux
with this quasi-crystal graph will be a close analogue of the interaction of the combinatorics of Young tableaux with
the crystal graph (and is, in the authors' view, just as elegant).

In fact, a related notion of `quasi-crystal' is found in Krob \& Thibon \cite{krob_noncommutative5}; see also
\cite{hivert_hecke}. However, the Krob--Thibon quasi-crystal describes the restriction to quasi-ribbon tableaux (or more
precisely words corresponding to quasi-ribbon tableaux) of the action of the usual Kashiwara operators: it does not
apply to all words and so does not give rise to isomorphisms that can be used to define the relation $\hypocong$.

The paper is organized as follows: \fullref{Section}{sec:preliminaries} sets up notation and discusses some
preliminaries relating to words, partitions, compositions, and the notion of weight. \fullref{Section}{sec:plactic}
reviews, without proof, the basic theory of Young tableaux, the plactic monoid, Kashiwara operators, and the crystal
graph; the aim is to gather the elegant properties of the crystal structure for the plactic monoid that should be
mirrored in the quasi-crystal structure for the hypoplactic monoid. These properties will also be used in the study of
the interactions of the crystal and quasi-crystal graphs. \fullref{Section}{sec:quasiribbontableau} recalls the
definitions of quasi-ribbon tableaux and the hypoplactic monoid. \fullref{Section}{sec:quasikashiwara} states the
definition of the quasi-Kashiwara operators and the quasi-crystal graph, and shows that isomorphisms between connected
components of this quasi-crystal graph give rise to a congruence on the free
monoid. \fullref{Section}{sec:quasicrystalandhypo} proves that the corresponding factor monoid is the hypoplactic
monoid. \textit{En route}, some of the properties of the quasi-crystal graph are
established. \fullref{Section}{sec:robinsonschensted} studies how the quasi-crystal graph interacts with the hypoplactic
version of the Robinson--Schensted--Knuth correspondence. \fullref{Section}{sec:quasicrystalstructure} systematically studies
the structure of the quasi-crystal graph. It turns out to be a subgraph of the crystal graph for the plactic monoid, and
the interplay of the subgraph and graph has some very neat properties. Finally, the quasi-crystal structure is applied
to prove some new results about the hypoplactic monoid in \fullref{Section}{sec:applications}, including an analogy of
the hook-length formula.

\section{Preliminaries and notation}
\label{sec:preliminaries}

\subsection{Alphabets and words}
\label{subsec:words}

Recall that for any alphabet $X$, the free monoid (that is, the set of all words, including the empty word) on the
alphabet $X$ is denoted $X^*$. The empty word is denoted $\emptyword$. For any $u \in X^*$, the length of $u$ is denoted
$|u|$, and, for any $x \in X$, the number of times the symbol $x$ appears in $u$ is denoted $|u|_x$. Suppose
$u = u_1\cdots u_k \in X^*$ (where $u_h \in X$). For any $i \leq j$, the word $u_i\cdots u_j$ is a \defterm{factor} of
$u$. For any $i_1,\ldots,i_m \in \set{1,\ldots,k}$ such that $i_1 < i_2 < \ldots < i_m$, the word
$u_{i_1}u_{i_2}\cdots u_{i_m}$ is a \defterm{subsequence} of $u$. Note that factors must be made up of consecutive
letters, whereas subsequences may be made up of non-consecutive letters. [To minimize potential confusion, the term
`subword' is not used in this paper, since it tends to be synonymous with `factor' in semigroup theory, but with
`subsequence' in combinatorics on words.]

For further background on the free monoid, see \cite{howie_fundamentals}; for presentations, see
\cite{higgins_techniques,ruskuc_phd}.

Throughout this paper, $\aA$ will be the set of natural numbers viewed as an infinite ordered alphabet:
$\aA = \set{1 < 2 < 3 < \ldots}$. Further, $n$ will be a natural number and $\aA_n$ will be the set of the first $n$
natural numbers viewed as an ordered alphabet: $\aA_n = \set{1 < 2 < \ldots < n}$.

A word $u \in \aA^*$ is \defterm{standard} if it contains each symbol in $\set{1,\ldots,|u|}$ exactly once. Let
$u = u_1\cdots u_k$ be a standard word. The word $u$ is identified with the permutation $h \mapsto u_h$, and $u^{-1}$
denotes the inverse of this permutation (which is also identifed with a standard word).

Further, the \defterm{descent set} of the standard word $u$ is
$D(u) = \gset[\big]{h \in \set{1,\ldots,|u|-1}}{u_h > u_{h+1}}$.

Let $u \in \aA^*$. The \defterm{standardization} of $u$, denoted $\std{u}$, is the standard word obtained by the
following process: read $u$ from left to right and, for each $a \in \aA$, attach a subscript $h$ to the $h$-th
appearance of $a$. Symbols with attached subscripts are ordered by
\[
a_i < b_j \iff \parens{a < b} \lor \parens[\big]{\parens{a = b} \land \parens{i < j}}.
\]
Replace each symbol with an attached subscript in $u$ with the corresponding symbol of the same rank from $\aA$. The
resulting word is $\std{u}$. For example:
\begin{align*}
w ={}& \ghost{2}{2_1}\ghost{4}{4_1}\ghost{3}{3_1}\ghost{2}{2_2}\ghost{4}{4_1}\ghost{5}{5_1}\ghost{5}{5_2}\ghost{6}{6_1}\ghost{5}{5_3} \\
\rightsquigarrow{}& 2_14_13_12_24_25_15_26_15_3 \\
\std{w} ={}& \ghost{1}{2_1}\ghost{4}{4_1}\ghost{3}{3_1}\ghost{2}{2_2}\ghost{5}{4_1}\ghost{6}{5_1}\ghost{7}{5_2}\ghost{9}{6_1}\ghost{8}{5_3}
\end{align*}
For further background relating to standard words and standardizaton, see \cite[\S~2]{novelli_hypoplactic}.

\subsection{Compositions and partitions}
\label{subsec:comppart}

A \defterm{weak composition} $\alpha$ is a finite sequence $(\alpha_1,\ldots,\alpha_m)$ with terms in
$\nset \cup \set{0}$. The terms $\alpha_h$ up to the last non-zero term are the \defterm{parts} of $\alpha$. The
\defterm{length} of $\alpha$, denoted $\clen\alpha$, is the number of its parts. The \defterm{weight} of $\alpha$,
denoted $\cwt\alpha$, is the sum of its parts (or, equivalently, of its terms): $\cwt\alpha = \alpha_1+\cdots+\alpha_m$.
For example, if $\alpha = (3,0,4,1,0)$, then $\clen\alpha = 4$ and $\cwt\alpha=8$.  Identify weak compositions whose
parts are the same (that is, that differ only in a tail of terms $0$). For example, $(3,1,5,2)$ is identified with
$(3,1,5,2,0)$ and $(3,1,5,2,0,0,0)$. (Note that this identification does not create ambiguity in the notions of parts
and weight.) A \defterm{composition} is a weak composition whose parts are all in $\nset$. For a composition
$\alpha = (\alpha_1,\ldots,\alpha_{\clen\alpha})$, define
$D(\alpha) = \set{\alpha_1,\alpha_1+\alpha_2,\ldots,\alpha_1+\ldots+\alpha_{\clen{\alpha-}1}}$. For a standard word
$u \in \aA^*$, define $\descomp{u}$ to be the unique composition of weight $\cwt{u}$ such that $D(\descomp{u}) = D(u)$,
where $D(u)$ is as defined in \fullref{Subsection}{subsec:words}. For example, if $u = 143256798$, then
$D(u) = \set{2,3,8}$ and so $\descomp{u} = (2,1,5,1)$.

A \defterm{partition} $\lambda$ is a non-increasing finite sequence $(\lambda_1,\ldots,\lambda_m)$ with terms in
$\nset$.
The terms $\lambda_h$ are the \defterm{parts} of $\lambda$. The \defterm{length} of $\lambda$, denoted
$\plen\lambda$, is the number of its parts. The \defterm{weight} of $\lambda$, denoted $\pwt\lambda$, is the sum of
its parts: $\pwt\lambda = \lambda_1+\cdots+\lambda_m$. For example, if $\lambda = (5,3,2,2)$, then $\plen\lambda = 4$
and $\pwt\lambda = 12$.

\subsection{Weight}
\label{subsec:weight}

The \defterm{weight function} $\wtlit$ is, informally, the function that counts the number of times each symbol
appears in a word. More formally, $\wtlit$ is defined by
\[
\wtlit : \aA^* \to (\nset\cup\set{0})^\aA\qquad u \mapsto \parens[\big]{|u|_1,|u|_2,|u|_3,\ldots}.
\]
Clearly, $\wt{\cdot}$ has an infinite tail of components $0$, so only the prefix up to the last non-zero term is
considered; thus $\wt{\cdot}$ is a weak composition. For example, $\wt{542164325224} = (1,4,1,3,2,1)$. See
\cite[\S~2.1]{shimozono_crystals} for a discussion of the basic properties of weight functions.

Weights are compared using the following order:
\begin{equation}
\label{eq:weightorder}
\begin{aligned}
&(\alpha_1,\alpha_2,\ldots) \leq (\beta_1,\beta_2,\ldots) \\
&\qquad\iff (\forall k \in \nset)\parens[\Big]{\sum_{i=1}^k \alpha_i \leq \sum_{i=1}^k \beta_i}.
\end{aligned}
\end{equation}
When $\pwt\alpha = \pwt\beta$, this is the \defterm{dominance order} of partitions \cite[\S~7.2]{stanley_enumerative2}.
When $\wt{u} < \wt{v}$, one says that $v$ has \defterm{higher weight} than $u$ (and $u$ has \defterm{lower weight} than
$v$). For example,
\begin{multline*}
\wt{542164325224} = (1,4,1,3,2,1) \\
< (5,3,2,2,0,0) = \wt{432143212111};
\end{multline*}
that is, $432143212111$ has higher weight than $542164325224$. For the purposes of this paper, it will generally not be
necessary to compare weights using \eqref{eq:weightorder}; the important fact will be how the Kashiwara operators affect
weight.

\section{Crystals and the plactic monoid}
\label{sec:plactic}

This section recalls in detail the three approaches to the plactic monoid discussed in the introduction, and discusses
further the very elegant interaction of the crystal structure with the combinatorics of Young tableaux and in particular
with the Robinson--Schensted--Knuth correspondence. The aim is to lay out the various properties one would hope for in the
quasi-crystal structure for the hypoplactic monoid.

\subsection{Young tableaux and insertion}
\label{subsec:youngtableaux}

The \defterm{Young diagram} of shape $\lambda$, where $\lambda$ is a partition, is a grid of boxes, with $\lambda_h$
boxes in the $h$-th row, for $h = 1,\ldots,\plen\lambda$, with rows left-aligned. For example, the Young diagram of shape $(5,3,2,2)$ is
\[
\tableau{
\null \& \null \& \null \& \null \& \null \\
\null \& \null \& \null \\
\null \& \null \\
\null \& \null \\
}
\]
A Young diagram of shape $(1,1,\ldots,1)$ is said to be a \defterm{column diagram} or to have \defterm{column
  shape}. Note that, in this paper, Young diagrams are top-left aligned, with longer rows at the top and the parts of
the partition specifying row lengths from top to bottom. There is an alternative convention of bottom-left aligned Young
diagrams, where longer rows are at the bottom.

A \defterm{Young tableau} is a Young diagram that is filled with symbols from $\aA$ so that the entries in each row are
non-decreasing from left to right, and the entries in each column are increasing from top to bottom. For example, a
Young tableau of shape $(5,3,2,2)$ is
\begin{equation}
\label{eq:youngtableaueg}
\tableau{
1 \& 2 \& 2 \& 2 \& 4 \\
2 \& 3 \& 5 \\
4 \& 4 \\
5 \& 6 \\
}
\end{equation}
A Young tableau of shape $(1,1,\ldots,1)$ is a \defterm{column}. (That is, a column is a Young tableau of column shape.)

A \defterm{standard Young tableau} of shape $\lambda$ is a Young diagram that is filled with symbols from
$\set{1,\ldots,\pwt\lambda}$, with each symbol appearing exactly once, so that the entries in each row are increasing from
left to right, and the entries in each column are increasing from top to bottom. For example, a standard Young tableau
of shape $(5,3,2,2)$ is
\[
\tableau{
1 \& 3 \& 6 \& 7 \& 10 \\
2 \& 4 \& 12 \\
5 \& 9 \\
8 \& 11 \\
}
\]
A \defterm{tabloid} is an array formed by concatenating columns, filled with symbols from $\aA$ so that the entries in
each column are increasing from top to bottom. (Notice that there is no restriction on the relative heights of the
columns; nor is there a condition on the order of entries in a row.) An example of a tabloid is
\begin{equation}
\label{eq:tabloideg}
\tableau{
2 \& 1 \& 4 \& 1\& 2 \\
5 \& 3 \&   \& 2 \\
  \& 4 \&   \& 4 \\
  \& 6 \&   \& 5 \\
}
\end{equation}
Note that a tableau is a special kind of tabloid. The shape of a tabloid cannot in general be expressed using a partition.

The \defterm{column reading} $\colreading{T}$ of a tabloid $T$ is the word in $\aA^*$ obtained by proceeding through the columns,
from leftmost to rightmost, and reading each column from bottom to top. For example, the column readings of the tableau
\eqref{eq:youngtableaueg} and the tabloid \eqref{eq:tabloideg} are respectively $5421\,6432\,52\,2\,4$ and
$52\,6431\,4\,5421\,2$ (where the spaces are simply for clarity, to show readings of individual columns), as illustrated
below:
\[
\begin{tikzpicture}
  \matrix[tableaumatrix] (t) {
    1 \& 2 \& 2 \& 2 \& 4 \\
    2 \& 3 \& 5 \\
    4 \& 4 \\
    5 \& 6 \\
  };
  \draw[gray,thick,rounded corners=3mm,->]
  ($ (t-4-1) + (1.5mm,-6mm) $) -- ($ (t-1-1) + (1.5mm,5mm) $) --
  ($ (t-4-2) + (1.5mm,-5mm) $) -- ($ (t-1-2) + (1.5mm,5mm) $) --
  ($ (t-2-3) + (1.5mm,-5mm) $) -- ($ (t-1-3) + (1.5mm,5mm) $) --
  ($ (t-1-4) + (1.5mm,-5mm) $) -- ($ (t-1-4) + (1.5mm,5mm) $) --
  ($ (t-1-5) + (1.5mm,-5mm) $) -- ($ (t-1-5) + (1.5mm,6mm) $);
\end{tikzpicture}
\qquad
\begin{tikzpicture}
  \matrix[tableaumatrix] (t) {
    2 \& 1 \& 4 \& 1\& 2   \\
    5 \& 3 \&   \& 2       \\
      \& 4 \&   \& 4       \\
      \& 6 \&   \& 5       \\
  };
  \draw[gray,thick,rounded corners=3mm,->]
  ($ (t-2-1) + (1.5mm,-16mm) $) -- ($ (t-1-1) + (1.5mm,5mm) $) --
  ($ (t-4-2) + (1.5mm,-5mm) $) -- ($ (t-1-2) + (1.5mm,5mm) $) --
  ($ (t-1-3) + (1.5mm,-5mm) $) -- ($ (t-1-3) + (1.5mm,5mm) $) --
  ($ (t-4-4) + (1.5mm,-5mm) $) -- ($ (t-1-4) + (1.5mm,5mm) $) --
  ($ (t-1-5) + (1.5mm,-5mm) $) -- ($ (t-1-5) + (1.5mm,6mm) $);
\end{tikzpicture}
\]
Let $w \in \aA^*$, and let $w^{(1)}\cdots w^{(m)}$ be the factorization of $w$ into maximal decreasing factors. Let
$\tabloid{w}$ be the tabloid whose $h$-th column has height $|w^{(h)}|$ and is filled with the symbols of $w^{(h)}$, for
$h = 1,\ldots,m$. Then $\colreading{\tabloid{w}} = w$. (Note that the notion of reading described here is the one normally used in the
study of Young tableaux and the plactic monoid, and is the opposite of the `Japanese reading' used in the theory of
crystals. Throughout the paper, definitions follow the convention compatible with the reading defined here. The
resulting differences from the usual practices in crystal theory will be explicitly noted.)

If $w$ is the column reading of some Young tableau $T$, it is called a \defterm{tableau word}. Note that not all words
in $\aA^*$ are tableau words. For example, $343$ is not a tableau word, since $\tableau{3 \& 3 \\ 4 \\}$ and
$\tableau{3 \& 3 \& 4 \\}$ are the only tableaux containing the correct symbols, and neither of these has column
reading $343$. The word $w$ is a tableau word if and only if $\tabloid{w}$ is a tableau.

The plactic monoid arises from an algorithm that computes a Young tableau $\P{w}$ from a word $w \in \aA^*$.

\begin{algorithm}[Schensted's algorithm]
\label{alg:placticinsertone}
~\par\nobreak
\textit{Input:} A Young tableau $T$ and a symbol $a \in \aA$.

\textit{Output:} A Young tableau $T \leftarrow a$.

\textit{Method:}
\begin{enumerate}

\item If $a$ is greater than or equal to every entry in the topmost row of $T$, add $a$ as an entry at the rightmost end of
  $T$ and output the resulting tableau.

\item Otherwise, let $z$ be the leftmost entry in the top row of $T$ that is strictly greater than $a$. Replace $z$ by
  $a$ in the topmost row and recursively insert $z$ into the tableau formed by the rows of $T$ below the topmost. (Note
  that the recursion may end with an insertion into an `empty row' below the existing rows of $T$.)

\end{enumerate}
\end{algorithm}

Using an iterative form of this algorithm, one can start from a word $a_1\cdots a_k$ (where $a_i \in \aA$) and compute a
Young tableau $\P{a_1\cdots a_k}$. Essentially, one simply starts with the empty tableau and inserts the symbols $a_1$,
\ldots, $a_k$ in order: However, the algorithm described below also computes a standard Young tableau
$\Q{a_1\cdots a_k}$ that is used in the Robinson--Schensted--Knuth correspondence, which will shortly be described:

\begin{algorithm}
\label{alg:placticinsert}
~\par\nobreak
\textit{Input:} A word $a_1\cdots a_k$, where $a_i \in \aA$.

\textit{Output:} A Young tableau $\P{a_1\cdots a_k}$ and a
standard Young tableau $\Q{a_1\cdots a_k}$.

\textit{Method:} Start with an empty Young tableau $P_0$ and an empty standard Young tableau $Q_0$. For each $i = 1$,
\ldots, $k$, insert the symbol $a_i$ into $P_{i-1}$ as per \fullref{Algorithm}{alg:placticinsertone}; let $P_i$ be the
resulting Young tableau. Add a cell filled with $i$ to the standard tableaux $Q_{i-1}$ in the same place as the unique
cell that lies in $P_i$ but not in $P_{i-1}$; let $Q_i$ be the resulting standard Young tableau.

Output $P_k$ for $\P{a_1\cdots a_k}$ and $Q_k$ as $\Q{a_1\cdots a_k}$.
\end{algorithm}

For example, the sequence of pairs $(P_i,Q_i)$ produced during the application of
\fullref{Algorithm}{alg:hypoplacticinsert} to the word $2213$ is
\begin{multline*}
\parens{\;,\;},\;
\parens[\big]{\tableau{2\\},\tableau{1\\}},\;
\parens*{\tableau{2\&2\\},\tableau{1\&2\\}},\;\\
\parens*{\tableau{1\&2\\2\\},\tableau{1\&2\\3\\}},\;
\parens*{\tableau{1\&2\&3\\2\\},\tableau{1\&2\&4\\3\\}}.
\end{multline*}
Therefore $\P{4323} = \tableau{1\&2\&3\\2\\}$ and $\Q{4323} = \tableau{1\&2\&4\\3\\}$. It is straightforward to see that
the map $w \mapsto \parens[\big]{P(w),Q(w)}$ is a bijection between words in $\aA^*$ and pairs consisting of a Young
tableau over $\aA$ and a standard Young tableau of the same shape. This is the celebrated
\defterm{Robinson--Schensted--Knuth correspondence} (see, for example, \cite[Ch.~4]{fulton_young} or
\cite[\S~7.11]{stanley_enumerative2}.

The possible definition (approach~P2 in the introduction) of the relation $\placcong$ using tableaux and insertion is
the following:
\[
u \placcong v \iff \P{u} = \P{v}.
\]
Using this as a definition, it follows that $\placcong$ is a congruence on $\aA^*$ \cite{knuth_permutations}. The factor
monoid $\aA^*/{\placcong}$ is the \defterm{plactic monoid} and is denoted $\plac$. The relation $\placcong$ is the
\defterm{plactic congruence} on $\aA^*$. The congruence $\placcong$ naturally restricts to a congruence on $\aA_n^*$,
and the factor monoid $\aA_n^*/{\placcong}$ is the \defterm{plactic monoid of rank $n$} and is denoted $\plac_n$.

If $w$ is a tableau word, then $w = \colreading{\P{w}}$ and $\tabloid{w} = \P{w}$. Thus the tableau words in $\aA^*$
form a cross-section (or set of normal forms) for $\plac$, and the tableau words in $\aA_n^*$ form a cross-section for
$\plac_n$.

\subsection{Kashiwara operators and the crystal graph}
\label{subsec:kashiwara}

The following discussion of crystals and Kashiwara operators is restricted to the context of $\plac_n$. For a more
general introduction to crystal bases, see \cite{cgm_crystal}.

The \defterm{Kashiwara operators} $\ke_i$ and $\kf_i$, where $i \in \set{1,\ldots,n-1}$, are partially-defined operators
on $\aA_n^*$. In representation-theoretic terms, the operators $\ke_i$ and $\kf_i$ act on a tensor product by acting on
a single tensor factor \cite[\S~4.4]{hong_quantumgroups}. This action can be described in a combinatorial way using the
the so-called signature or bracketing rule. This paper describes the action directly using this rule, since the the
analogous quasi-Kashiwara operators are defined by modifying this rule.

The definitions of $\ke_i$ and $\kf_i$ start from the \defterm{crystal basis} for $\plac_n$, which will form a connected
component of the crystal graph:
\begin{equation}
\label{eq:placcrystalbasis}
\includegraphics{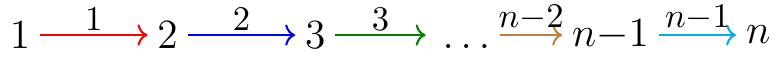}
\end{equation}
Each operator $\kf_i$ is defined so that it `moves' a symbol $a$ forwards along a directed edge labelled by $i$ whenever
such an edge starts at $a$, and each operator $\ke_i$ is defined so that it `moves' a symbol $a$ backwards along a
directed edge labelled by $i$ whenever such an edge ends at $a$:
\[
\includegraphics{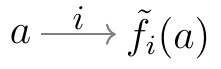}
\text{ and }
\includegraphics{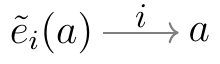}
\]
Using the crystal basis given above, one sees that
\begin{align*}
\ke_i(i+1) &= i, &&\text{$\ke_i(x)$ is undefined for $x \neq i+1$;} \\
\kf_i(i)&=i+1, &&\text{$\kf_i(x)$ is undefined for $x \neq i$.}
\end{align*}
The definition is extended to $\aA_n^* \setminus \aA_n$ by the recursion
\begin{align*}
  \ke_i(uv) &= \begin{cases}
    \ke_i(u)\,v & \text{if $\kecount_i(u) > \kfcount_i(v)$;} \\
    u\, \ke_i(v) & \text{if $\kecount_i(u) \leq \kfcount_i(v)$,}
  \end{cases}\displaybreak[0]\\
  \kf_i(uv) &= \begin{cases}
    \kf_i(u)\,v & \text{if $\kecount_i(u) \geq \kfcount_i(v)$;} \\
    u\,\kf_i(v) & \text{if $\kecount_i(u) < \kfcount_i(v)$,}
  \end{cases}
\end{align*}
where $\kecount_i$ and $\kfcount_i$ are auxiliary maps defined by
\begin{align*}
  \kecount_i(w) & = \max\gset[\big]{k \in \nset\cup\set{0}}{\text{$\underbrace{\ke_i\cdots\ke_i}_{\text{$k$ times}}(w)$ is defined}}; \\
  \kfcount_i(w) & = \max\gset[\big]{k \in \nset\cup\set{0}}{\text{$\underbrace{\kf_i\cdots\kf_i}_{\text{$k$ times}}(w)$ is defined}}.
\end{align*}
Notice that this definition is not circular: the definitions of $\ke_i$ and $\kf_i$ depend, via $\kecount_i$ and
$\kfcount_i$, only on $\ke_i$ and $\kf_i$ applied to \emph{strictly shorter} words; the recursion terminates with
$\ke_i$ and $\kf_i$ applied to single letters from the alphabet $\aA_n$, which was defined using the crystal basis
\eqref{eq:placcrystalbasis}. (Note that this definition is in a sense the mirror image of
\cite[Theorem~1.14]{kashiwara_classical}, because of the choice of definition for readings of tableaux used in this
paper. Thus the definition of $\ke_i$ and $\kf_i$ is the same as \cite[p.~8]{shimozono_crystals}.)

Although it is not immediate from the definition, the operators $\ke_i$ and $\kf_i$ are well-defined. Furthermore,
$\ke_i$ and $\kf_i$ are mutually inverse whenever they are defined, in the sense that if $\ke_i(w)$ is defined, then
$w = \kf_i(\ke_i(w))$, and if $\kf_i(w)$ is defined, then $w = \ke_i(\kf_i(w))$.

The \defterm{crystal graph} for $\plac_n$, denoted $\Gamma(\plac_n)$, is the directed labelled graph with vertex set
$\aA_n^*$ and, for $w,w' \in \aA_n^*$, an edge from $w$ to $w'$ labelled by $i$ if and only if $w' = \kf_i(w)$ (or,
equivalently, $w = \ke_i(w')$). \fullref{Figure}{fig:plac3crystal} shows part of the crystal graph
$\Gamma(\plac_3)$.

\begin{figure}[t]
  \centering
  \includegraphics{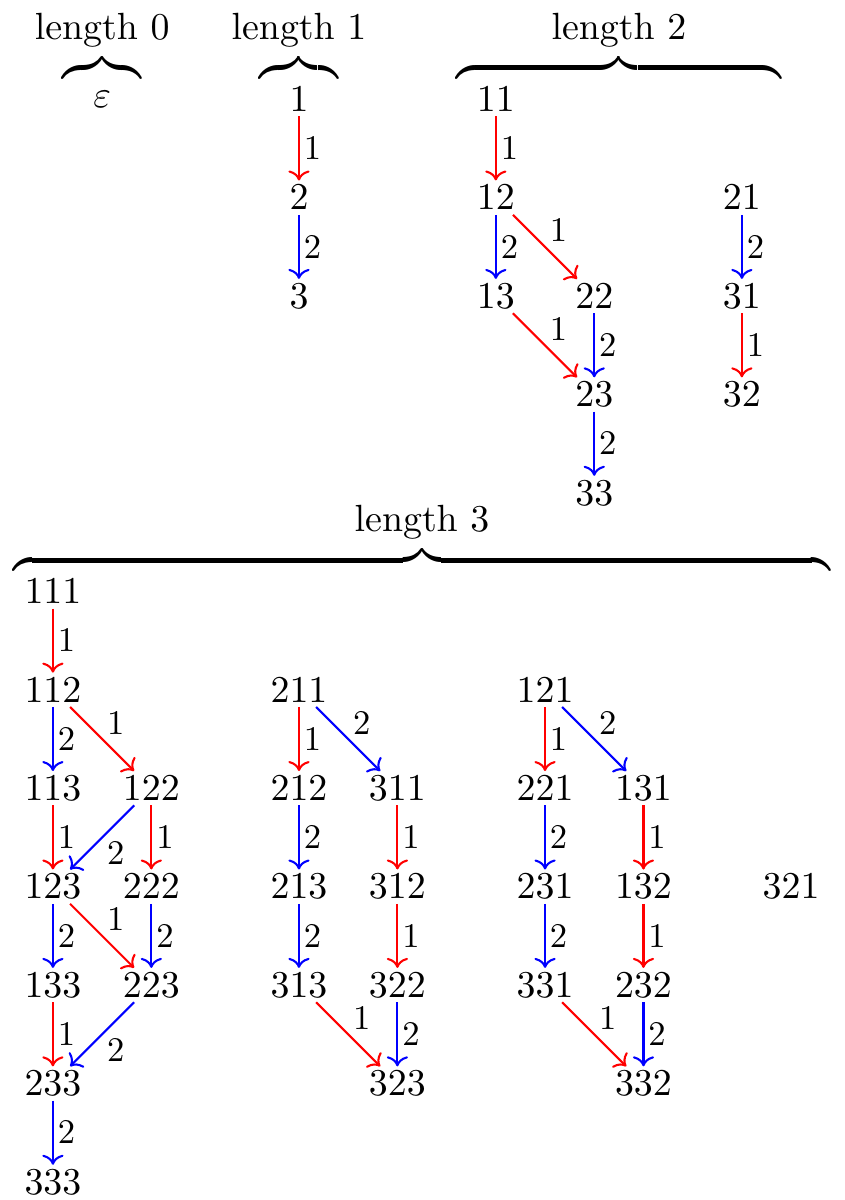}
  \caption{Part of the crystal graph for $\plac_3$. Note that each connected component consists of words of the same
    length. In particular, the empty word $\emptyword$ is an isolated vertex, and the words of length $1$ form a single
    connected component, which is the crystal basis for $\plac_3$. The two connected components whose highest-weight words are $211$
    and $121$ are isomorphic. However, the components consisting of the isolated vertex $321$ is not.}
  \label{fig:plac3crystal}
\end{figure}

Since the operators $\ke_i$ and $\kf_i$ preserve lengths of words, and since there are finitely many words in $\aA_n^*$
of each length, each connected component in the crystal graph must be finite.

For any $w \in \aA_n^*$, let $\Gamma(\plac_n,w)$ denote the connected component of $\Gamma(\plac_n)$ that contains the
vertex $w$. Notice that the crystal basis \eqref{eq:placcrystalbasis} is the connected component $\Gamma(\plac_n,1)$.

A \defterm{crystal isomorphism} between two connected components is a weight-preserving labelled digraph
isomorphism. That is, a map $\theta : \Gamma(\plac_n,w) \to \Gamma(\plac_n,w')$ is a crystal isomorphism if it has the
following properties:
\begin{itemize}
\item $\theta$ is bijective;
\item $\wt{\theta(u)} = \wt{u}$ for all $u \in \Gamma(\plac_n,w)$;
\item for all $u,v \in \Gamma(\plac_n,w)$, there is an edge $\includegraphics{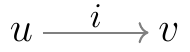}$
if and only if there is an edge $\includegraphics{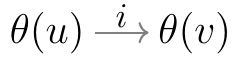}$.
\end{itemize}
The possible definition (approach~P3 in the introduction) of the relation $\placcong$ using the crystal graph
$\Gamma(\plac_n)$ is the following: two words in $\aA_n^*$ are related by $\placcong$ if and only if they lie in the
same place in isomorphic connected components. More formally, $w \placcong w'$ if and only if there is a crystal
isomorphism $\theta : \Gamma(\plac_n,w) \to \Gamma(\plac_n,w')$ such that $\theta(w) = w'$ (see, for example,
\cite{lecouvey_survey}). For example, $2213 \placcong 2231$, and these words appear in the same position in two of the
connected components shown in \fullref{Figure}{fig:plac3crystalisocomps}. (This figure also illustrates other properties
that will be discussed shortly.)

\begin{figure}[t]
  \centering
  \includegraphics{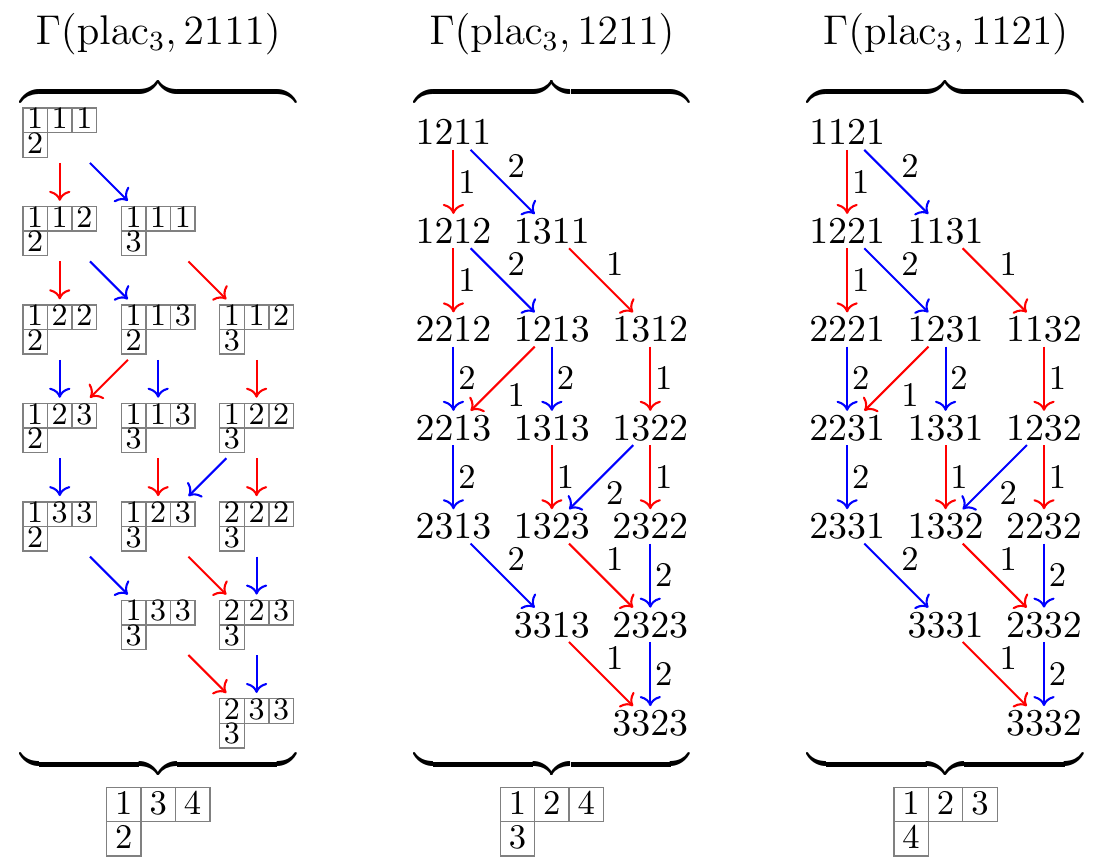}
  \caption{Three isomorphic components of the crystal graph for $\plac_3$. In the component containing column readings of tableaux, the
    tableaux themselves are shown instead of words. The standard tableaux below each component is $\Q{w}$ for all words
    $w$ in that component.}
  \label{fig:plac3crystalisocomps}
\end{figure}

\subsection{Computing the Kashiwara operators}
\label{subsec:computingkashiwara}

The recursive definition of the Kashiwara operators $\ke_i$ and $\kf_i$ given above is not particularly convenient for
practical computation. The following method, outlined in \cite{kashiwara_classical}, is more useful: Let
$i\in \set{1,\ldots,n-1}$ and let $w\in \aA_n^*$.  Form a new word in $\set{{+},{-}}^*$ by replacing each letter $i$ of
$w$ by the symbol $+$, each letter $i+1$ by the symbol $-$, and every other symbol with the empty word, keeping a record
of the original letter replaced by each symbol. Then delete factors ${-}{+}$ until no such factors remain: the
resulting word is ${+}^{\kfcount_i(w)}{-}^{\kecount_i(w)}$, and is denoted by $\rho_i(w)$. Note that factors ${+}{-}$
are \emph{not} deleted. (The method given in \cite{kashiwara_classical} involved deleting factors ${+}{-}$; again,
this difference is a consequence of the choice of convention for reading tableaux.)

If $\kecount_i(w)=0$, then $\ke_i(w)$ is undefined. If $\kecount_i(w)>0$ then one obtains $\ke_i(w)$ by taking the letter
$i+1$ which was replaced by the leftmost $-$ of $\rho_i(w)$ and changing it to $i$. If $\kfcount_i(w)=0$, then
$\kf_i(w)$ is undefined. If $\kfcount_i(w)>0$ then one obtains $\kf_i(w)$ by taking the letter $i$ which was replaced by
the rightmost $+$ of $\rho_i(w)$ and changing it to $i+1$.

For a purely combinatorial proof that this method of computation is correct, see
\cite[Proposition~2.1]{cgm_crystal}. (Note that \cite{cgm_crystal} uses the tableaux-reading convention from
representation theory, so that the result must be reflected to fit the convention used here.)

\subsection{Properties of the crystal graph}
\label{subsec:propertiesgammaplacn}

In the crystal graph $\Gamma(\plac_n)$, the length of the longest path consisting of edges labelled by
$i \in \set{1,\ldots,n-1}$ that ends at $w \in \aA_n^*$ is $\kecount_i(w)$. The length of the longest path consisting of
edges labelled by $i$ that starts at $w \in \aA_n^*$ is $\kfcount_i(w)$.

The operators $\ke_i$ and $\kf_i$ respectively increase and decrease weight whenever they are defined, in
the sense that if $\ke_i(w)$ is defined, then $\wt{\ke_i(w)} > \wt{w}$, and if $\kf_i(w)$ is defined, then
$\wt{\kf_i(w)} < \wt{w}$. This is because $\ke_i$ replaces a symbol $i+1$ with $i$ whenever it is defined, which
corresponds to decrementing the $i+1$-th component and incrementing the $i$-th component of the weight, which results in
an increase with respect to the order \eqref{eq:weightorder}. Similarly, $\kf_i$ replaces a symbol $i$ with $i+1$
whenever it is defined. For this reason, the $\ke_i$ and $\kf_i$ are respectively known as the Kashiwara
\defterm{raising} and \defterm{lowering} operators.

Every connected component in $\Gamma(\plac_n)$ contains a unique \defterm{highest-weight} vertex: a vertex whose weight
is higher than all other vertices in that component. This means that no Kashiwara raising operator $\ke_i$ is defined on
this vertex. See \cite[\S~2.4.2]{shimozono_crystals} for proofs and background. (The existence, but not the uniqueness,
of a highest-weight vertex is a consequence of the finiteness of connected components.)

Whenever they are defined, the operators $\ke_i$ and $\kf_i$ preserve the property of being a tableau word and the shape
of the corresponding tableau \cite{kashiwara_classical}. Furthermore, all the tableau words corresponding to tableaux of
a given shape with entries in $\aA_n$ lie in the same connected component. As shown in
\fullref{Figure}{fig:plac3crystalisocomps}, the left-hand component $\Gamma(\plac_3,2111)$ is made up of all the tableau
words corresponding to tableaux of shape $(3,1)$ with entries in $\aA_3$.

Each connected component in $\Gamma(\plac_n)$ corresponds to exactly one standard tableau, in the sense that
$\Q{w} = \Q{u}$ if and only if $u$ and $w$ lies in the same connected component of $\Gamma(\plac_n)$. In terms of the
bijection $w \mapsto (\P{w},\Q{w})$ of the Robinson--Schensted--Knuth correspondence, specifying $\Q{w}$ locates the particular
connected component $\Gamma(\plac_n,w)$, and specifying $\P{w}$ locates the word $w$ within that component.

Highest-weight words in $\Gamma(\plac_n)$, and in particular highest-weight tableau words, admit a useful
characterization as Yamanouchi words. A word $w_1\cdots w_m \in \aA_n^*$ (where $w_i \in \aA_n$) is a
\defterm{Yamanouchi word} if, for every $j = 1,\ldots,m$, the weight of the suffix $w_j\cdots w_m$ is a non-increasing
sequence (that is, a partition). Thus $1231$ is not a Yamanouchi word, since $\wt{31} = (1,0,1)$, but $(1121$) is a
Yamanouchi word, since $\wt{1121} = (3,1,0)$; $\wt{121} = (2,1,0)$; $\wt{21} = (1,1,0)$, and $\wt{1} = (1,0,0)$. A word
is highest-weight if and only if it is a Yamanouchi word. See \cite[Ch.~5]{lothaire_algebraic} for further background.

Highest-weight tableau words also have a neat characterization: a tableau word is highest-weight if and only if its
weight is equal to the shape of the corresponding tableau. That is, a tableau word whose corresponding tableau has shape
$\lambda$ is highest-weight if and only if, for each $i \in \aA_n$, the number of symbols $i$ it contains is
$\lambda_i$. It follows that a tableau whose reading is a highest-weight word must contain only symbols $i$ on its
$i$-th row, for all $i \in \set{1,\ldots,\plen\lambda}$. For example, the tableau of shape $(5,3,2,2)$ whose reading is
a highest-weight word is:
\[
\tableau{
1 \& 1 \& 1 \& 1 \& 1 \\
2 \& 2 \& 2 \\
3 \& 3 \\
4 \& 4 \\
}
\]
See \cite{kashiwara_classical} for further background.

\section{Quasi-ribbon tableaux and insertion}
\label{sec:quasiribbontableau}

This section gathers the relevant definitions and background on quasi-ribbon tableaux, the analogue of the
Robinson--Schensted--Knuth correspondence, and the hypoplactic monoid. For further background, see
\cite{krob_noncommutative4,novelli_hypoplactic}.

Let $\alpha = (\alpha_1,\ldots,\alpha_m)$ and $\beta = (\beta_1,\ldots,\beta_p)$ be compositions with
$\cwt\alpha = \cwt\beta$. Then $\beta$ is \defterm{coarser} than $\alpha$, denoted $\beta \preceq \alpha$, if each partial
sum $\beta_1+\ldots+\beta_{p'}$ (for $p' < p$) is equal to \emph{some} partial sum $\alpha_1+\ldots+\alpha_{m'}$ for
some $m' < m$. (Essentially, $\beta$ is coarser than $\alpha$ if it can be formed from $\alpha$ by `merging' consecutive
parts.) Thus $(11) \preceq (3,8) \preceq (3,6,2) \preceq (3,1,5,2)$.

A \defterm{ribbon diagram} of shape $\alpha$, where $\alpha$ is a composition, is an array of boxes, with $\alpha_h$
boxes in the $h$-th row, for $h = 1,\ldots,\clen\alpha$ and counting rows from top to bottom, aligned so that the leftmost cell in each row is below the
rightmost cell of the previous row. For example, the ribbon diagram
of shape $(3,1,5,2)$ is:
\begin{equation}
\label{eq:ribboneg}
\begin{tikzpicture}[baseline=(baselinemarker)]
  \matrix[tableaumatrix,name=mainmatrix]
  {
    \null \& \null \& \null                                        \\
    \&       \& \null                                              \\
    \&       \& \null \& \null \& \null \& \null \& \null          \\
    \&       \&       \&       \&       \&       \& \null \& \null \\
  };
  \node[font=\scriptsize,anchor=east] at (mainmatrix-1-1.west) {row $1$};
  \node[font=\scriptsize,anchor=east] at ($ (mainmatrix-2-3.west) + 2*(-5mm,0) $) {row $2$};
  \node[font=\scriptsize,anchor=east] at ($ (mainmatrix-3-3.west) + 2*(-5mm,0) $) {row $3$};
  \node[font=\scriptsize,anchor=east] at ($ (mainmatrix-4-7.west) + 6*(-5mm,0) $) {row $4$};
  \coordinate (baselinemarker) at ($ (mainmatrix-2-3.base) + (0,-2.5mm) $);
\end{tikzpicture}
\end{equation}
Notice that a ribbon diagram cannot contain a $2\times 2$ subarray (that is, of the
form $\tikz[shapetableau,baseline=0mm]\matrix{ \null \& \null \\ \null \& \null \\};$).

In a ribbon diagram of shape $\alpha$, the number of rows is $\clen\alpha$ and the number of boxes is $\cwt\alpha$.

A \defterm{quasi-ribbon tableau} of shape $\alpha$, where $\alpha$ is a composition, is a ribbon diagram of shape $\alpha$ filled with symbols from $\aA$ such that the entries in every row
are non-decreasing from left to right and the entries in every column are strictly increasing from top to bottom. An
example of a quasi-ribbon tableau is
\begin{equation}
\label{eq:qrteg}
\tikz[tableau]\matrix
{
 1 \& 2 \& 2                          \\
   \&   \& 3                          \\
   \&   \& 4 \& 4 \& 5 \& 5 \& 5      \\
   \&   \&   \&   \&   \&   \& 6 \& 7 \\
};
\end{equation}
Note the following immediate consequences of the definition of a quasi-ribbon tableau: (1) for each $a \in \aA$, the
symbols $a$ in a quasi-ribbon tableau all appear in the same row, which must be the $j$-th row for some $j \leq a$; (2)
the $h$-th row of a quasi-ribbon tableau cannot contain symbols from $\set{1,\ldots,h-1}$.

A \defterm{quasi-ribbon tabloid} is a ribbon diagram filled with symbols from $\aA$ such that the entries in every
column are strictly increasing from top to bottom. (Notice that there is no restriction on the order of entries in a
row.) An example of a quasi-ribbon tabloid is
\begin{equation}
\label{eq:qrtoideg}
\tikz[tableau]\matrix
{
 1 \& 5 \& 2                          \\
   \&   \& 3                          \\
   \&   \& 6 \& 2 \& 4 \& 5 \& 4      \\
   \&   \&   \&   \&   \&   \& 5 \& 7 \\
};
\end{equation}
Note that a quasi-ribbon tableau is a special kind of quasi-ribbon tabloid.

A \defterm{recording ribbon} of shape $\alpha$, where $\alpha$ is a composition, is a ribbon diagram of shape $\alpha$
filled with symbols from $\set{1,\ldots,\cwt\alpha}$, with each symbol appearing exactly once, such that the entries in
every row are increasing from left to right, and entries in every column are increasing from bottom to top. (Note that
the condition on the order of entries in rows is the same as in quasi-ribbon tableau, but the condition on the order of
entries in columns is the opposite of that in quasi-ribbon tableau.)  An example of a recording ribbon of shape
$(3,1,5,2)$ is
\begin{equation}
\label{eq:rreg}
\tikz[tableau]\matrix
{
 1 \& 2 \& 9                           \\
   \&   \& 8                           \\
   \&   \& 3 \& 4 \& 6 \& 7 \& 11      \\
   \&   \&   \&   \&   \&   \& 5 \& 10 \\
};
\end{equation}

The \defterm{column reading} $\colreading{T}$ of a quasi-ribbon tabloid $T$ is the word in $\aA^*$ obtained by proceeding through
the columns, from leftmost to rightmost, and reading each column from bottom to top. For example, the column reading of
the quasi-ribbon tableau \eqref{eq:qrteg} and the quasi-ribbon tabloid \eqref{eq:qrtoideg} are respectively
$1\,2\,432\,4\,5\,5\,65\,7$ and $1\,5\,632\,2\,4\,5\,54\,7$ (where the spaces are simply for clarity, to show readings
of individual columns), as illustrated below:
\[
\begin{tikzpicture}
  \matrix[tableaumatrix] at (0,0) (qrteg)
  {
    1 \& 2 \& 2                          \\
      \&   \& 3                          \\
      \&   \& 4 \& 4 \& 5 \& 5 \& 5      \\
      \&   \&   \&   \&   \&   \& 6 \& 7 \\
  };
  \draw[gray,thick,rounded corners=3mm,->]
  ($ (qrteg-1-1) + (-1.5mm,-20mm) $) -- ($ (qrteg-1-1) + (-1.5mm,7mm) $) --
  ($ (qrteg-1-2) + (-1.5mm,-7mm) $) -- ($ (qrteg-1-2) + (-1.5mm,7mm) $) --
  ($ (qrteg-3-3) + (-1.5mm,-7mm) $) -- ($ (qrteg-1-3) + (-1.5mm,7mm) $) --
  ($ (qrteg-3-4) + (-1.5mm,-7mm) $) -- ($ (qrteg-3-4) + (-1.5mm,7mm) $) --
  ($ (qrteg-3-5) + (-1.5mm,-7mm) $) -- ($ (qrteg-3-5) + (-1.5mm,7mm) $) --
  ($ (qrteg-3-6) + (-1.5mm,-7mm) $) -- ($ (qrteg-3-6) + (-1.5mm,7mm) $) --
  ($ (qrteg-4-7) + (-1.5mm,-7mm) $) -- ($ (qrteg-3-7) + (-1.5mm,7mm) $) --
  ($ (qrteg-4-8) + (-1.5mm,-7mm) $) -- ($ (qrteg-4-8) + (-1.5mm,7mm) $) --
  ($ (qrteg-4-8) + (-1.5mm,20mm) $);
\end{tikzpicture}
\qquad
\begin{tikzpicture}
  \matrix[tableaumatrix] at (0,0) (qrteg)
  {
    1 \& 5 \& 2 \\
    \&   \& 3 \\
    \&   \& 6 \& 2 \& 4 \& 5 \& 4 \\
    \&   \&   \&   \&   \&   \& 5 \& 7 \\
  };
  \draw[gray,thick,rounded corners=3mm,->]
  ($ (qrteg-1-1) + (-1.5mm,-20mm) $) -- ($ (qrteg-1-1) + (-1.5mm,7mm) $) --
  ($ (qrteg-1-2) + (-1.5mm,-7mm) $) -- ($ (qrteg-1-2) + (-1.5mm,7mm) $) --
  ($ (qrteg-3-3) + (-1.5mm,-7mm) $) -- ($ (qrteg-1-3) + (-1.5mm,7mm) $) --
  ($ (qrteg-3-4) + (-1.5mm,-7mm) $) -- ($ (qrteg-3-4) + (-1.5mm,7mm) $) --
  ($ (qrteg-3-5) + (-1.5mm,-7mm) $) -- ($ (qrteg-3-5) + (-1.5mm,7mm) $) --
  ($ (qrteg-3-6) + (-1.5mm,-7mm) $) -- ($ (qrteg-3-6) + (-1.5mm,7mm) $) --
  ($ (qrteg-4-7) + (-1.5mm,-7mm) $) -- ($ (qrteg-3-7) + (-1.5mm,7mm) $) --
  ($ (qrteg-4-8) + (-1.5mm,-7mm) $) -- ($ (qrteg-4-8) + (-1.5mm,7mm) $) --
  ($ (qrteg-4-8) + (-1.5mm,20mm) $);
\end{tikzpicture}
\]

Let $w \in \aA^*$, and let $w^{(1)}\cdots w^{(m)}$ be the factorization of $w$ into maximal decreasing factors. Let
$\qrtabloid{w}$ be the quasi-ribbon tabloid whose $h$-th column has height $|w^{(h)}|$ and is filled with the symbols of $w^{(h)}$, for
$h = 1,\ldots,m$. (So each maximal decreasing factor of $w$ corresponds to a column of $\qrtabloid{w}$.) Then $\colreading{\qrtabloid{w}} = w$.

If $w$ is the column reading of some quasi-ribbon tableau $T$, it is called a \defterm{quasi-ribbon word}. It is easy to
see that the word $w$ is a quasi-ribbon word if and only if $\qrtabloid{w}$ is a quasi-ribbon tableau. For example $433$
is not a quasi-ribbon word, since $\qrtabloid{433} = \tableau{3 \\4 \& 3\\}$ is the only quasi-ribbon tabloid whose
column reading is $433$.

\begin{proposition}[{\cite[Proposition~3.4]{novelli_hypoplactic}}]
\label{prop:uqrwiffstduqrw}
A word $u \in \aA^*$ is a quasi-ribbon word if and only if $\std{u}$ is a quasi-ribbon word.
\end{proposition}

The following algorithm gives a method for inserting a symbol into a quasi-ribbon tableau. It is due to Krob \& Thibon,
but is stated here in a slightly modified form:

\begin{algorithm}[{\cite[\S~7.2]{krob_noncommutative4}}]
\label{alg:hypoplacticinsertionone}
~\par\nobreak
\textit{Input:} A quasi-ribbon tableau $T$ and a symbol $a \in \aA$.

\textit{Output:} A quasi-ribbon tableau $T\leftarrow a$.

\textit{Method:} If there is no entry in $T$ that is less than or equal to $a$, output the quasi-ribbon tableau obtained
by creating a new entry $a$ and attaching (by its top-left-most entry) the quasi-ribbon tableau $T$ to the bottom of $a$.

If there is no entry in $T$ that is greater than $a$, output the word obtained by creating a new entry $a$ and attaching
(by its bottom-right-most entry) the quasi-ribbon tableau $T$ to the left of $a$.

Otherwise, let $x$ and $z$ be the adjacent entries of the quasi-ribbon tableau $T$ such that $x \leq a < z$.
(Equivalently, let $x$ be the right-most and bottom-most entry of $T$ that is less than or equal to $a$, and let $z$ be
the left-most and top-most entry that is greater than $a$. Note that $x$ and $z$ could be either horizontally or
vertically adjacent.) Take the part of $T$ from the top left down to and including $x$, put a new entry $a$ to
the right of $x$ and attach the remaining part of $T$ (from $z$ onwards to the bottom right) to the bottom of the new
entry $a$, as illustrated here:
\begin{align*}
\tikz[tableau]\matrix
{
 \null \& x \&   \& \\
 \& z \& \null \& \null \\
 \&   \&   \& \null \\
}; \leftarrow a &=
\tikz[tableau]\matrix
{
 \null \& x \& a  \& \\
 \& \& z \& \null \& \null \\
 \& \& \&   \& \null \\
}; &&
\begin{array}{l}
\text{[where $x$ and $z$ are} \\
\text{\phantom{[}vertically adjacent]}
\end{array}
\displaybreak[0]\\
\tikz[tableau]\matrix
{
 \null \& \null \&   \& \\
 \& x \& z \& \null \\
 \&   \&   \& \null \\
}; \leftarrow a &=
\tikz[tableau]\matrix
{
 \null \& \null \\
 \& x \& a \\
 \&  \& z \& \null \\
 \& \& \& \null \\
}; &&
\begin{array}{l}
\text{[where $x$ and $z$ are} \\
\text{\phantom{[}horizontally adjacent]}
\end{array}
\end{align*}
Output the resulting quasi-ribbon tableau.
\end{algorithm}

Using an iterative form of this algorithm, one can start from a word $a_1\cdots a_k$ (where $a_i \in \aA$) and compute a
quasi-ribbon tableau $\qrtableau{a_1\cdots a_k}$. Essentially, one simply starts with the empty quasi-ribbon tableau and
inserts the symbols $a_1$, $a_2$, \ldots, $a_k$ in order. However, the algorithm described below also computes a
recording ribbon $\recribbon{a_1\cdots a_k}$, which will be used later in discussing an analogue of the
Robinson--Schensted--Knuth correspondence.

\begin{algorithm}[{\cite[\S~7.2]{krob_noncommutative4}}]
\label{alg:hypoplacticinsert}
~\par\nobreak
\textit{Input:} A word $a_1\cdots a_k$, where $a_i \in \aA$.

\textit{Output:} A quasi-ribbon tableau $\qrtableau{a_1\cdots a_k}$ and a recording ribbon $\recribbon{a_1\cdots a_k}$.

\textit{Method:} Start with the empty quasi-ribbon tableau $Q_0$ and an empty recording ribbon $R_0$. For each $i = 1$, \ldots,
$k$, insert the symbol $a_i$ into $Q_{i-1}$ as per \fullref{Algorithm}{alg:hypoplacticinsert}; let $Q_i$ be the
resulting quasi-ribbon tableau. Build the recording ribbon $R_i$, which has the same shape as $Q_i$, by adding an entry
$i$ into $R_{i-1}$ at the same place as $a_i$ was inserted into $Q_{i-1}$.

Output $Q_k$ for $\qrtableau{a_1\cdots a_k}$ and $R_k$ as $\recribbon{a_1\cdots a_k}$.
\end{algorithm}

For example, the sequence of pairs $(Q_i,R_i)$ produced during the application of
\fullref{Algorithm}{alg:hypoplacticinsert} to the word $4323$ is
\[
(\;,\;),\;
\bigl(\tableau{4\\},\tableau{1\\}\bigr),\;
\left(\tableau{3\\4\\},\tableau{2\\1\\}\right),\;
\left(\tableau{2\\3\\4\\},\tableau{3\\2\\1\\}\right),\;
\left(\tableau{2\\3\&3\\\&4\\},\tableau{3\\2\&4\\\&1\\}\right).
\]
Therefore $\qrtableau{4323} = \tableau{2\\3\&3\\\&4\\}$ and $\recribbon{4323} = \tableau{3\\2\&4\\\&1\\}$. It is
straightforward to see that the map $u \mapsto (\qrtableau{u},\recribbon{u)}$ is a bijection between words in $\aA^*$ and pairs
consisting of a quasi-ribbon tableau and a recording ribbon \cite[\S~7.2]{krob_noncommutative4} of the same shape; this
is an analogue of the Robinson--Schensted--Knuth correspondence. For instace, if $\qrtableau{u}$ is \eqref{eq:qrteg} and $\recribbon{u}$ is
\eqref{eq:rreg}, then $u = 12446553275$.

Recall the definition of $\descomp{\textvisiblespace}$ from \fullref{Subsection}{subsec:comppart}.

\begin{proposition}[{\cite[Theorems~4.12 \&~4.16]{novelli_hypoplactic}}]
\label{prop:shapeqrtuisdescomp}
For any word $u \in \aA^*$, the shape of $\qrtableau{u}$ (and of $\recribbon{u}$) is $\descomp{\std{u}^{-1}}$.
\end{proposition}

The possible definition (approach~H2 in the introduction) of the relation $\hypocong$ using tableaux and insertion is
the following:
\[
u \hypocong v \iff \qrtableau{u} = \qrtableau{v}.
\]
Using this as a definition, it follows that $\hypocong$ is a congruence on $\aA^*$ \cite{novelli_hypoplactic}, which is
known as the \defterm{hypoplactic congruence} on $\aA^*$. The factor monoid $\aA^*/{\hypocong}$ is the
\defterm{hypoplactic monoid} and is denoted $\hypo$. The congruence $\hypocong$ naturally restricts to a congruence on
$\aA_n^*$, and the factor monoid $\aA_n^*/{\hypocong}$ is the \defterm{hypoplactic monoid of rank $n$} and is denoted
$\hypo_n$.

As noted above, if $w$ is a quasi-ribbon word, then $w = \colreading{\qrtableau{w}}$. Thus the
quasi-ribbon words in $\aA^*$ form a cross-section (or set of normal forms) for $\hypo$, and the quasi-ribbon words in
$\aA_n^*$ form a cross-section for $\hypo_n$.

This paper also uses the following equivalent characterization of $\hypocong$ \cite[Theorem~4.18]{novelli_hypoplactic}:
\begin{equation}
\label{eq:charhypocong}
\begin{aligned}
&u \hypocong v \iff\\
&\qquad\qquad \descomp{\std{u}^{-1}} = \descomp{\std{v}^{-1}} \land \wt{u} = \wt{v}.
\end{aligned}
\end{equation}

\section{Quasi-Kashiwara operators and the quasi-crystal graph}
\label{sec:quasikashiwara}

This section defines the quasi-Kashiwara operators and the quasi-crystal graph, and shows that isomorphisms between
components of this graph give rise to a monoid. The following section will prove that this monoid is in fact the
hypoplactic monoid.

Let $i \in \set{1,\ldots,n-1}$. Let $u \in \aA_n^*$. The word $u$ has an \defterm{$i$-inversion} if it contains a symbol
$i+1$ to the left of a symbol $i$. Equivalently, $u$ has an $i$-inversion if it contains a subsequence $(i+1)i$. (Recall
from \fullref{Subsection}{subsec:words} that a subsequence may be made up of non-consecutive letters.) If the word $u$
does not have an $i$-inversion, it is said to be \defterm{$i$-inversion-free}.

Define the \defterm{quasi-Kashiwara operators} $\e_i$ and $\f_i$ on $\aA_n^*$ as follows: Let $u \in \aA_n^*$.
\begin{itemize}
\item If $u$ has an $i$-inverstion, both $\e_i(u)$ and $\f_i(u)$ are undefined.
\item If $u$ is $i$-inversion-free, but $u$ contains at least one symbol $i+1$, then $\e_i(u)$ is the word obtained from $u$
  by replacing the left-most symbol $i+1$ by $i$; if $u$ contains no symbol $i+1$, then $\e_i(u)$ is undefined.
\item If $u$ is $i$-inversion-free, but $u$ contains at least one symbol $i$, then $\f_i(u)$ is the word obtained from $u$
  by replacing the right-most symbol $i$ by $i+1$; if $u$ contains no symbol $i$, then $\f_i(u)$ is undefined.
\end{itemize}
For example,
\begin{align*}
\e_2(3123) &\;\text{is undefined since $3123$ has a $2$-inversion;} \\
\f_2(3131) &\;\text{is undefined since $3131$ is $2$-inversion free} \\
&\;\qquad\text{but does not contains a symbol $2$;} \\
\f_1(3113) &= 3123.
\end{align*}
Define
\begin{align*}
  \ecount_i(u) & = \max\gset[\big]{k \in \nset\cup\set{0}}{\text{$\underbrace{\e_i\cdots\e_i}_{\text{$k$ times}}(u)$ is defined}}; \\
  \fcount_i(u) & = \max\gset[\big]{k \in \nset\cup\set{0}}{\text{$\underbrace{\f_i\cdots\f_i}_{\text{$k$ times}}(u)$ is defined}}.
\end{align*}
An immediate consequence of the definitions of $\e_i$ and $\f_i$ is that:
\begin{itemize}
\item If $u$ has an $i$-inverstion, then $\ecount_i(u) = \fcount_i(u) = 0$;
\item If $u$ is $i$-inversion-free, then every symbol $i$ lies to the left of every symbol $i+1$ in $u$, and so
  $\ecount_i(u) = \abs{u}_{i+1}$ and $\fcount_i(u) = \abs{u}_i$.
\end{itemize}

\begin{remark}
  It is worth noting how the quasi-Kashiwara operators $\e_i$ and $\f_i$ relate to the standard Kashiwara operators
  $\ke_i$ and $\kf_i$ as defined in \fullref{Subsection}{subsec:kashiwara}. As discussed in
  \fullref{Subsection}{subsec:computingkashiwara}, one computes the action of $\ke_i$ and $\kf_i$ on a word
  $u \in \aA_n^*$ by replacing each symbol $i$ with $+$, each symbol $i+1$ with $-$, every other symbol with the empty
  word, and then iteratively deleting factors ${-}{+}$ until a word of the form ${+}^p{-}^q$ remains, whose left-most
  symbol ${-}$ and right-most symbol ${+}$ (if they exist) indicate the symbols in $u$ changed by $\ke_i$ and $\kf_i$
  respectively. The deletion of factors ${-}{+}$ corresponds to rewriting to normal form a word representing an element
  of the bicyclic monoid $\pres{{+},{-}}{({-}{+},\emptyword)}$. To compute the action of the quasi-Kashiwara operators
  $\e_i$ and $\f_i$ defined above on a word $u \in \aA_n^*$, replace the symbols in the same way as before, but now
  rewrite to normal form as a word representing an element of the monoid $\pres{{+},{-}}{({-}{+},0)}$, where $0$ is a
  multiplicative zero. Any word that contains a symbol ${-}$ to the left of a symbol ${+}$ will be rewritten to $0$ and
  so $\e_i$ and $\f_i$ will be undefined in this case. If the word does not contain a symbol ${-}$ to the left of a
  symbol ${+}$, then it is of the form ${+}^p{-}^q$, and the left-most symbol $-$ and the right-most symbol $+$
  indictate the symbols in $u$ changed by $\e_i$ and $\f_i$. In essence, one obtains the required analogies of the
  Kashiwara operators by replacing the bicyclic monoid $\pres{{+},{-}}{({-}{+},\emptyword)}$, where ${-}{+}$ rewrites to
  the identity, with the monoid $\pres{{+},{-}}{({-}{+},0)}$, where ${-}{+}$ rewrites to the zero.
\end{remark}

The action of the quasi-Kashiwara operators is essentially a restriction of the action of the Kashiwara operators:

\begin{proposition}
  \label{prop:keddefinedimpliesefdefined}
  Let $u \in \aA_n^*$. If $\e_i(u)$ is defined, so is $\ke_i(u)$, and $\ke_i(u) = \e_i(u)$. If $\f_i(u)$ is defined, so
  is $\kf_i(u)$, and $\kf_i(u) = \f_i(u)$.
\end{proposition}

\begin{proof}
  Let $u \in \aA_n^*$. Suppose the quasi-Kashiwara operator $\e_i$ is defined on $u$. Then $u$ contains at least one
  symbol $i+1$ but is $i$-inversion-free, so that every symbol $i$ lies to the left of every symbol
  $i+1$ in $u$. So when one computes the action of $\ke_i$, replacing every symbol $i$ with the symbol ${+}$ and every
  symbol $i+1$ with the symbol ${-}$ leads immediately to the word ${+}^{\kfcount_i(u)}{-}^{\kecount_i(u)}$ (that is,
  there are no factors ${-}{+}$ to delete). Hence $\kecount_i(u) > 0$, and so the Kashiwara operator $\ke_i$ is defined
  on $u$. Furthermore, $\ke_i$ acts by changing the symbol $i+1$ that contributed the leftmost ${-}$ to $i$, and, since
  there was no deletion of factors ${-}{+}$, this symbol must be the leftmost symbol $i+1$ in $u$. Thus
  $\ke_i(u) = \e_i(u)$.

  Similarly, if the quasi-Kashiwara operator $\f_i$ is defined on $u$, so is the Kashiwara operator $\kf_i$, and
  $\kf_i(u) = \f_i(u)$.
\end{proof}

The original definition of the Kashiwara operators $\ke_i$ and $\kf_i$ in \fullref{Subsection}{subsec:kashiwara} was
recursive: whether the action on $uv$ recurses to the action on $u$ or on $v$ depends on the maximum number of times
each operator can be applied to $u$ and $v$ separately. It seems difficult to give a similar recursive definition for
the quasi-Kashiwara operators $\e_i$ and $\f_i$ defined here: if, for example, both operators can be applied zero times
to $u$, this may mean that $u$ does not contain symbols $i$ or $i+1$, in which case the operators may still be defined
on $uv$, or it may mean that $u$ contains a symbol $i+1$ to the left of a symbol $i$, in which case the operators are
certainly not defined on $uv$.

This concludes the discussion contrasting the standard Kashiwara operators with the quasi-Kashiwara operators defined
here. The aim now is to use the operators $\e_i$ and $\f_i$ to build the quasi-crystal graph and to establish some of
its properties.

\begin{lemma}
  \label{lem:efinverse}
  For all $i \in \set{1,\ldots,n-1}$, the operators $\e_i$ and $\f_i$ are mutually inverse, in the sense that if
  $\e_i(u)$ is defined, $u = \f_i(\e_i(u))$, and if $\f_i(u)$ is defined, $u = \e_i(\f_i(u))$.
\end{lemma}

\begin{proof}
  Let $u \in \aA_n^*$. Suppose that $\e_i(u)$ is defined. Then $u$ contains at least one symbol $i+1$ but is
  $i$-inversion-free, so that every symbol $i+1$ is to the right of every symbol $i$. Since $\e_i(u)$ is obtained from
  $u$ by replacing the \emph{left-most} symbol $i+1$ by $i$ (which becomes the right-most symbol $i$), every symbol
  $i+1$ is to the right of every symbol $i$ in the word $\e_i(u)$ and so $\e_i(u)$ is $i$-inversion-free, and $\e_i(u)$
  contains at least one symbol $i$. Thus $\f_i(\e_i(u))$ is defined, and is obtained from $\e_i(u)$ by replacing the
  right-most symbol $i$ by $i+1$, which produces $u$. Hence if $\e_i(u)$ is defined, $u = \f_i(\e_i(u))$. Similar
  reasoning shows that if $\f_i(u)$ is defined, $u = \e_i(\f_i(u))$.
\end{proof}

The operators $\e_i$ and $\f_i$ respectively increase and decrease weight whenever they are defined, in
the sense that if $\e_i(u)$ is defined, then $\wt{\e_i(u)} > \wt{u}$, and if $\f_i(u)$ is defined, then
$\wt{\f_i(u)} < \wt{u}$. This is because $\e_i$ replaces a symbol $i+1$ with $i$ whenever it is defined, which
corresponds to decrementing the $i+1$-th component and incrementing the $i$-th component of the weight, which results in
an increase with respect to the order \eqref{eq:weightorder}. Similarly, $\f_i$ replaces a symbol $i$ with $i+1$
whenever it is defined. For this reason, the $\e_i$ and $\f_i$ are respectively called the quasi-Kashiwara
\defterm{raising} and \defterm{lowering} operators.

The \defterm{quasi-crystal graph} $\Gamma_n$ is the labelled directed graph with vertex set $\aA_n^*$ and, for all
$u \in \aA_n^*$ and $i \in \set{1,\ldots,n-1}$, an edge from $u$ to $u'$ labelled by $i$ if and only if $\f_i(u) = u'$
(or, equivalently by \fullref{Lemma}{lem:efinverse}, $u = \e_i(u')$). Part of $\Gamma_4$ is shown in
\fullref{Figure}{fig:gamma4}. (The notation $\Gamma_n$ will be discarded in favour of $\Gamma(\hypo_n)$ after it has
been shown that the relationship between $\Gamma_n$ and $\hypo_n$ is analogous to that between $\Gamma(\plac_n)$ and
$\plac_n$.)

For any $u \in \aA_n^*$, let $\Gamma_n(u)$ be the connected component of $\Gamma_n$ containing the vertex $u$. Notice
that every vertex of $\Gamma_n$ has at most one incoming and at most one outgoing edge with a given label. A
\defterm{quasi-crystal isomorphism} between two connected components is a weight-preserving labelled digraph
isomorphism. (This parallels the definition of a crystal isomorphism in \fullref{Subsection}{subsec:kashiwara}.)

\begin{figure}[t]
  \centering
  \includegraphics{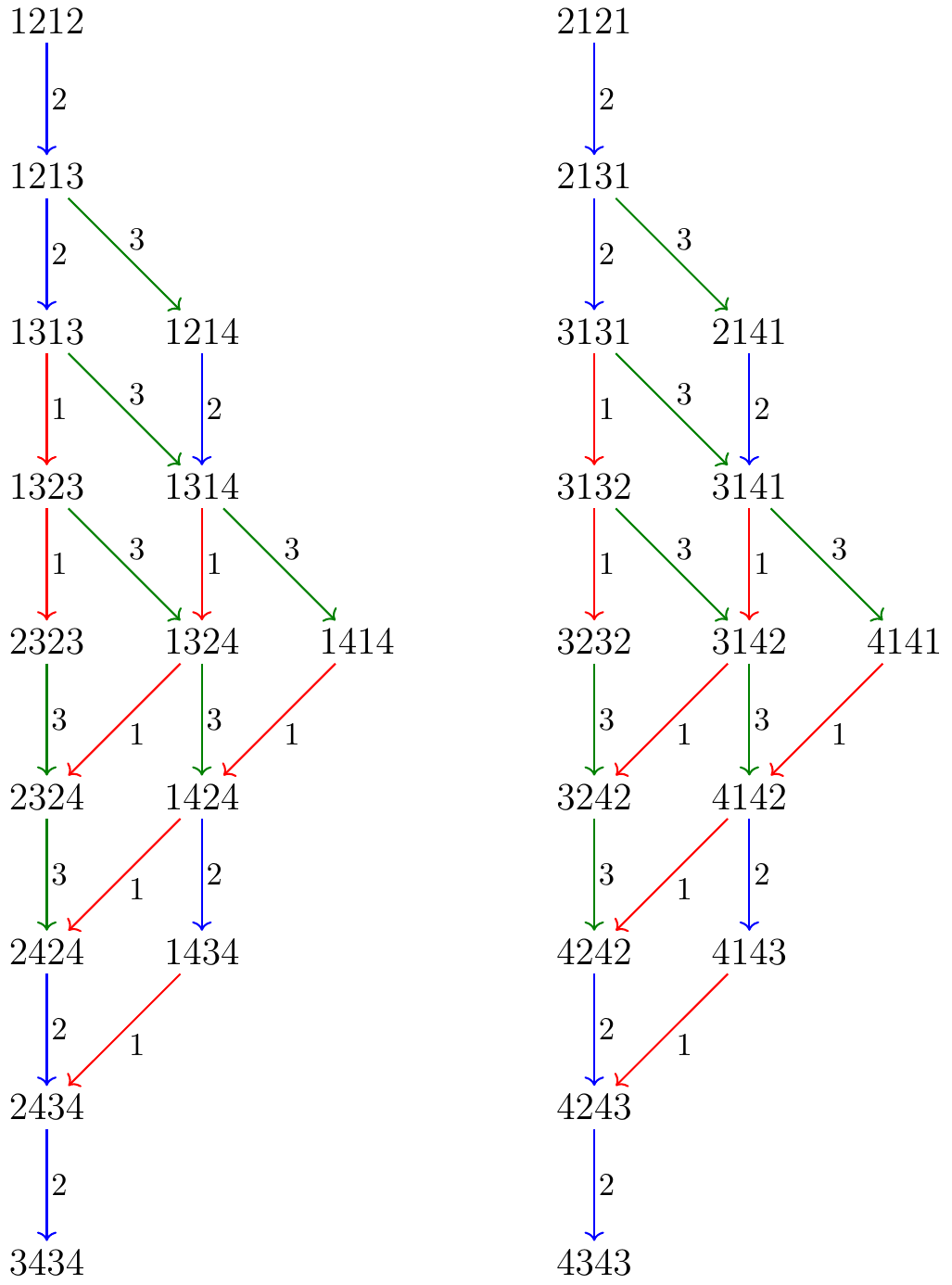}
  \caption{The isomorphic components $\Gamma_4(1212)$ and $\Gamma_4(2121)$ of the quasi-crystal graph $\Gamma_4$.}
  \label{fig:gamma4}
\end{figure}

Define a relation $\sim$ on the free monoid $\aA_n^*$ as follows: $u \sim v$ if and only if there is a quasi-crystal isomorphism
$\theta : \Gamma_n(u) \to \Gamma_n(v)$ such that $\theta(u) = v$. That is, $u \sim v$ if and only if $u$ and $v$ are in
the same position in isomorphic connected components of $\Gamma_n$. For example, $1324 \sim 3142$, since these words are
in the same position in their connected components, as can be seen in \fullref{Figure}{fig:gamma4}.

The rest of this section is dedicated to proving that the relation $\sim$ is a congruence on $\aA_n^*$. The proofs of this result and
the necessary lemmata parallel the purely combinatorial proofs by the present authors and Gray
\cite[\S~2.4]{cgm_crystal} that isomorphisms of crystal graphs give rise to congruences.

\begin{lemma}
  \label{lem:efdefinedpreservedbyiso}
  Let $u,v,u',v' \in \aA_n^*$.  Suppose $u \sim v$ and $u' \sim v'$, and let $\theta : \Gamma_n(u) \to \Gamma_n(v)$ and
  $\theta' : \Gamma_n(u') \to \Gamma_n(v')$ be quasi-crystal isomorphisms such that $\theta(u) = v$ and
  $\theta'(u') = v'$. Let $i \in \set{1,\ldots,n-1}$. Then:
  \begin{enumerate}
  \item $\e_i(uu')$ is defined if and only if $\e_i(vv')$ is defined. If both are defined, exactly one of the following
    statements holds:
    \begin{enumerate}
    \item $\e_i(uu') = u\e_i(u')$ and $\e_i(vv') = v\e_i(v')$;
    \item $\e_i(uu') = \e_i(u)u'$ and $\e_i(vv') = \e_i(v)v'$.
    \end{enumerate}
  \item $\f_i(uu')$ is defined if and only if $\f_i(vv')$ is defined. If both are defined, exactly one of the
    following statements holds:
    \begin{enumerate}
    \item $\f_i(uu') = u\f_i(u')$ and $\f_i(vv') = v\f_i(v')$;
    \item $\f_i(uu') = \f_i(u)u'$ and $\f_i(vv') = \f_i(v)v'$.
    \end{enumerate}
  \end{enumerate}
\end{lemma}

\begin{proof}
  Suppose $\e_i(uu')$ is defined. Then $uu'$ is $i$-inversion-free, and contains at least one symbol
  $i+1$. Hence both $u$ and $u'$ are $i$-inversion-free, and at least one of $u$ and $u'$ contains a symbol
  $i+1$. Since $\theta$ and $\theta'$ are crystal isomorphisms, $\wt{u} = \wt{v}$ and $\wt{u'} = \wt{v'}$ (and thus $u$
  and $v$ contain the same number of each symbol, and $u'$ and $v'$ contain the same number of each symbol). Consider
  separately two cases depending on whether $u$ contains a symbol $i+1$:
  \begin{itemize}

  \item Suppose that $u$ does not contain a symbol $i+1$ (and hence $v$ does not contain a symbol $i+1$). Then $u'$ must
    contain a symbol $i+1$ and indeed the left-most symbol $i+1$ of $uu'$ must lie in $u'$. So $\e_i(u')$ is defined and
    $\e_i(uu') = u\e_i(u')$. Thus there is an edge labelled by $i$ ending at $u'$. Since $\theta'$ is a quasi-crystal
    isomorphism, there is an edge labelled by $i$ ending at $v'$. Hence $\e_i(v')$ is defined, and so $v'$ is
    $i$-inversion-free. Since $v$ does not contain any symbol $i+1$, it follows that $vv'$ is $i$-inversion-free. Hence
    $\e_i(vv')$ is defined and, since the left-most symbol $i+1$ of $vv'$ lies in $v'$, it also holds that
    $\e_i(vv') = v\e_i(v')$.

  \item Suppose that $u$ contains a symbol $i+1$ (and hence $v$ contains a symbol $i+1$). Then $u'$ cannot contain a
    symbol $i$, and the left-most symbol $i+1$ of $uu'$ must lie in $u$. Then $\e_i(u)$ is defined and
    $\e_i(uu') = \e_i(u)u'$. So there is an edge labelled by $i$ ending at $u$. Since $\theta$ is a quasi-crystal
    isomorphism, there is an edge labelled by $i$ ending at $v$. Hence $\e_i(v)$ is defined, and so $v$ is
    $i$-inversion-free. Since $v'$ does not contain any symbol $i$, it follows that $vv'$ is $i$-inversion-free. Hence
    $\e_i(vv')$ is defined and, since the left-most symbol $i+1$ in $vv'$ lies in $v$, it also holds that
    $\e_i(vv') = \e_i(v)v'$.

  \end{itemize}
  This proves the forward implications of the three statements relating to $\e_i$ in part~(1). Interchanging $u$ and $u'$
  with $v$ and $v'$ proves the reverse implications. The statements for $\f_i$ in part~(2) follow similarly.
\end{proof}

\begin{lemma}
  \label{lem:gdefinedpreservedbyiso}
  Let $u,v,u',v' \in \aA_n^*$. Suppose $u \sim v$ and $u' \sim v'$, and let $\theta : \Gamma_n(u) \to \Gamma_n(v)$ and
  $\theta' : \Gamma_n(u') \to \Gamma_n(v')$ be quasi-crystal isomorphisms such that $\theta(u) = v$ and
  $\theta'(u') = v'$. Let $\g_{i_1},\ldots,\g_{i_r} \in \gset[\big]{\e_i,\f_i}{i \in \set{1,\ldots,n-1}}$. Then:
  \begin{enumerate}
  \item $\g_{i_1}\cdots\g_{i_r}(uu')$ is defined if and only if $\g_{i_1}\cdots\g_{i_r}(vv')$ is defined.
  \item When both $\g_{i_1}\cdots\g_{i_r}(uu')$ and $\g_{i_1}\cdots\g_{i_r}(vv')$ are defined, the sequence
    $\g_{i_1},\ldots,\g_{i_r}$ partitions into two subsequences $\g_{j_1},\ldots,\g_{j_s}$ and
    $\g_{k_1},\ldots,\g_{k_t}$ such that
    \begin{align*}
      \g_{i_1}\cdots\g_{i_r}(uu') &= \g_{j_1}\cdots\g_{j_s}(u)\g_{k_1}\cdots\g_{k_t}(u'), \\
      \g_{i_1}\cdots\g_{i_r}(vv') &= \g_{j_1}\cdots\g_{j_s}(v)\g_{k_1}\cdots\g_{k_t}(v'); \\
      \intertext{where}
      \theta\bigl(\g_{j_1}\cdots\g_{j_s}(u)\bigr) &= \g_{j_1}\cdots\g_{j_s}(v), \\
      \theta'\bigl(\g_{k_1}\cdots\g_{k_t}(u')\bigr) &= \g_{k_1}\cdots\g_{k_t}(v').
    \end{align*}
  \end{enumerate}
\end{lemma}

\begin{proof}
  This result follows by iterated application of \fullref{Lemma}{lem:efdefinedpreservedbyiso}, with the last two
  equalities holding because $\theta$ and $\theta'$ are quasi-crystal isomorphisms with $\theta(u) = v$ and
  $\theta'(u') = v'$.
\end{proof}

\begin{proposition}
  \label{prop:isomdefinescong}
  The relation $\sim$ is a congruence on the free monoid $\aA_n^*$.
\end{proposition}

\begin{proof}
  It is clear from the definition that $\sim$ is an equivalence relation; it thus remains to prove that $\sim$ is
  compatible with multiplication in $\aA_n^*$.

  Suppose $u \sim v$ and $u' \sim v'$. Then there exist quasi-crystal isomorphisms $\theta : \Gamma_n(u) \to \Gamma_n(v)$ and
  $\theta : \Gamma_n(u') \to \Gamma_n(v')$ such that $\theta(u) = v$ and $\theta'(u') = v'$.

  Define a map $\Theta : \Gamma_n(uu') \to \Gamma_n(vv')$ as follows. For $w \in \Gamma_n(uu')$, choose
  $\g_{i_1},\ldots,\g_{i_r} \in \gset[\big]{\e_i,\f_i}{i \in \set{1,\ldots,n-1}}$ such that $\g_{i_1}\cdots\g_{i_r}(uu') = w$; such a sequence
  exists because $w$ lies in the connected component $\Gamma_n(uu')$. Define $\Theta(w)$ to be
  $\g_{i_1}\cdots\g_{i_r}(vv')$; note that this is defined by \fullref[(1)]{Lemma}{lem:gdefinedpreservedbyiso}.

  It is necessary to prove that $\Theta$ is well-defined. Suppose that
  $\g_{\hat{\imath}_1},\ldots,\g_{\hat{\imath}_m} \in \gset[\big]{\e_i,\f_i}{i \in \set{1,\ldots,n-1}}$ is such that
  $\g_{\hat{\imath}_1}\cdots\g_{\hat{\imath}_m}(uu') = w$, and let
  $z = \g_{\hat{\imath}_1}\cdots\g_{\hat{\imath}_m}(vv')$. Note that by
  \fullref[(2)]{Lemma}{lem:gdefinedpreservedbyiso},
  \begin{itemize}
  \item the
    sequence $\g_{i_1},\ldots,\g_{i_r}$ partitions into two subsequences $\g_{j_1},\ldots,\g_{j_s}$ and
    $\g_{k_1},\ldots,\g_{k_t}$ such that
    \begin{equation}
      \label{eq:isomdefinescong1}
      \g_{i_1}\cdots\g_{i_r}(uu') = \g_{j_1}\cdots\g_{j_s}(u)\g_{k_1}\cdots\g_{k_t}(u');
    \end{equation}
  \item the sequence
    $\g_{\hat{\imath}_1},\ldots,\g_{\hat{\imath}_m}$ partitions into two subsequences $\g_{\hat{\jmath}_1},\ldots,\g_{\hat{\jmath}_p}$ and
    $\g_{\hat{k}_1},\ldots,\g_{\hat{k}_q}$ such that
    \begin{equation}
      \label{eq:isomdefinescong2}
      \g_{\hat{\imath}_1}\cdots\g_{\hat{\imath}_m}(uu') = \g_{\hat{\jmath}_1}\cdots\g_{\hat{\jmath}_p}(u)\g_{\hat{k}_1}\cdots\g_{\hat{k}_q}(u').
    \end{equation}
  \end{itemize}
  Since both $\g_{i_1}\cdots\g_{i_r}(uu')$ and $\g_{\hat{\imath}_1}\cdots\g_{\hat{\imath}_m}(uu')$ equal $w$, and since $\g_{j_1}\cdots\g_{j_s}(u)$ and $\g_{\hat{\jmath}_1}\cdots\g_{\hat{\jmath}_p}(u)$ have length $|u|$, it follows that
  \begin{equation}
    \label{eq:isomdefinescong3}
    \begin{aligned}
      \g_{j_1}\cdots\g_{j_s}(u) &= \g_{\hat{\jmath}_1}\cdots\g_{\hat{\jmath}_p}(u); \\
      \g_{k_1}\cdots\g_{k_t}(u') &= \g_{\hat{k}_1}\cdots\g_{\hat{k}_q}(u').
    \end{aligned}
  \end{equation}
  Then
  \begin{align*}
    \Theta(w) &= \g_{i_1}\cdots\g_{i_r}(vv') && \text{[by definition]} \\
              &= \g_{j_1}\cdots\g_{j_s}(v)\g_{k_1}\cdots\g_{k_t}(v') && \text{[by \eqref{eq:isomdefinescong1}]} \displaybreak[0]\\
              &= \theta\bigl(\g_{j_1}\cdots\g_{j_s}(u)\bigr)\theta'\bigl(\g_{k_1}\cdots\g_{k_t}(u')\bigr) \displaybreak[0]\\
              &= \theta\bigl(\g_{\hat{\jmath}_1}\cdots\g_{\hat{\jmath}_p}(u)\bigr)\theta'\bigl(\g_{\hat{k}_1}\cdots\g_{\hat{k}_q}(u')\bigr) && \text{[by
                                                                                                                                  \eqref{eq:isomdefinescong3}]}\displaybreak[0]\\
              &= \g_{\hat{\jmath}_1}\cdots\g_{\hat{\jmath}_p}(v)\g_{\hat{k}_1}\cdots\g_{\hat{k}_q}(v') \displaybreak[0]\\
              &= \g_{\hat{\imath}_1}\cdots\g_{\hat{\imath}_m}(vv') && \text{[by \eqref{eq:isomdefinescong2}]} \\
              &= z.
  \end{align*}
  So $\Theta$ is well-defined.

  By \fullref[(1)]{Lemma}{lem:gdefinedpreservedbyiso}, $\Theta$ and its inverse preserve labelled edges. To see that
  $\Theta$ preserves weight, proceed as follows. Since $\theta$ and $\theta'$ are quasi-crystal isomorphisms,
  $\wt{u} = \wt{v}$ and $\wt{u'} = \wt{v'}$. Hence $\wt{uu'} = \wt{vv'}$. Therefore if $\g_{i_1}\cdots\g_{i_r}(uu')$ and
  $\g_{i_1}\cdots\g_{i_r}(vv')$ are both defined, then both sequences of operators partition as above, and so
  \begin{align*}
    \wt{w} ={}& \wt[\big]{\g_{i_1}\cdots\g_{i_r}(uu')} \displaybreak[0]\\
    ={}& \wt[\big]{\g_{j_1}\cdots\g_{j_s}(u)\g_{k_1}\cdots\g_{k_t}(u')} \displaybreak[0]\\
    ={}& \wt[\big]{\g_{j_1}\cdots\g_{j_s}(u)}+\wt[\big]{\g_{k_1}\cdots\g_{k_t}(u')} \displaybreak[0]\\
    ={}& \wt[\big]{\theta(\g_{j_1}\cdots\g_{j_s}(u))}+\wt[\big]{\theta'(\g_{k_1}\cdots\g_{k_t}(u'))} \\
    &\qquad\text{[since $\theta$ and $\theta'$ preserve weights]} \displaybreak[0]\\
    ={}& \wt[\big]{\g_{j_1}\cdots\g_{j_s}(\theta(u))}+\wt[\big]{\g_{k_1}\cdots\g_{k_t}(\theta'(u'))} \\
    &\qquad\text{[since $\theta$ and $\theta'$ are quasi-crystal isomorphisms]} \displaybreak[0]\\
    ={}& \wt[\big]{\g_{j_1}\cdots\g_{j_s}(v)}+\wt[\big]{\g_{k_1}\cdots\g_{k_t}(v')} \displaybreak[0]\\
    ={}& \wt[\big]{\g_{j_1}\cdots\g_{j_s}(v)\g_{k_1}\cdots\g_{k_t}(v')} \displaybreak[0]\\
    ={}& \wt[\big]{\g_{i_1}\cdots\g_{i_r}(vv')} \\
    ={}& \wt{\Theta(w)}.
  \end{align*}

  Thus $\Theta$ is a quasi-crystal isomorphism. Hence $uu' \sim vv'$. Therefore the relation $\sim$ is a congruence.
\end{proof}

\section{The quasi-crystal graph and the hypoplactic monoid}
\label{sec:quasicrystalandhypo}

Since $\sim$ is a congruence, it makes sense to define the factor monoid $H_n = \aA_n^*/{\sim}$. The aim is now to show
that $H_n$ is the hypoplactic monoid $\hypo_n$ by showing that $\sim$ is equal to the relation $\hypocong$ on $\aA_n^*$
as defined in \fullref{Section}{sec:quasiribbontableau} using quasi-ribbon tableaux and
\fullref{Algorithm}{alg:hypoplacticinsert}. Some of the lemmata in this section are more complicated than necessary for
this aim, because they will be used in future sections to prove other results.

\begin{proposition}
  \label{prop:eipreservesstd}
  Let $u \in \aA_n^*$ and $i \in \set{1,\ldots,n-1}$.
  \begin{enumerate}
  \item If $\e_i(u)$ is defined, then $\std{\e_i(u)} = \std{u}$.
  \item If $\f_i(u)$ is defined, then $\std{\f_i(u)} = \std{u}$.
  \end{enumerate}
\end{proposition}

\begin{proof}
  Suppose $\e_i(u)$ is defined, and that $u$ contains $\sigma$ symbols $i$ and $\tau$ symbols $i+1$. Then during the computation
  of $\std{u}$, these symbols $i$ and $i+1$ become $i_1,\ldots,i_\sigma$ and $(i+1)_1,\ldots,(i+1)_\tau$ when subscripts are
  attached.  Since $\e_i(u)$ is defined, $u$ is $i$-inversion-free and so every symbol $i$ in $u$ is to the left of every
  symbol $i+1$, and $\e_i(u)$ is obtained from $u$ by replacing the leftmost symbol $i+1$ by $i$. Thus, during the
  computation of $\std{\e_i(u)}$, the symbols $i$ and $i+1$ become $i_1,\ldots,i_{\sigma+1}$ and
  $(i+1)_1,\ldots,(i+1)_{\tau-1}$. Thus, symbols $i_1,\ldots,i_\sigma,(i+1)_1,\ldots,(i+1)_{\tau}$ and
  $i_1,\ldots,i_{\sigma+1},(i+1)_1,\ldots,(i+1)_{\tau-1}$ are replaced by the same symbols of the same rank in $\aA$. Hence
  $\std{\e_i(u)} = \std{u}$. This proves part~(1); similar reasoning proves part~(2).
\end{proof}

\begin{corollary}
  \label{corol:qrwsameshapecomponent}
  If $u \in \aA_n^*$ is a quasi-ribbon word, then every word in $\Gamma_n(u)$ is quasi-ribbon word, and all of the
  corresponding quasi-ribbon tableaux have the same shape as $\qrtableau{u}$.
\end{corollary}

\begin{proof}
  Let $w \in \Gamma_n(u)$. By \fullref{Proposition}{prop:eipreservesstd}, $\std{w} = \std{u}$. Thus by
  \fullref{Proposition}{prop:uqrwiffstduqrw}, $w$ is a quasi-ribbon word, and by
  \fullref{Proposition}{prop:shapeqrtuisdescomp}, $\qrtableau{w}$ has the same shape as $\qrtableau{u}$.
\end{proof}

\begin{lemma}
  \label{lem:efdefinedfortableau}
  Let $i \in \set{1,\ldots,n-1}$. Let $u \in \aA_n^*$ be a quasi-ribbon word. Then:
  \begin{enumerate}
  \item $\e_i(u)$ is defined if and only if $\qrtableau{u}$ contains some symbol $i+1$ but does not contain $i+1$ below $i$
    in the same column;
  \item $\f_i(u)$ is defined if and only if $\qrtableau{u}$ contains some symbol $i$ but does not contain $i+1$ below $i$
    in the same column.
  \end{enumerate}
\end{lemma}

\begin{proof}
  By the definition of the reading of a quasi-ribbon tableau, and the fact that the rows of quasi-ribbon tableaux are
  non-decreasing from left to right, the word $u$ has an $i$-inversion if and only if $i+1$ and $i$ appear in the same
  column of $\qrtableau{u}$. Thus part~(1) follows immediately from the definition of $\e_i$. Part~(2) follows by
  similar reasoning for $\f_i$.
\end{proof}

\begin{proposition}
  \label{prop:charhighestweighttableau}
  Let $\alpha$ be a composition.
  \begin{enumerate}
  \item The set of quasi-ribbon words corresponding to quasi-ribbon tableaux of shape $\alpha$ forms a single connected
    component of $\Gamma_n$.
  \item In this connected component, there is a unique highest-weight word, which corresponds to the quasi-ribbon
    tableau of shape $\alpha$ whose $j$-th row consists entirely of symbols $j$, for $j=1,\ldots,\clen\alpha$.
  \end{enumerate}
\end{proposition}

\begin{proof}
  Let $w$ be the quasi-ribbon word such that $\qrtableau{w}$ has shape $\alpha$ and has $j$-th row full of symbols $j$,
  for each $j=1,\ldots,\clen\alpha$. Clearly $\e_i(w)$ cannot be defined for any $i \geq \clen\alpha$, since there are
  no symbols $i+1$ in the word $w$. Furthermore, for $i = 1,\ldots,\clen\alpha-1$, the right-most entry $i$ in the
  $i$-th row of $T$ lies immediately above the left-most entry $i+1$ in the $i+1$-th row. Thus, by
  \fullref{Lemma}{lem:efdefinedfortableau}, $\e_i(w)$ is not defined. Hence $w$ is highest-weight. (Note that it is
  still necessary to show that $w$ is the \emph{unique} highest-weight word in its connected component of $\Gamma_n$.)

  Let $u$ be a quasi-ribbon word such that $\qrtableau{u}$ has shape $\alpha$ but has the property that the $j$-th row
  does not consist entirely of symbols $j$, for at least one $j \in \set{1,\ldots,\clen\alpha}$. Thus there is some
  symbol $i$ that appears in the $h$-th row, where $h < i$. Without loss of generality, assume $i$ is minimal. Clearly
  $i > 1$ since $h \geq 1$. Consider two cases:
  \begin{itemize}
  \item The left-most symbol in the $h$-th row is $i$. Then, by the minimality of $i$, the symbol $i-1$ cannot appear in
    $u$, for it would have to appear on the $h-1$-row of $\qrtableau{u}$. Thus $u$ does not contain a symbol $i-1$;
    hence $\e_{i-1}(u)$ is defined.
  \item The left-most symbol in the $h$-th row is not $i$. Then every symbol $i-1$ in $\qrtableau{u}$ must appear in a
    column strictly to the left of the left-most symbol $i$. Hence, by \fullref{Lemma}{lem:efdefinedfortableau},
    $\e_{i-1}(u)$ is defined.
  \end{itemize}
  Thus some quasi-Kashiwara operator $\e_i$ can be applied to $u$ and so $u$ is not a highest-weight word.

  This implies that $w$ is the unique highest-weight quasi-ribbon word that has shape $\alpha$. Furthermore, by
  \fullref{Corollary}{corol:qrwsameshapecomponent}, applying $\e_i$ raises the weight of a word but maintains the
  property of being a quasi-ribbon word and the shape $\alpha$ of its corresponding quasi-ribbon tableau. Therefore some
  sequence of operators $\e_i$ must transform the word $u$ to the word $w$. Thus the set of quasi-ribbon words whose
  corresponding quasi-ribbon tableau have shape $\alpha$ forms a connected component.
\end{proof}

\begin{figure}[t]
  \centering
  \includegraphics{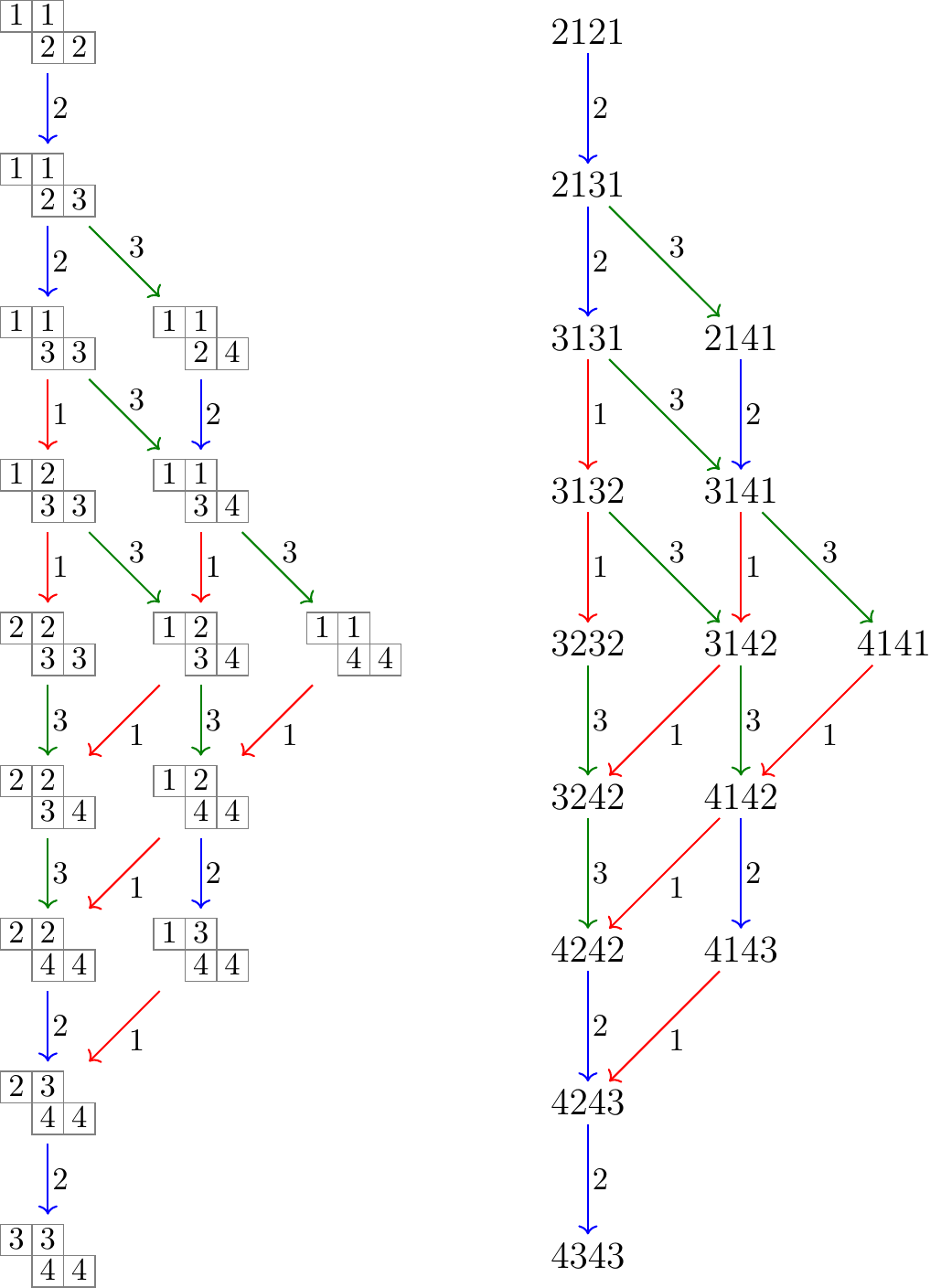}
  \caption{The isomorphic components $\Gamma_4(1212)$ (left) and $\Gamma_4(2121)$ (right) of the quasi-crystal graph
    $\Gamma_4$, with elements of $\Gamma_4(1212)$ drawn as quasi-ribbon tableau instead of written as words. The component
    $\Gamma_4(1212)$ consists of all quasi-ribbon words whose quasi-ribbon tableaux have shape $(2,2)$. None of the
    words in $\Gamma_4(2112)$ is a quasi-ribbon word.}
  \label{fig:gamma4tableaux}
\end{figure}

Thus, for example, the highest-weight quasi-ribbon word corresponding to a quasi-ribbon tableau of shape $(3,1,5,2)$ is
$11321333434$, corresponding to the quasi-ribbon tableau
\[
\tikz[tableau]\matrix
{
 1 \& 1 \& 1                          \\
   \&   \& 2                          \\
   \&   \& 3 \& 3 \& 3 \& 3 \& 3      \\
   \&   \&   \&   \&   \&   \& 4 \& 4 \\
};
\]
See also \fullref{Figures}{fig:gamma4} and \ref{fig:gamma4tableaux}, where the connected component $\Gamma_4(1212)$,
shown on the left, consists entirely of quasi-ribbon words corresponding to quasi-ribbon tableaux of shape $(2,2)$, and
its unique highest-weight word is $1212$.

Note that \fullref[(2)]{Proposition}{prop:charhighestweighttableau} only establishes the existence of unique
highest-weight words in connected components consisting of quasi-ribbon words, and not in every component.

\begin{corollary}
  \label{corol:charhighestweighttableaubyshapes}
  Let $\alpha$ be a composition, and let $w$ be a quasi-ribbon word such that $\qrtableau{w}$ has shape $\alpha$. Then
  $w$ is a highest-weight word if and only if $\wt{w} = \alpha$.
\end{corollary}

\begin{proof}
  Suppose $w$ is a highest-weight word. By \fullref[(2)]{Proposition}{prop:charhighestweighttableau}, the $j$-th row of
  $\qrtableau{w}$ consists entirely of symbols $j$, for $j=1,\ldots,\clen\alpha$. Thus $w$ contains exactly $\alpha_j$
  symbols $j$ for each $j$, and so $\wt{w} = \alpha$.

  Now suppose that $\wt{w} = \alpha$. In a quasi-ribbon tableau, a symbol $j$ can only appear in rows $1$ to $j$. Hence
  the $\alpha_1$ symbols $1$ in $\qrtableau{w}$ must appear in row $1$, which has length $\alpha_1$, so this row is full
  of symbols $1$. The $\alpha_2$ symbols $2$ must appear in rows $1$ and $2$, but row $1$ is full of symbols $1$, so row
  $2$, which has length $\alpha_2$, is full of symbols $2$. Continuing in this way, row $j$, which has length
  $\alpha_j$, is full of symbols $j$. By \fullref[(2)]{Proposition}{prop:charhighestweighttableau}, $w$ is a
  highest-weight word.
\end{proof}

\begin{corollary}
  \label{corol:qrthighestweightsameweightequal}
  Let $u,v \in \aA_n^*$ be highest-weight quasi-ribbon words. If $\wt{u} = \wt{v}$, then $u=v$.
\end{corollary}

\begin{proof}
  By \fullref{Corollary}{corol:charhighestweighttableaubyshapes}, since $u$ and $v$ are highest-weight, $\qrtableau{u}$
  has shape $\wt{u}$ and $\qrtableau{v}$ has shape $\wt{v}$. Since these weights are equal, $\qrtableau{u}$ and
  $\qrtableau{v}$ have the same shape. By the characterization of highest-weight words in
  \fullref[(2)]{Proposition}{prop:charhighestweighttableau}, it follows that $u = v$.
\end{proof}

\begin{proposition}
  \label{prop:uniqueqrtinsimclass}
  There is at most one quasi-ribbon word in each $\sim$-class.
\end{proposition}

\begin{proof}
  Suppose $u,v \in \aA_n^*$ are quasi-ribbon words with $u \sim v$. By
  \fullref{Proposition}{prop:charhighestweighttableau}, there is a sequence of operators $\e_{i_1},\ldots,\e_{i_k}$ such
  that $u' = \e_{i_1}\cdots\e_{i_k}(u)$ is highest-weight. Since $u \sim v$ implies that the words $u$ and $v$ are at
  the same location in the isomorphic components $\Gamma_n(u)$ and $\Gamma_n(v)$, the word
  $v' = \e_{i_1}\cdots\e_{i_k}(v)$ is also highest-weight.  The isomorphism from $\Gamma_n(u)$ to $\Gamma_n(v)$
  preserves weight, and so $\wt{u} = \wt{v}$. Thus $u=v$ by \fullref{Corollary}{corol:qrthighestweightsameweightequal}.
\end{proof}

\fullref{Proposition}{prop:uniqueqrtinsimclass} shows that each element of $H_n = \aA_n^*/{\sim}$ has at most one
representative as a quasi-ribbon word. The aim is now to show that if two words are $\hypocong$-related, then they are
$\sim$-related, which will establish a one-to-one correspondence between the elements of $H_n$ and quasi-ribbon
tableaux:

\begin{lemma}
  \label{lem:hypocongimpliessameedges}
  Let $u \in \aA_n^*$. Then $u$ and $\colreading{\qrtableau{u}}$ have exactly the same labelled edges incident to them
  in $\Gamma_n$.
\end{lemma}

\begin{proof}
  The aim is to prove that $\e_i(u)$ is defined if and only if $\e_i(\colreading{\qrtableau{u}})$ is defined. Similar
  reasoning shows that $\f_i(u)$ is defined if and only if $\f_i(\colreading{\qrtableau{u}})$ is defined.

  First, note that since $u \hypocong \colreading{\qrtableau{u}}$, it follows from \eqref{eq:charhypocong} that
  $\wt{u} = \wt{\colreading{\qrtableau{u}}}$.  Suppose $\e_i(u)$ is defined. Then $u$ contains a symbol $i+1$ but is
  $i$-inversion-free. Hence every symbol $i$ is to the left of every symbol $i+1$ in $u$. Therefore
  \fullref{Algorithm}{alg:hypoplacticinsert} appends the symbols $i+1$ to the right of any symbols $i$. Hence
  by~\fullref{Lemma}{lem:efdefinedfortableau}, $\e_i(\colreading{\qrtableau{u}})$ is defined. On the other hand, suppose
  that $\e_i(\colreading{\qrtableau{u}})$ is defined. Then by~\fullref{Lemma}{lem:efdefinedfortableau},
  $\colreading{\qrtableau{u}}$ contains a symbol $i+1$, but no symbol $i$ immediately above $i+1$ in a column. Hence the
  computation of $\qrtableau{u}$ using \fullref{Algorithm}{alg:hypoplacticinsert} cannot involve inserting a symbol $i$
  later than a symbol $i+1$. That is, $u$ is $i$-inversion-free. Since $\qrtableau{u}$ contains a symbol $i+1$, so does
  $u$, and thus $\e_i(u)$ is defined.
\end{proof}

\begin{lemma}
  \label{lem:eipreserveshypocong}
  Let $u \in \aA_n^*$ and $i \in \set{1,\ldots,n-1}$.
  \begin{enumerate}
  \item Suppose that $\e_i(u)$ is defined. Then $\e_i(u) \hypocong \e_i(\colreading{\qrtableau{u}})$ and so
    $\e_i(\colreading{\qrtableau{u}}) = \colreading{\qrtableau{\e_i(u)}}$.
  \item Suppose that $\f_i(u)$ is defined. Then $\f_i(u) \hypocong \f_i(\colreading{\qrtableau{u}})$ and so
    $\f_i(\colreading{\qrtableau{u}}) = \colreading{\qrtableau{\f_i(u)}}$.
  \end{enumerate}
\end{lemma}

\begin{proof}
  Suppose that $\e_i(u)$ is defined. Then it follows from \fullref{Lemma}{lem:hypocongimpliessameedges} that
  $\e_i(\colreading{\qrtableau{u}})$ is also defined. Now, $u \hypocong \colreading{\qrtableau{u}}$ by the definition of
  $\hypocong$. Thus it follows from \eqref{eq:charhypocong} that $\wt{u} = \wt{\colreading{\qrtableau{u}}}$ and so
  $\wt{\e_i(u)} = \wt{\e_i(\colreading{\qrtableau{u}})}$ since $\e_i$ replaces one symbol $i+1$ by a symbol $i$ in both
  words. Further, it again follows from \eqref{eq:charhypocong} that
  $\descomp{\std{u}^{-1}} = \descomp{\std{\colreading{\qrtableau{u}}}^{-1}}$. Hence, since $\e_i$ preserves
  standardizations by \fullref{Proposition}{prop:eipreservesstd}, it follows that
  $\descomp{\std{\e_i(u)}^{-1}} = \descomp{\std{\e_i(\colreading{\qrtableau{u}})}^{-1}}$. Combining this with the
  equality of weights and using \eqref{eq:charhypocong} again shows that $\e_i(u) \hypocong \e_i(\colreading{\qrtableau{u}})$.

  By \fullref{Corollary}{corol:qrwsameshapecomponent}, $\e_i(\colreading{\qrtableau{u}})$ is a quasi-ribbon word. Since
  it is $\hypocong$-related to $\e_i(u)$, it follows that
  $\e_i(\colreading{\qrtableau{u}}) = \colreading{\qrtableau{\e_i(u)}}$. This completes the proof of part~(1); similar
  reasoning proves part~(2).
\end{proof}

\begin{proposition}
  \label{prop:hypocongimpliessim}
  Let $u,v \in \aA_n^*$. Then $u \hypocong v \implies u \sim v$.
\end{proposition}

\begin{proof}
  By \fullref{Lemmata}{lem:hypocongimpliessameedges} and \ref{lem:eipreserveshypocong}, the map
  $\Theta : \Gamma_n(u) \to \Gamma_n(\colreading{\qrtableau{u}})$ with $\Theta(w) = \colreading{\qrtableau{w}}$ (for
  $w \in \Gamma_n(u)$) is a quasi-crystal isomorphism, and so $u \sim \colreading{\qrtableau{u}}$. Similarly,
  $v \sim \colreading{\qrtableau{v}}$.  It follows from $u \hypocong v$ that $\qrtableau{u} = \qrtableau{v}$. Since
  $\sim$ is transitive, $u \sim v$.
\end{proof}

Let $u \in \aA_n^*$. By \fullref{Proposition}{prop:hypocongimpliessim}, $\qrtableau{u}$ is a quasi-ribbon word that is
$\sim$-related to $u$. By \fullref{Proposition}{prop:uniqueqrtinsimclass}, $\qrtableau{u}$ is the \emph{unique}
quasi-ribbon word that is $\sim$-related to $u$. Thus the following result has been proven:

\begin{theorem}
Let $u,v \in \aA_n^*$. Then $u \hypocong v \iff u \sim v$.
\end{theorem}

\begin{corollary}
$\hypo_n = \aA_n^*/{\hypocong} \simeq \aA_n^*/{\sim} = H_n$.
\end{corollary}

In light of this, henceforth the quasi-crystal graph $\Gamma_n$ is denoted $\Gamma(\hypo_n)$, and the connected component
$\Gamma_n(w)$ is denoted $\Gamma(\hypo_n,w)$

Before moving on to study how the quasi-crystal graph interacts with the hypoplactic version of the Robinson--Schensted--Knuth
correspondence, it is necessary to prove one more fundamental property of the quasi-crystal
graph. \fullref[(2)]{Proposition}{prop:charhighestweighttableau} showed that the connected components comprising
quasi-ribbon words contain unique highest-weight words. The same holds for \emph{all} connected components:

\begin{proposition}
  In every connected component in $\Gamma(\hypo_n)$ there is a unique highest-weight word.
\end{proposition}

\begin{proof}
  Let $u \in \aA_n^*$. Since $u \sim \colreading{\qrtableau{u}}$, the connected component $\Gamma(\hypo_n,u)$ is isomorphic to
  $\Gamma(\hypo_n,\colreading{\qrtableau{u}})$. By \fullref[(2)]{Proposition}{prop:charhighestweighttableau}, there is a
  unique highest-weight word in $\Gamma(\hypo_n,\colreading{\qrtableau{u}})$. Consequently,
  $\Gamma(\hypo_n,u)$ contains a unique highest-weight word.
\end{proof}

In the crystal graph of the plactic monoid (that is the plactic monoid of type $A_n$), the highest-weight words are
characterized combinatorially as follows. Recall from \fullref{Subsection}{subsec:propertiesgammaplacn} that a
Yamanouchi word is a word $w$ in $\aA_n^*$ such for any suffix $v$ of $w$, it holds that
$|v|_1 \geq |v|_2 \geq \ldots \geq |v|_n$. The highest-weight words in the crystal graph of the plactic monoid are
precisely the Yamanouchi words \cite[\S~5.5]{lothaire_algebraic}.

Yamanouchi words do \emph{not} characterize highest-weight words in connected components of the quasi-crystal graph
$\Gamma(\hypo_n)$. For example, the highest-weight word in $\Gamma(\hypo_n,2112)$ is $2112$, which has the suffix
$v = 2$ that does not satisfy $|v|_1 \geq |v|_2$.

Let $\max(u)$ denote the largest symbol in the word $u$.

\begin{proposition}
  \label{prop:charhighestweight}
  A word $u \in \aA_n^*$ is highest-weight in a component of $\Gamma(\hypo_n)$ if and only if it contains all symbols in
  $\set{1,\ldots,\max(u)}$, with the condition that it has an $i$-inversion for all
  $i \in \set{1,\ldots,\max(u)-1}$.
\end{proposition}

\begin{proof}
  These are precisely the words for which all operators $\e_i$ are undefined.
\end{proof}

In the crystal graph $\Gamma(\plac_n)$, suffixes of highest-weight words are also highest-weight; this is an immediate
consequence of the definition of a Yamanouchi word. This does not hold in $\Gamma(\hypo_n)$: the highest-weight word
$2112$ has the suffix $2$, which is not highest-weight.

\section{The quasi-crystal graph and the Robinson--Schensted--Knuth correspondence}
\label{sec:robinsonschensted}

The previous sections constructed the quasi-crystal graph $\Gamma(\hypo_n)$ and showed that the hypoplactic congruence
$\hypocong$ corresponds to quasi-crystal isomorphisms (that is, weight-preserving labelled digraph isomorphisms) between
its connected components, just as the plactic congruence $\placcong$ corresponds to crystal isomorphisms (that is,
weight-preserving labelled digraph isomorphisms) between connected components of the crystal graph
$\Gamma(\plac_n)$. The following result shows that the interaction of the quasi-crystal graph and the hypoplactic
analogue of the Robinson--Schensted--Knuth correspondence exactly parallels the very elegant interaction of the crystal graph
and the usual Robinson--Schensted--Knuth correspondence. Just as the connected components of the crystal graph $\Gamma(\plac_n)$ are indexed
by standard Young tableaux, the connected components of the quasi-crystal graph $\Gamma(\hypo_n)$ are indexed by recording ribbons:

\begin{theorem}
  Let $u,v \in \aA_n^*$. The words $u$ and $v$ lie in the same connected component of $\Gamma(\hypo_n)$ if and only if
  $\recribbon{u} = \recribbon{v}$.
\end{theorem}

\begin{proof}
  Suppose that $u$ and $v$ lie in the same connected component of $\Gamma(\hypo_n)$. Note first that $|u| = |v|$. Let
  $u = u_1\cdots u_k$ and $v = v_1\cdots v_k$, where $u_h,v_h \in \aA_n$. By \fullref{Proposition}{prop:eipreservesstd},
  $\std{u} = \std{v}$. Therefore $\std{u_1\cdots u_h} = \std{v_1\cdots v_h}$ for all $h$ (this is immediate from the
  definition of standardization \cite[Lemma~2.2]{novelli_hypoplactic}). Thus $\qrtableau{u_1\cdots u_h}$ and $\qrtableau{v_1\cdots v_h}$
  both have shape $\descomp{\std{u_1\cdots u_h}^{-1}}$ for all $h$. The sequence of these shapes determines where new
  symbols are inserted during the computation of $\qrtableau{u}$ and $\qrtableau{v}$ by
  \fullref{Algorithm}{alg:hypoplacticinsert}, and so $\recribbon{u} = \recribbon{v}$.

  Now suppose $\recribbon{u} = \recribbon{v}$. Let $\hat{u}$ and $\hat{v}$ be the highest-weight words in
  $\Gamma(\hypo_n,u)$ and $\Gamma(\hypo_n,v)$, respectively. By the forward implications,
  $\recribbon{\hat{u}} = \recribbon{u} = \recribbon{v} = \recribbon{\hat{v}}$. Note that
  $\hat{u} \sim \colreading{\qrtableau{\hat{u}}}$ and $\hat{v} \sim \colreading{\qrtableau{\hat{v}}}$, and
  $\colreading{\qrtableau{\hat{u}}}$ and $\colreading{\qrtableau{\hat{v}}}$ are highest-weight quasi-ribbon
  words. Furthermore, the quasi-ribbon tableaux $\qrtableau{\hat{u}}$ and $\qrtableau{\hat{v}}$ have the same shape,
  since they both have the same shape as $\recribbon{\hat{u}} = \recribbon{\hat{v}}$. By
  \fullref{Corollary}{corol:charhighestweighttableaubyshapes},
  $\wt{\colreading{\qrtableau{\hat{u}}}} = \wt{\colreading{\qrtableau{\hat{v}}}}$; thus
  $\colreading{\qrtableau{\hat{u}}} = \colreading{\qrtableau{\hat{u}}}$ by
  \fullref{Corollary}{corol:qrthighestweightsameweightequal}. Since $\qrtableau{\hat{u}} = \qrtableau{\hat{u}}$ and
  $\recribbon{\hat{u}} = \recribbon{\hat{v}}$, it follows that $\hat{u} = \hat{v}$ by the quasi-ribbon tableau version
  of the Robinson--Schensted--Knuth correspondence (see the discussion in \fullref{Section}{sec:quasiribbontableau}). Hence
  $\Gamma(\hypo_n,u) = \Gamma(\hypo_n,\hat{u}) = \Gamma(\hypo_n,\hat{v}) = \Gamma(\hypo_n,v)$. This completes the proof
  of the reverse implication.
\end{proof}

\section{Structure of the quasi-crystal graph}
\label{sec:quasicrystalstructure}

The aim of this section is to study the structure of quasi-crystal graph and how it interacts with
usual crystal graph. Many of these results are illustrated in the components of $\Gamma(\hypo_4)$ and $\Gamma(\plac_4)$
shown in \fullref{Figures}{fig:gamma4partnonisom} to \ref{fig:gammaplac4containinggammahypo4}. These figures show both
the actions of quasi-Kashiwara operators (indicated by solid lines) and other actions of Kashiwara operators (indicated
by dotted lines). (Recall from \fullref{Proposition}{prop:keddefinedimpliesefdefined} that the action of the
quasi-Kashiwara operators is a restriction of the action of the Kashiwara operators.)

\begin{figure}[t]
  \centering
  \includegraphics{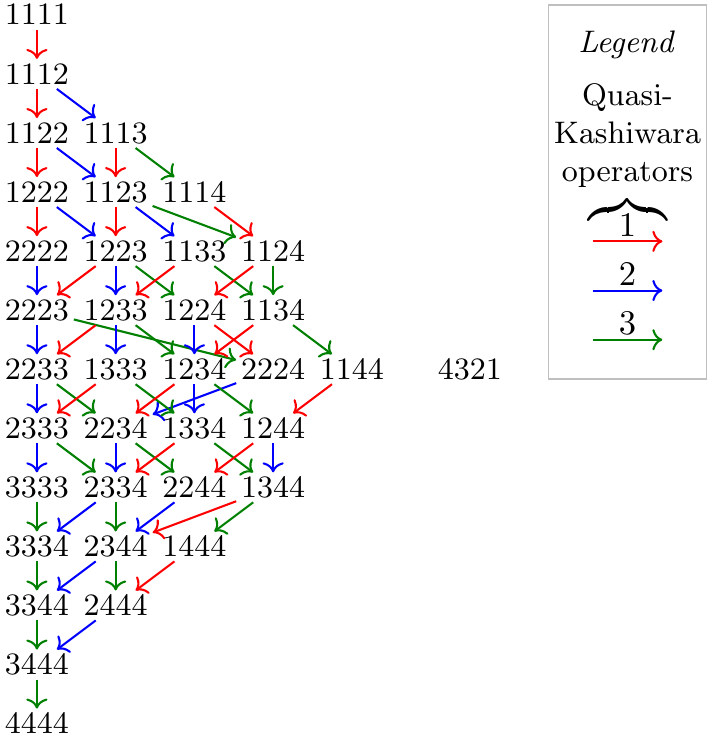}
  \caption{The connected components $\Gamma(\hypo_4,1111)$ and $\Gamma(\hypo_4,4321)$ of the quasi-crystal graph
    $\Gamma(\hypo_4)$. These are also the connected components $\Gamma(\plac_4,1111)$ and $\Gamma(\plac_4,4321)$ of the
    crystal graph $\Gamma(\plac_4)$: there is no vertex in these components where some $\ke_i$ or $\kf_i$ is defined but
    the corresponding $\e_i$ or $\f_i$ is not. These two components are not isomorphic to any other components of either
    $\Gamma(\hypo_4)$ or $\Gamma(\plac_4)$. The other components of $\Gamma(\hypo_n)$ and $\Gamma(\plac_n)$ whose
    vertices are length-$4$ words are shown in \fullref{Figures}{fig:gamma4partisom1} and \ref{fig:gamma4partisom2}.}%
  \label{fig:gamma4partnonisom}%
\end{figure}
\begin{figure}[p]
  \centering
  \includegraphics{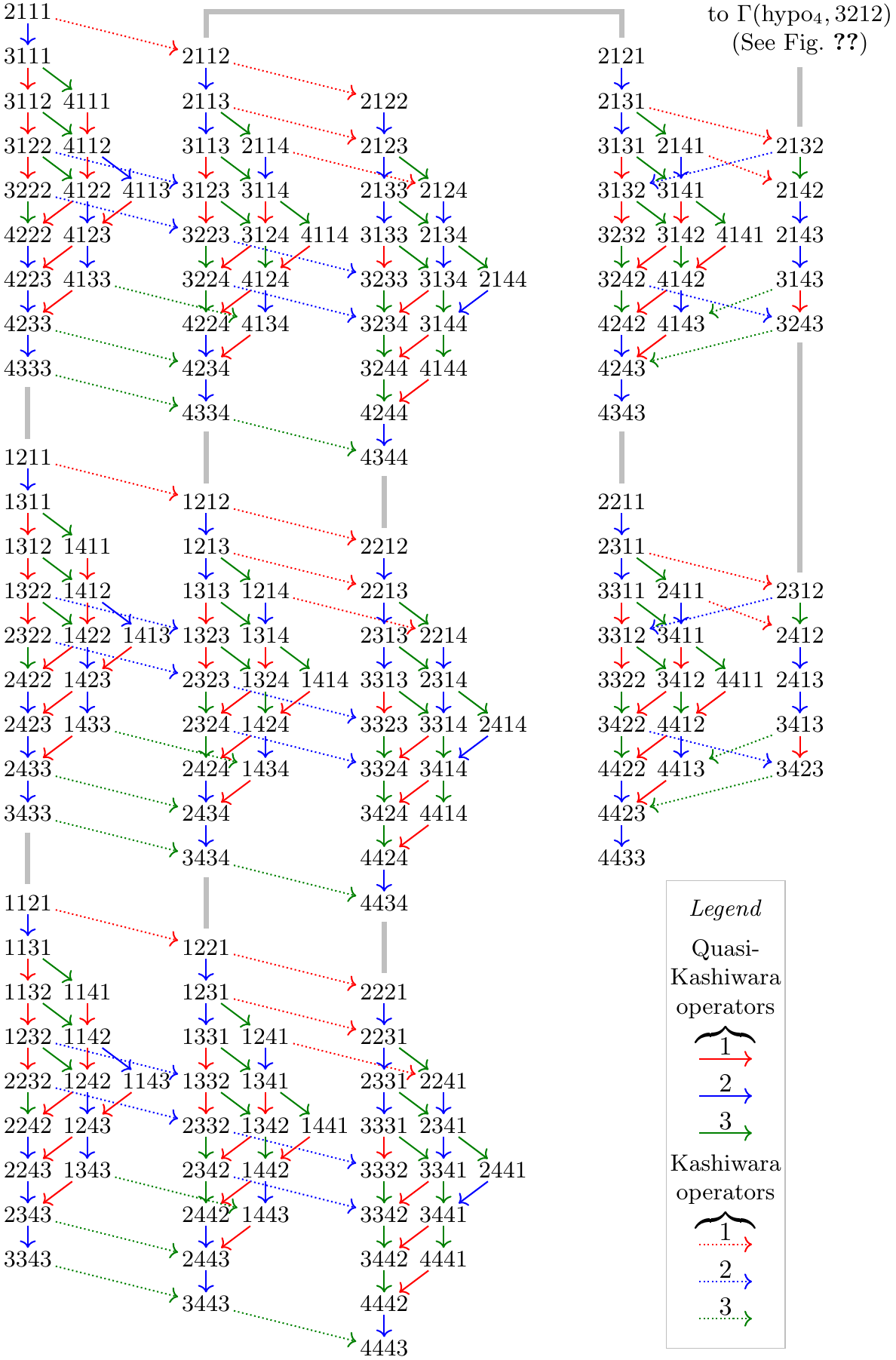}
  \caption{Thirteen components of $\Gamma(\hypo_4)$, which together lie inside five components of
    $\Gamma(\plac_n)$. Dotted lines show actions of operators $\ke_i$ and $\kf_i$ that are not also actions of operators
    $\e_i$ and $\f_i$. Each of these components of $\Gamma(\hypo_4)$ is isomorphic to at least one other component; grey
    lines link isomorphic components. Other components are shown in \fullref{Figures}{fig:gamma4partnonisom} and
    \ref{fig:gamma4partisom2}.}%
  \label{fig:gamma4partisom1}%
\end{figure}
\begin{figure}[t]
  \centering
  \includegraphics{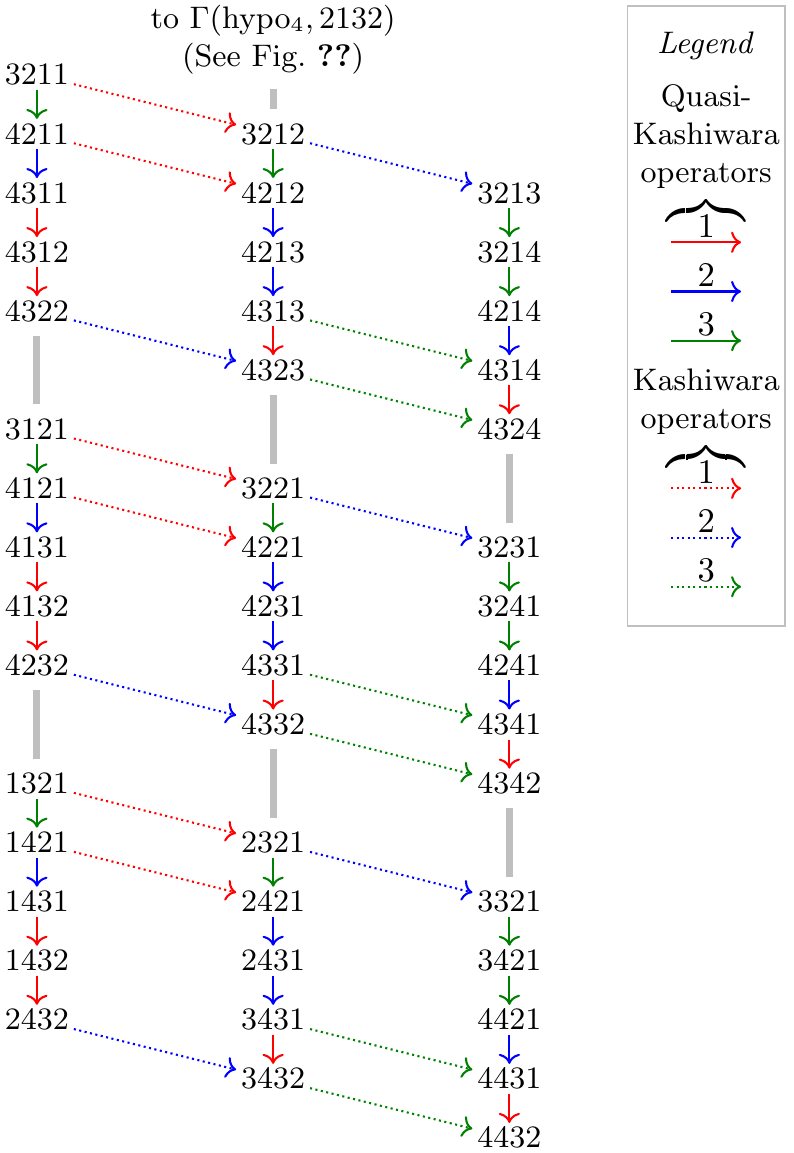}
  \caption{Nine components of $\Gamma(\hypo_4)$, which together lie inside three of the connected components
    $\Gamma(\plac_n)$. Dotted lines show actions of operators $\ke_i$ and $\kf_i$ that are not also actions of operators
    $\e_i$ and $\f_i$. Each of these components of $\Gamma(\hypo_4)$ is isomorphic to at least one other component; grey
    lines link isomorphic components. The other components whose vertices are length-$4$ words are shown in
    \fullref{Figures}{fig:gamma4partnonisom} and \ref{fig:gamma4partisom1}.}%
  \label{fig:gamma4partisom2}%
\end{figure}

\subsection{Sizes of classes and quasi-crystals}

As previously noted, by fixing a Young tableau $P$ and varying $Q$ over all standard Young tableaux of the same shape as
$P$, one obtains the plactic class corresponding to $P$. Notice that one obtains as a corollary that the size of this
class depends on the \emph{shape} of $P$, not on the \emph{entries} of $P$. The so-called hook-length formula gives
the number of standard Young tableaux of a given shape (see \cite{frame_hookgraphs} and \cite[\S~4.3]{fulton_young}) and
thus of the size of a plactic class corresponding to an element of that shape.

This section uses the quasi-crystal graph to give an analogue of the hook-length formula for quasi-ribbon tableaux, in
the sense of stating a formula for the size of hypoplactic classes corresponding to quasi-ribbon tableau of a given
shape. Novelli \cite[Theorem~5.1]{novelli_hypoplactic} proves that for any compositions $\alpha = (\alpha_1,\ldots,\alpha_h)$ and
$\gamma$,
\begin{equation}
\label{eq:hypoplacticclasssizenovelli}
\sum_{\beta \preceq \alpha} \#(\beta,\gamma) = \binom{\alpha_1+\ldots+\alpha_h}{\alpha_1,\ldots,\alpha_h},
\end{equation}
where $\#(\beta,\gamma)$ is the size of the hypoplactic class corresponding to the quasi-ribbon tableau of shape
$\beta$ and content $\gamma$. (Recall that $\beta\preceq\alpha$ denotes that $\beta$ is a coarser composition than
$\alpha$; see \fullref{Section}{sec:quasiribbontableau}.) Since the size of the hypoplactic class corresponding to a
quasi-ribbon tableau of shape $(\alpha_1)$ is always $1$, the formula \eqref{eq:hypoplacticclasssizenovelli} allows one
to iteratively compute the size of an arbitrary hypoplactic class.

However, it is not immediately clear from \eqref{eq:hypoplacticclasssizenovelli} that the size of a hypoplactic class is only
dependent on the \emph{shape} of the quasi-ribbon tableau, not on its content:

\begin{proposition}
  \label{prop:hypoplacticclasssizeshape}
  Hypoplactic classes corresponding to quasi-ribbon tableaux of the same shape all have the same size.
\end{proposition}

\begin{proof}
  Let $u,v \in \aA_n^*$ be such that $\qrtableau{u}$ and $\qrtableau{v}$ have the same shape. Then
  $\colreading{\qrtableau{u}}$ and $\colreading{\qrtableau{v}}$ lie in the same connected component by
  \fullref[(1)]{Proposition}{prop:charhighestweighttableau}. Thus there is a sequence $\g_{i_1},\ldots,\g_{i_r}$ of
  operators $\e_i$ and $\f_i$ such that
  $\g_{i_1}\cdots\g_{i_r}(\colreading{\qrtableau{u}}) = \colreading{\qrtableau{v}}$. Each $\e_i$ and $\f_i$, when
  defined, is a bijection between $\hypocong$-classes, and so $\brackets[\big]{\colreading{\qrtableau{u}}}_{\hypocong}$ and
  $\brackets[\big]{\colreading{\qrtableau{v}}}_{\hypocong}$ have the same size. Since $u \sim \colreading{\qrtableau{u}}$ and
  $v \sim \colreading{\qrtableau{v}}$, it follows that $[u]_{\hypocong}$ and $[v]_{\hypocong}$ have the same size.
\end{proof}

The formula for the size of hypoplactic classes is also straightforward when one uses the quasi-crystal graph:

\begin{theorem}
  \label{thm:hypoplacticclasssize}
  The size of any hypoplactic class in $\aA_n^*$ whose quasi-ribbon tableau has shape $\alpha$ is
  \begin{equation}
    \label{eq:hypoplacticclasssize}
    \begin{cases}
      {\displaystyle\sum_{\beta\preceq\alpha}(-1)^{\clen{\alpha}-\clen{\beta}} \binom{\cwt\beta}{\beta_1,\cdots,\beta_{\clen\beta}}} & \text{if $\clen\alpha \leq n$,} \\
        0 & \text{otherwise.}
    \end{cases}
  \end{equation}
\end{theorem}

\begin{proof}
  Note first that in a quasi-ribbon tableau, any symbol in $\aA_n$ must lie in the first $n$ rows. Thus if
  $\clen\alpha > n$, then there is no quasi-ribbon tableau of shape $\alpha$ with entries in $\aA_n$ and so the
  corresponding hypoplactic class if empty. So assume henceforth that $\clen\alpha \leq n$.

  Let $T$ be a quasi-ribbon tableau of shape $\alpha$; the aim is to describe the cardinality of the set
  $U_T = \gset{u \in \aA_n^*}{\qrtableau{u} = T}$. Since the operators $\e_i$ and $\f_i$ are bijections between
  hypoplactic classes, assume without loss of generality that $\colreading{T}$ is highest-weight. By
  \fullref{Corollary}{corol:charhighestweighttableaubyshapes}, $\wt{\colreading{T}} = \alpha$. Since all words in a
  hypoplactic class have the same weight, every word in $U_T$ has weight $\alpha$ (and thus contains exactly $\alpha_i$
  symbols $i$, for each $i \in \aA_n$). Furthermore, all words in $U_T$ are highest-weight and so have $i$-inversions
  for each $i \in \set{1,\ldots,\clen{\alpha}}$. Thus $U_T$ consists of exactly the words of weight $\alpha$, but that
  \emph{do not} have the property of containing all symbols $i$ to the left of all symbols $i+1$ for some
  $i \in \set{1,\ldots,\clen{\alpha}-1}$.

  For any weak composition $\gamma$ with $\clen\gamma \leq n$, the number of words in $\aA_n^*$ with weight $\gamma$ is
  \[
  \binom{\cwt{\gamma}}{\gamma_1,\cdots,\gamma_{\clen\gamma}}.
  \]
  Consider a word $u \in \aA_n^*$ with $\wt{u} = \gamma$ that is $i$-inversion-free. Then in $u$, every symbol $i$ lies
  to the left of every symbol $i+1$. Replacing each symbol $j+1$ by $j$ for each $j \geq i$ yields a word $u'$ with
  weight $\gamma' = (\gamma_1,\ldots,\gamma_{i-1},\gamma_i+\gamma_{i+1},\gamma_{i+2}\ldots,\gamma_n)$; note that
  $\gamma' \preceq \gamma$ and $\clen{\gamma'} = \clen\gamma - 1$. On the other hand, starting from $u'$ and replacing
  each symbol $j$ by $j+1$ for $j \geq i+1$ and replacing the rightmost $\gamma_{i+1}$ symbols $i$ by $i+1$ yields
  $u$. Thus there is a one-to-one correspondence between words of weight $\gamma$ that are $i$-inversion-free and words
  of weight $\gamma'$.

  Iterating this argument shows that there is a one-to-one correspondence between words of weight $\gamma$ that are
  $i$-inversion-free for $i \in I \subseteq \set{1,\ldots,n-1}$ and words of weight $\gamma_I$ for a (uniquely
  determined) $\gamma_I \preceq \gamma$.

  Hence, by the inclusion--exclusion principle, the number of such words that are $i$-inversion-free for all
  $i \in \set{1,\ldots,n-1}$ is given by \eqref{eq:hypoplacticclasssize}.
\end{proof}

For example, let $\alpha = (2,1,1,2)$. By \fullref{Theorem}{thm:hypoplacticclasssize}, the size of the
hypoplactic class whose quasi-ribbon tableau has shape $\alpha$ is
\begin{align*}
\frac{6!}{2!1!1!2!}
&- \frac{6!}{3!1!2!} - \frac{6!}{2!2!2!} - \frac{6!}{2!1!3!} \\
&+ \frac{6!}{4!2!} + \frac{6!}{3!3!} + \frac{6!}{2!4!} \\
&- \frac{6!}{6!} = 19.
\end{align*}
The quasi-ribbon word $143214$ corresponds to a quasi-ribbon tableau with shape $\alpha$, and the hypoplactic
class containing $u$ is
\begin{align*}
\set{&143214,413214,431214,432114,\\
&143241,413241,431241,432141,\\
&143421,413421,431421,432411,\\
&144321,414321,434121,434211,\\
&\phantom{414321,}\,441321,443121,443211},
\end{align*}
which contains $19$ elements, as expected.

For the sake of completeness, this section closes with a formula for the number of quasi-ribbon tableaux of a given
shape, which allows one to compute the size of a connected component of $\Gamma(\hypo_n)$. Notice that the proof of this
does not depend on applying the crystal structure.

\begin{theorem}
\label{thm:noqrtofshapei}
The number of quasi-ribbon tableaux of shape $\alpha$ and symbols from $\aA_n$ is
\[
\begin{cases}
\binom{n+\abs{\alpha}-\clen\alpha}{n-\clen\alpha} & \text{if $\clen\alpha \leq n$} \\
0 & \text{if $\clen\alpha > n$}.
\end{cases}
\]
\end{theorem}

\begin{proof}
  Consider a ribbon diagram of shape $\alpha$, where the boundaries between the cells, including the left boundary of
  the first cell and the right boundary of the last cell, are indexed by the numbers $0,1,\ldots,\abs{\alpha}$. Notice that
  $D(\alpha)$ be the set of indices of boundaries between vertically adjacent cells. For example, for shape
  $\alpha = (4,4,2,3)$, the indices are as follows:
  \[
  \begin{tikzpicture}
    \matrix[tableaumatrix] (t) {
      \null \& \null \& \null \& \null    \\
      \&\&\&\null \&\null \&\null \&\null \\
      \&\&\&\&\&\&\null \&\null           \\
      \&\&\&\&\&\&\&\null \&\null \&\null \\
    };
    \begin{scope}[every node/.append style={font=\scriptsize,inner sep=0.25mm}]
      \node (l0) at ($ (t-1-1.north west) + (0,3mm) $) {$0$};
      \node (l1) at ($ (t-1-2.north west) + (0,3mm) $) {$1$};
      \node (l2) at ($ (t-1-3.north west) + (0,3mm) $) {$2$};
      \node (l3) at ($ (t-1-4.north west) + (0,3mm) $) {$3$};
      \node (l4) at ($ (t-2-4.north west) + (-3mm,-3mm) $) {$4$};
      \node (l5) at ($ (t-2-5.south west) + (0,-3mm) $) {$5$};
      \node (l6) at ($ (t-2-6.north west) + (0,3mm) $) {$6$};
      \node (l7) at ($ (t-2-7.north west) + (0,3mm) $) {$7$};
      \node (l8) at ($ (t-3-7.north west) + (-3mm,-3mm) $) {$8$};
      \node (l9) at ($ (t-3-8.north west) + (3mm,3mm) $) {$9$};
      \node (l10) at ($ (t-4-8.north west) + (-3mm,-3mm) $) {$10$};
      \node (l11) at ($ (t-4-9.south west) + (0mm,-3mm) $) {$11$};
      \node (l12) at ($ (t-4-10.north west) + (0,3mm) $) {$12$};
      \node (l13) at ($ (t-4-10.north east) + (0,3mm) $) {$13$};
    \end{scope}
    \begin{scope}[line width=.6pt]
      \draw ($ (t-1-1.north west) + (0,-3.5mm) $) to (l0);
      \draw ($ (t-1-2.north west) + (0,-3.5mm) $) to (l1);
      \draw ($ (t-1-3.north west) + (0,-3.5mm) $) to (l2);
      \draw ($ (t-1-4.north west) + (0,-3.5mm) $) to (l3);
      \draw ($ (t-2-4.north west) + (3.5mm,0) $) to (t-2-4.north west) to (l4);
      \draw ($ (t-2-5.south west) + (0,3.5mm) $) to (l5);
      \draw ($ (t-2-6.north west) + (0,-3.5mm) $) to (l6);
      \draw ($ (t-2-7.north west) + (0,-3.5mm) $) to (l7);
      \draw ($ (t-3-7.north west) + (3.5mm,0) $) to (t-3-7.north west) to (l8);
      \draw ($ (t-3-8.north west) + (0,-3.5mm) $) to (t-3-8.north west) to (l9);
      \draw ($ (t-4-8.north west) + (3.5mm,0) $) to (t-4-8.north west) to (l10);
      \draw ($ (t-4-9.south west) + (0,3.5mm) $) to (l11);
      \draw ($ (t-4-10.north west) + (0,-3.5mm) $) to (l12);
      \draw ($ (t-4-10.north east) + (0,-3.5mm) $) to (l13);
    \end{scope}
  \end{tikzpicture}
  \]
  In this case, $D(\alpha) = \set{4,8,10}$.

  A filling of such a quasi-ribbon tableau by symbols from $\aA_n$ is weakly increasing from upper left to lower right,
  and is specified exactly by listing $n-1$ boundaries between adjacent cells where the increase from $i$ to $i+1$
  occurs in this filling. Note that such a list may contain repeated entries (and is thus formally a multiset),
  indicating that the difference between the entries in the cells incident on this boundary differ by more than $1$. For
  example, for $n=9$ the filling
  \[
  \begin{tikzpicture}
    \matrix[tableaumatrix] (t) {
      2 \& 4 \& 4 \& 4 \\
      \&\&\&5 \&5 \&6 \&7 \\
      \&\&\&\&\&\&8 \&8 \\
      \&\&\&\&\&\&\&9 \&9 \&9 \\
    };
    \begin{scope}[every node/.append style={font=\scriptsize,inner sep=0.25mm}]
      \node (l0) at ($ (t-1-1.north west) + (0,3mm) $) {$0$};
      \node (l1) at ($ (t-1-2.north west) + (0,3mm) $) {$1$};
      \node (l2) at ($ (t-1-3.north west) + (0,3mm) $) {$2$};
      \node (l3) at ($ (t-1-4.north west) + (0,3mm) $) {$3$};
      \node (l4) at ($ (t-2-4.north west) + (-3mm,-3mm) $) {$4$};
      \node (l5) at ($ (t-2-5.south west) + (0,-3mm) $) {$5$};
      \node (l6) at ($ (t-2-6.north west) + (0,3mm) $) {$6$};
      \node (l7) at ($ (t-2-7.north west) + (0,3mm) $) {$7$};
      \node (l8) at ($ (t-3-7.north west) + (-3mm,-3mm) $) {$8$};
      \node (l9) at ($ (t-3-8.north west) + (3mm,3mm) $) {$9$};
      \node (l10) at ($ (t-4-8.north west) + (-3mm,-3mm) $) {$10$};
      \node (l11) at ($ (t-4-9.south west) + (0mm,-3mm) $) {$11$};
      \node (l12) at ($ (t-4-10.north west) + (0,3mm) $) {$12$};
      \node (l13) at ($ (t-4-10.north east) + (0,3mm) $) {$13$};
    \end{scope}
    \begin{scope}[line width=.6pt]
      \draw ($ (t-1-1.north west) + (0,-3.5mm) $) to (l0);
      \draw ($ (t-1-2.north west) + (0,-3.5mm) $) to (l1);
      \draw ($ (t-1-3.north west) + (0,-3.5mm) $) to (l2);
      \draw ($ (t-1-4.north west) + (0,-3.5mm) $) to (l3);
      \draw ($ (t-2-4.north west) + (3.5mm,0) $) to (t-2-4.north west) to (l4);
      \draw ($ (t-2-5.south west) + (0,3.5mm) $) to (l5);
      \draw ($ (t-2-6.north west) + (0,-3.5mm) $) to (l6);
      \draw ($ (t-2-7.north west) + (0,-3.5mm) $) to (l7);
      \draw ($ (t-3-7.north west) + (3.5mm,0) $) to (t-3-7.north west) to (l8);
      \draw ($ (t-3-8.north west) + (0,-3.5mm) $) to (t-3-8.north west) to (l9);
      \draw ($ (t-4-8.north west) + (3.5mm,0) $) to (t-4-8.north west) to (l10);
      \draw ($ (t-4-9.south west) + (0,3.5mm) $) to (l11);
      \draw ($ (t-4-10.north west) + (0,-3.5mm) $) to (l12);
      \draw ($ (t-4-10.north east) + (0,-3.5mm) $) to (l13);
    \end{scope}
  \end{tikzpicture}
  \]
  corresponds to the multiset
  \[
  \mset{0,1,1,4,6,7,8,10}.
  \]
  Note that this multiset has length $n-1 = 8$ and contains the entry $0$, indicating the presence of $2$ in the first
  cell, and a repeated entry $1$, indicating the jump from $2$ to $4$ at this boundary. However, some of the entries in
  this multiset are forced by the shape of the tableau: there must be increases at boundaries $4$, $8$, and $10$, which
  are between vertically adjacent cells.

  There is thus a one-to-one correspondence between multisets with $n-1$ elements drawn from $\set{0,\ldots,\abs{\alpha}}$ and
  that contain $D(\alpha)$ (which indicates the position of the `forced' increases) and fillings of a tableau of shape
  $\alpha$. Since $|D(\alpha)| = \clen\alpha - 1$, the number of such multisets is $0$ if $\clen\alpha > n$, and is
  otherwise the number of multisets with $(n-1)-(\clen\alpha-1)$ elements drawn from $\set{0,\ldots,\abs{\alpha}}$, which is
  $\binom{(\abs{\alpha}+1)+((n-1)-(\clen\alpha-1))-1}{(n-1)-(\clen\alpha-1)}$ by the standard formula for the number of multisets
  \cite[\S~1.2]{stanley_enumerative1}, which simplifies to $\binom{n+\abs{\alpha}-\clen\alpha}{n-\clen\alpha}$.
\end{proof}

\subsection{Interaction of the crystal and quasi-crystal graphs}

This section examines the interactions of the crystal graph $\Gamma(\plac_n)$ and quasi-crystal graph
$\Gamma(\hypo_n)$. The first, and most fundamental, observation, is how connected components in $\Gamma(\plac_n)$ are
made up of connected components in $\Gamma(\hypo_n)$:

\begin{proposition}
  The vertex set of every connected component of $\Gamma(\plac_n)$ is a union of vertex sets of connected components of
  $\Gamma(\hypo_n)$.
\end{proposition}

\begin{proof}
  By \fullref{Proposition}{prop:keddefinedimpliesefdefined}, any edge in $\Gamma(\hypo_n)$ (whose edges indicate the
  action of the quasi-Kashiwara operators $\e_i$ and $\f_i$) is also an edge in $\Gamma(\plac_n)$ (whose edges indicate
  the action of the Kashiwara operators $\ke_i$ and $\kf_i$). Hence every connected component in $\Gamma(\hypo_n)$ lies
  entirely within a connected component of $\Gamma(\plac_n)$; the result follows immediately.
\end{proof}

For example, as shown in \fullref{Figure}{fig:gamma4partisom1}, the connected component $\Gamma(\plac_4,2111)$ is made up of
the three connected components $\Gamma(\hypo_4,2111)$, $\Gamma(\hypo_4,2112)$, and $\Gamma(\hypo_4,2122)$.

\begin{proposition}
  \label{prop:isomorphismsrestrict}
  Let $\Theta : \Gamma(\plac_n,u) \to \Gamma(\plac_n,\Theta(u))$ be a crystal isomorphism. Then $\Theta$ restricts to a
  quasi-crystal isomorphism from $\Gamma(\hypo_n,w)$ to $\Gamma(\hypo_n,\Theta(w))$, for all $w \in \Gamma(\plac_n,u)$.
\end{proposition}

\begin{proof}
  Let $\Theta : \Gamma(\plac_n,u) \to \Gamma(\plac_n,\Theta(u))$ be a crystal isomorphism and let
  $w \in \Gamma(\plac_n,u)$. Suppose the Kashiwara operator $\kf_i$ is defined on $w$ but the quasi-Kashiwara operator
  $\f_i$ is not defined on $w$. Since $\kf_i$ is defined, $w$ must contain at least one symbol $i$. Thus, since $\f_i$
  is undefined, $w$ must have an $i$-inversion. Since $\Theta$ is a crystal isomorphism, $w \placcong \Theta(w)$. The
  defining relations in $\drel{R}_\plac$ preserve the property of having $i$-inversions (for if a defining relation in
  $\drel{R}_\plac$ commutes a symbol $i$ and $i+1$, then $a = i$ and $c= i+1$ in this defining relation, and so
  $b \in \set{i,i+1}$, and so the applied relation is either $(i(i+1)i,(i+1)ii)$ or $((i+1)i(i+1),(i+1)(i+1)i)$, and
  both sides of these relations have $i$-inversions). Hence $\Theta(w)$ has an $i$-inversion. Hence
  $\f_i$ is not defined on $w$.

  A symmetric argument shows that if $\kf_i$ is defined on $\Theta(w)$ but $\f_i$ is not, then $\f_i$ is not defined on
  $w$. Hence the isomorphism $\Theta$ maps edges corresponding to actions of quasi-Kashiwara operators in
  $\Gamma(\plac_n,u)$ to edges corresponding to actions of quasi-Kashiwara operators in $\Gamma(\plac_n,\Theta(u))$, and
  vice versa, and so restricts to a quasi-crystal isomorphism from $\Gamma(\hypo_n,w)$ to $\Gamma(\hypo_n,\Theta(w))$
  for all $w \in \Gamma(\plac_n,u)$.
\end{proof}

\begin{figure}[t]
  \centering
  \includegraphics{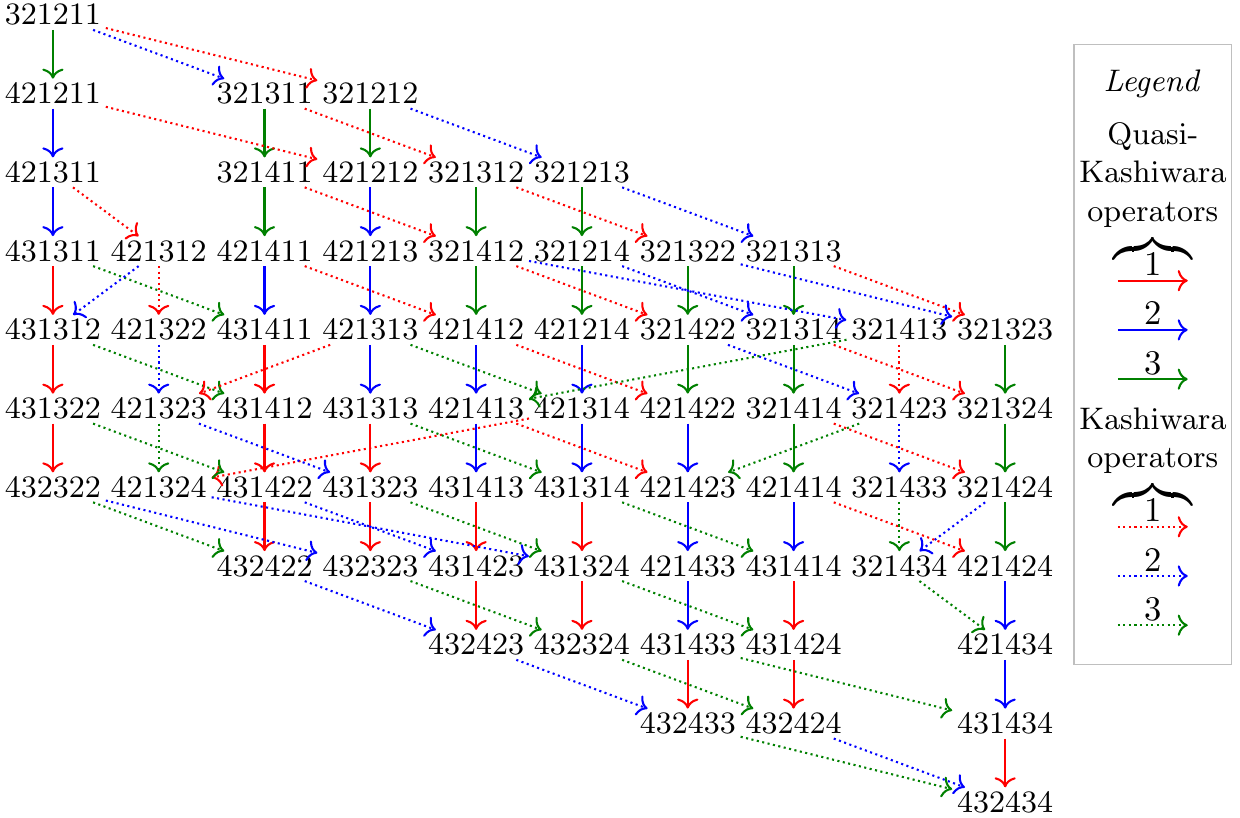}
  \caption{The connected component $\Gamma(\plac_4,321211)$, drawn so that the connected components of $\Gamma(\hypo_4)$
    are arranged vertically. Notice that it contains the isomorphic connected components $\Gamma(\hypo_4,321213)$ and
    $\Gamma(\hypo_4,321312)$, and also the isomorphic one-vertex connected components $\Gamma(\hypo_4,421323)$ and
    $\Gamma(\hypo_4,321423)$.}
  \label{fig:gammaplac4containinggammahypo4}
\end{figure}

Notice that it is possible for a single connected component of $\Gamma(\plac_n)$ to contain distinct isomorphic
connected components of $\Gamma(\hypo_n)$. For example, as shown in
\fullref{Figure}{fig:gammaplac4containinggammahypo4}, the connected component $\Gamma(\plac_4,321211)$ contains the
isomorphic connected components $\Gamma(\hypo_4,321213)$ and $\Gamma(\hypo_4,321312)$.

Note also that $\Gamma(\plac_4,321211)$, contains the isomorphic one-vertex connected components
$\Gamma(\hypo_4,421323)$ and $\Gamma(\hypo_4,321423)$. (There are other one-vertex components of $\Gamma(\hypo_4)$
inside $\Gamma(\plac_4,321211)$, but these have different weights and so there are no quasi-crystal isomorphisms between
them.) The quasi-ribbon tableau $\qrtableau{421323}$ has shape $(1,2,2,1)$; thus, by
\fullref{Theorem}{thm:hypoplacticclasssize}, the hypoplactic class containing $421323$ has size
\begin{align*}
\frac{6!}{1!2!2!1!}
&- \frac{6!}{3!2!1!} - \frac{6!}{1!4!1!} - \frac{6!}{1!2!3!} \\
&+ \frac{6!}{5!1!} + \frac{6!}{3!3!} + \frac{6!}{1!5!} \\
&- \frac{6!}{6!} = 
61.
\end{align*}
Since there are an odd number of components that are isomorphic to $\Gamma(\hypo_4,421323)$, it follows that there must
be at least one component of $\Gamma(\plac_4)$ that contains an odd number of these components. Thus different
components of $\Gamma(\plac_4)$ may contain different numbers of isomorphic components of $\Gamma(\hypo_4)$.

\begin{corollary}
  \label{corol:qrtcomponentsindifferentplaces}
  Let $u,v \in \aA_n^*$ be such that $u \placcong v$ but $u \neq v$, so that there is a non-trivial crystal isormorphism
  $\Theta : \Gamma(\plac_n,u) \to \Gamma(\plac_n,v)$ with $\Theta(u) = v$. Let $s \in \Gamma(\plac_n,u)$ and
  $t \in \Gamma(\plac_n,v)$ be quasi-ribbon words. Then $\Theta$ does not map $\Gamma(\hypo_n,s)$ to
  $\Gamma(\hypo_n,t)$. More succinctly, quasi-ribbon word components of $\Gamma(\hypo_n)$ cannot lie in the same places
  in distinct isomorphic components of $\Gamma(\plac_n)$.
\end{corollary}

\begin{proof}
  Without loss of generality, assume $s$ and $t$ are highest-weight in $\Gamma(\hypo_n,s)$ and $\Gamma(\hypo_n,t)$
  respectively. Suppose, with the aim of obtaining a contradiction, that $\Theta$ maps $\Gamma(\hypo_n,s)$ to
  $\Gamma(\hypo_n,t)$. Then $\Theta(s) = t$, and so $s \hypocong t$. Since $s$ and $t$ are quasi-ribbon words, $s=t$ by
  \fullref{Proposition}{prop:charhighestweighttableau} and so $\Theta$ is trivial, which is a contradiction.
\end{proof}

\fullref{Corollary}{corol:qrwsameshapecomponent} showed that the quasi-Kashiwara operators preserve shapes of
quasi-ribbon tableau. In fact, quasi-Kashiwara operators and, more generally, Kashiwara operators, preseve shapes of
quasi-ribbon tabloids (see \fullref{Section}{sec:quasiribbontableau} for the definitions of quasi-ribbon tabloids):

\begin{proposition}
  \label{prop:kekfpreserveqrtshapes}
  Let $i \in \set{1,\ldots,n-1}$. Let $u \in \aA_n^*$.
  \begin{enumerate}
  \item If the Kashiwara operator $\ke_i$ is defined on $u$, then $\qrtabloid{\ke_i(u)}$ and $\qrtabloid{u}$ have the
    same shape.
  \item If the Kashiwara operator $\kf_i$ is defined on $u$, then $\qrtabloid{\kf_i(u)}$ and $\qrtabloid{u}$ have the
    same shape.
  \end{enumerate}
\end{proposition}

\begin{proof}
  Let $u \in \aA_n^*$ and let $u = u^{(1)}\cdots u^{(m)}$ be the factorization of $u$ into maximal decreasing factors
  (which are entries of the columns of $\qrtabloid{u}$).

  Suppose that the Kashiwara operator $\ke_i$ is defined on $u$, and that the application of $\ke_i$ to $u$ replaces the
  (necessarily unique) symbol $i+1$ in $u^{(k)}$ by a symbol $i$; let $\hat{u}^{(k)}$ be the result of this
  replacement. Then $u^{(k)}$ cannot contain a symbol $i$, for if it did, then during the computation of the action of
  $\ke_i$ as described in \fullref{Subsection}{subsec:computingkashiwara}, the symbols $i+1$ and $i$ in $u^{(k)}$ (which
  would be adjacent since $u^{(k)}$ is strictly decreasing) would have been replaced by $-$ and $+$ and so would have
  been deleted, and so $\ke_i$ would not act on this symbol $i+1$. Hence $\hat{u}^{(k)}$ is also a decreasing word.

  Furthermore, the first symbol of $u^{(k+1)}$ is greater than or equal to the last symbol of $u^{(k)}$, and so is
  certainly greater than or equal to the last symbol of $\hat{u}^{(k)}$ since $\ke_i$ can only decrease a
  symbol.

  Similarly, the first symbol of $u^{(k)}$ is greater than or equal to the last symbol of $u^{(k-1)}$. If $u^{(k)}$
  does not start with the symbol $i+1$, the first symbol of $\hat{u}^{(k)}$ is greater than or equal to the last symbol of
  $u^{(k-1)}$. So assume $u^{(k)}$ starts with the symbol $i+1$; since the factorization is into maximal decreasing
  factors, $u^{(k-1)}$ ends with a symbol that is less than or equal to $i+1$. If $u^{(k-1)}$ ends with a symbol that
  is strictly less than $i+1$, then the first symbol of $\hat{u}^{(k)}$ is greater than or equal to the last symbol of
  $u^{(k-1)}$. So assume $u^{(k-1)}$ ends with the symbol $i+1$. Then during the computation of the action of $\ke_i$ as
  described in \fullref{Subsection}{subsec:computingkashiwara}, the adjacent symbols $i+1$ at the end of $u^{(k-1)}$ and
  at the start of $u^{(k)}$ are both replaced by symbols $-$, and neither of these symbols are removed by deletion of
  factors ${-}{+}$, since $\ke_i$ acts on the symbol $i+1$ at the start of $u^{(k)}$. But this contradicts the fact
  that $\ke_i$ acts on the symbol replaced by the leftmost ${-}$. Thus this case cannot arise, and so one of the
  previous possibilities must have held true.

  Combining the last three paragraphs shows that the factorization of $\ke_i(u)$ into maximal decreasing factors is
  $\ke_i(u) = u^{(1)}\cdots u^{(k-1)}\hat{u}^{(k)}u^{(k+1)}\cdots u^{(m)}$. This proves part~(1). Similar reasoning for
  $\kf_i$ proves part~(2).
\end{proof}

\begin{figure}[t]
  \centering
  \includegraphics{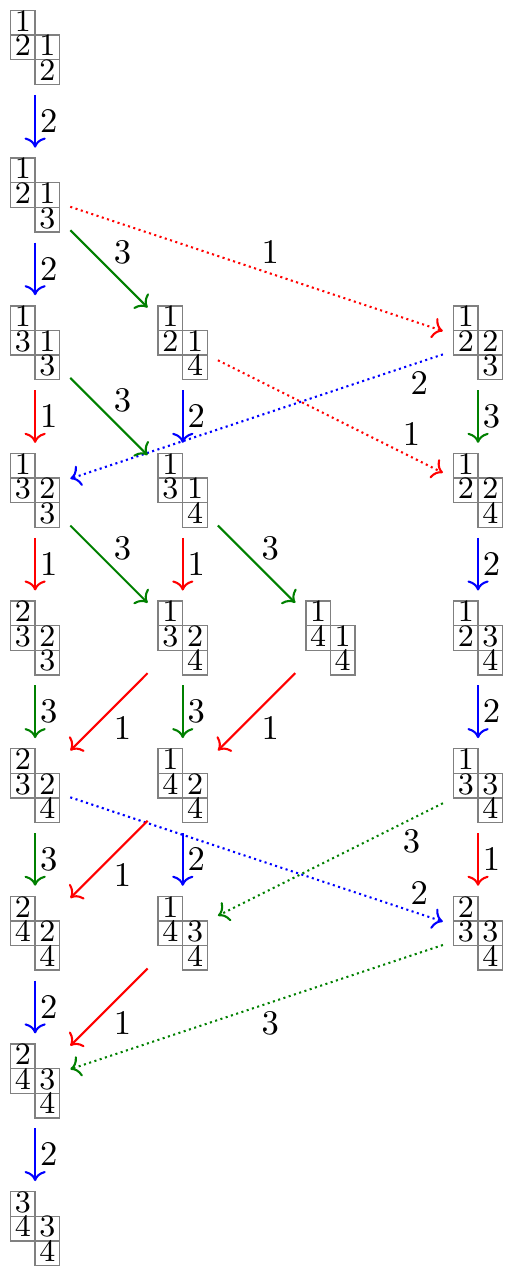}
  \caption{The component $\Gamma(\plac_4,2121)$, with words drawn as quasi-ribbon tabloids of the same shape. Only the
    component $\Gamma(\hypo_4,2132)$ consists of quasi-ribbon words, which thus appear here as quasi-ribbon tableaux.}
\end{figure}

\begin{proposition}
  \label{prop:placccomponentcontainsatmostonehypoqrwcomponent}
  A connected component of $\Gamma(\plac_n)$ contains at most one quasi-ribbon word component of $\Gamma(\hypo_n)$.
\end{proposition}

\begin{proof}
  Suppose the connected component $\Gamma(\plac_n,u)$ contains connected components $\Gamma(\hypo_n,w)$ and
  $\Gamma(\hypo_n,w')$ that both consist of quasi-ribbon words. Without loss of generality, assume that $w$ is highest-
  weight in $\Gamma(\hypo_n,w)$ and $w'$ has highest-weight in $\Gamma(\hypo_n,w')$. Since $w$ and $w'$ are in the
  connected component $\Gamma(\plac_n,u)$, the quasi-ribbon tableaux $\qrtableau{w}$ and $\qrtableau{w'}$ have the same
  shape by \fullref{Proposition}{prop:kekfpreserveqrtshapes}. Hence by
  \fullref{Corollary}{corol:charhighestweighttableaubyshapes}, $\wt{w} = \wt{w'}$ and so $w = w'$ by
  \fullref{Corollary}{corol:qrthighestweightsameweightequal}. Thus $\Gamma(\hypo_n,w) = \Gamma(\hypo_n,w')$.
\end{proof}

It is possible that a connected component of $\Gamma(\plac_n)$ contains no quasi-ribbon word components of
$\Gamma(\hypo_n)$. For example, as can be seen in \fullref{Figure}{fig:gamma4partisom1}, $\Gamma(\plac_4,2211)$ contains
the connected components $\Gamma(\hypo_4,2211)$ and $\Gamma(\hypo_4,2312)$, and neither $2211$ nor $2312$ is a
quasi-ribbon word. Thus the next aim is to characterize those connected components of $\Gamma(\plac_n)$ that contain a
(necessarily unique) quasi-ribbon word component of $\Gamma(\hypo_n)$. In order to do this, it is useful to discuss a
shortcut that allows one to calculate quickly the Young tableau $\P{w}$ obtained when $w$ is a quasi-ribbon word.

The \defterm{slide up--slide left algorithm} takes a filled quasi-ribbon diagram $D$ and produces a filled Young diagram
as follows: Start from the quasi-ribbon diagram $D$. Slide all the columns upwards until the topmost entry of each is
on row $1$. Now slide all the symbols leftwards along their rows until the leftmost entry in each row is in the first
column and there are no gaps in each row.

As will be shown in \fullref[(1)]{Proposition}{prop:slideupslideleftpq}, applying the slide up--slide left algorithm to a
quasi-ribbon tableau $\qrtableau{u}$ gives the Young tableau $\P{u}$, as in the following example:
\begin{align*}
  \qrtableau{1325436768} ={} &
  \begin{tikzpicture}[
    x=5mm,
    y=5mm,
    baseline=(firstmatrix-1-1.base),
    tableaumatrix/.append style={outer sep=0mm},
    ]
    \matrix[tableaumatrix,name=firstmatrix,anchor=north west]{
      1 \& 2                     \\
      \& 3 \& 3                  \\
      \&   \& 4                  \\
      \&   \& 5 \& 6 \& 6        \\
      \&   \&   \&   \& 7 \& 8   \\
    };
  \end{tikzpicture}              \\
  \rightsquigarrow{}&
  \begin{tikzpicture}[
    x=5mm,
    y=5mm,
    baseline=(firstmatrix-1-1.base),
    tableaumatrix/.append style={outer sep=0mm},
    ]
    \draw[gray,thick,->] (2.5,-.6) -- (2.5,0);
    \draw[gray,thick,->] (3.5,-1.8) -- (3.5,0);
    \draw[gray,thick,->] (4.5,-1.8) -- (4.5,0);
    \draw[gray,thick,->] (5.5,-2.4) -- (5.5,0);
    \matrix[tableaumatrix,name=firstmatrix,anchor=north west] at (0,0) {
      1 \& 2 \\
      \& 3   \\
    };
    \matrix[tableaumatrix,anchor=north west] at (2,-.6) {
      3                \\
      4                  \\
      5                  \\
    };
    \matrix[tableaumatrix,anchor=north west] at (3,-1.8) {
      6 \& 6                     \\
      \& 7                       \\
    };
    \matrix[tableaumatrix,anchor=north west] at (5,-2.4) {
      8 \\
    };
  \end{tikzpicture} \displaybreak[0]\\
  \rightsquigarrow{}&
  \begin{tikzpicture}[
    x=5mm,
    y=5mm,
    baseline=(firstmatrix-1-1.base),
    tableaumatrix/.append style={outer sep=0mm},
    ]
    \matrix[tableaumatrix,name=firstmatrix,anchor=north west] at (0,0) {
      1 \& 2 \& 3 \& 6 \& 6 \& 8 \\
      \& 3 \& 4 \&     \& 7      \\
      \&   \& 5                  \\
    };
  \end{tikzpicture} \displaybreak[0]\\
  \rightsquigarrow{}&
  \begin{tikzpicture}[
    x=5mm,
    y=5mm,
    baseline=(firstmatrix-1-1.base),
    tableaumatrix/.append style={outer sep=0mm},
    ]
    \draw[gray,thick,->] (.6,-1.5) -- (0,-1.5);
    \draw[gray,thick,->] (1.2,-2.5) -- (0,-2.5);
    \draw[gray,thick,->] (3.2,-1.5) -- (2.7,-1.5);
    \matrix[tableaumatrix,name=firstmatrix,anchor=north west] at (0,0) {
      1 \& 2 \& 3 \& 6 \& 6 \& 8 \\
    };
    \matrix[tableaumatrix,anchor=north west] at (.6,-1) {
      3 \& 4 \\
    };
    \matrix[tableaumatrix,anchor=north west] at (1.2,-2) {
      5 \\
    };
    \matrix[tableaumatrix,anchor=north west] at (3.2,-1) {
      7 \\
    };
  \end{tikzpicture} \displaybreak[0]\\
  \rightsquigarrow{}&
  \begin{tikzpicture}[
    x=5mm,
    y=5mm,
    baseline=(firstmatrix-1-1.base),
    tableaumatrix/.append style={outer sep=0mm},
    ]
    \matrix[tableaumatrix,name=firstmatrix,anchor=north west] at (0,0) {
      1 \& 2 \& 3 \& 6 \& 6 \& 8 \\
      3 \& 4 \& 7                \\
      5                          \\
    };
  \end{tikzpicture} = \P{1325436768}.
\end{align*}
Similarly, as will be shown in \fullref[(2)]{Proposition}{prop:slideupslideleftpq}, applying the slide up--slide left
algorithm to a quasi-ribbon tableau of the same shape $\alpha$ as $\qrtableau{u}$, filled with entries
$1,\ldots,\clen\alpha$, gives the standard Young tableau $\Q{u}$, as in the following example:
\begin{align*}
  \begin{tikzpicture}[
    x=5mm,
    y=5mm,
    baseline=(firstmatrix-1-1.base),
    tableaumatrix/.append style={outer sep=0mm},
    ]
    \matrix[tableaumatrix,name=firstmatrix,anchor=north west]{
      1 \& 2                     \\
      \& 3 \& 4                  \\
      \&   \& 5                  \\
      \&   \& 6 \& 7 \& 8        \\
      \&   \&   \&   \& 9 \& 10  \\
    };
  \end{tikzpicture}
  \rightsquigarrow{}&
  \begin{tikzpicture}[
    x=5mm,
    y=5mm,
    baseline=(firstmatrix-1-1.base),
    tableaumatrix/.append style={outer sep=0mm},
    ]
    \matrix[tableaumatrix,name=firstmatrix,anchor=north west] at (0,0) {
      1 \& 2 \& 4 \& 7 \& 8 \& 10 \\
      3 \& 5 \& 9                 \\
      6                           \\
    };
  \end{tikzpicture} = \Q{1325436768}.
\end{align*}

\begin{proposition}
  \label{prop:slideupslideleftpq}
  Let $T$ be a quasi-ribbon tableau of shape $\alpha$.
  \begin{enumerate}
  \item Applying the slide up--slide left algorithm to $T$ yields the Young tableau $\P{\colreading{T}}$.
  \item Applying the slide up--slide left algorithm to the \textparens{unique} quasi-ribbon tableau of shape $\alpha$ filled with
    entries $1,\ldots,\cwt\alpha$ yields the standard Young tableau $\Q{\colreading{T}}$.
  \end{enumerate}
\end{proposition}

\begin{proof}
  Let $U$ be the unique quasi-ribbon tableau of shape $\alpha$ filled with entries $1,\ldots,\cwt\alpha$.  The proof is
  by induction on the number of columns $m$ in a ribbon diagram of shape $\alpha$.

  Suppose $m=1$. Then applying the slide up--slide left algorithm to $T$ and $U$ yields $T$ and $U$, respectively (viewed as filled
  Young diagrams). Since $T$ satisifies the condition for being a Young tableau, $\P{\colreading{T}} = T$. Furthermore,
  $U$ is also the unique standard Young tableau with a single column of the same shape as $T$, and so must be
  $\Q{\colreading{T}}$.

  Now suppose $m > 1$ and that the result holds for all quasi-ribbon tableau $T'$ with fewer than $m$ columns. In
  particular, it holds for the quasi-ribbon tableau $T'$ formed by the first $m-1$ columns of $T$. In particular, by
  applying the slide up--slide left algorithm to $T'$, one obtains $\P{\colreading{T'}}$. If the $m$-th column of $T$
  contains entries $a_1,\ldots, a_k$ (listed from bottom to top, so that $a_1 > \ldots > a_k$), then applying the slide
  up--slide left algorithm to $T$ yields the filled Young diagram obtained by applying it to $T'$ (which yields
  $\P{\colreading{T'}}$) and then adding the symbol $a_h$ at the rightmost end of the $h$-th row. Since each symbol
  $a_1,\ldots,a_k$ is greater or equal to than every symbol in $T'$ and thus in $\P{\colreading{T'}}$, using
  \fullref{Algorithm}{alg:hypoplacticinsertionone} to insert the symbols $a_1,\ldots,a_k$ into $\P{\colreading{T'}}$
  does not involve bumping any symbols in $\P{\colreading{T'}}$. That is, the symbols $a_1,\ldots,a_k$ (which form a
  strictly decreasing sequence) bump each other up the rightmost edge of $\P{\colreading{T'}}$, as in the following
  example
  \[
  \begin{tikzpicture}[
    x=5mm,
    y=5mm,
    tableaumatrix/.append style={outer sep=0mm},
    baseline=(baselinematrix-1-1.base),
    ]
    \matrix[tableaumatrix,name=baselinematrix] at (6,0) {a_{3} \\};
    \matrix[tableaumatrix] at (5,-1) {a_{2} \\};
    \matrix[tableaumatrix] at (5,-2) {a_{1} \\};
    \draw[gray,xshift=-2.5mm,yshift=2.5mm] (0,0)--(6,0)--(6,-1)--(5,-1)--(5,-3)--(3,-3)--(3,-4)--(1,-4)--(0,-4)--cycle;
    \node at (2,-1) {$\P{\colreading{T'}}$};
  \end{tikzpicture}
  \leftarrow a_4 =
  \begin{tikzpicture}[
    x=5mm,
    y=5mm,
    tableaumatrix/.append style={outer sep=0mm},
    baseline=(baselinematrix-1-1.base),
    ]
    \matrix[tableaumatrix,name=baselinematrix] at (6,0) {a_{4} \\};
    \matrix[tableaumatrix] at (5,-1) {a_{3} \\};
    \matrix[tableaumatrix] at (5,-2) {a_{2} \\};
    \matrix[tableaumatrix] at (3,-3) {a_{1} \\};
    \draw[gray,xshift=-2.5mm,yshift=2.5mm] (0,0)--(6,0)--(6,-1)--(5,-1)--(5,-3)--(3,-3)--(3,-4)--(1,-4)--(0,-4)--cycle;
    \node at (2,-1) {$\P{\colreading{T'}}$};
  \end{tikzpicture}
  \]

  Furthermore, applying the slide up--slide left algorithm to the unique quasi-ribbon tableau of the same shape as $T'$
  filled with entries $1,\ldots,(\cwt\alpha - k)$ yields the standard Young tableau $\Q{\colreading{T'}}$. Thus applying
  the slide up--slide left algorithm to the unique quasi-ribbon tableau of shape $\alpha$ filled with entries
  $1,\ldots,\cwt\alpha$ yields the filled Young diagram obtained by applying it to $\Q{\colreading{T'}}$ and adding
  $\cwt\alpha - k + h$ to the end of the $h$-th row for $h = 1,\ldots,k$. By the above analysis of the behaviour of
  \fullref{Algorithm}{alg:hypoplacticinsertionone}, this is the standard Young tableau $\Q{\colreading{T}}$.
\end{proof}

\begin{proposition}
  \label{prop:slideupslideleftqfromuniqueqrt}
  Let $Q$ be a standard Young tableau. There is at most one quasi-ribbon tableau $T$ such that $Q$ can be obtained by
  applying the slide up--slide left algorithm to $T$.
\end{proposition}

\begin{proof}
  Suppose that $T$ and $T'$ are quasi-ribbon tableaux such that applying the slide up--slide left algorithm to $T$
  and $T'$ yields $Q$. Then by \fullref[(1--2)]{Proposition}{prop:slideupslideleftpq},
  $\P{\colreading{T}} = \Q{\colreading{T}} = Q = \P{\colreading{T'}} = \Q{\colreading{T'}}$. Thus by the
  Robinson--Schensted--Knuth correspondence, $\colreading{T} = \colreading{T'}$ and so $T = T'$.
\end{proof}

\begin{proposition}
  \label{prop:charplaccomponentscontainingqrt}
  A connected component $\Gamma(\plac_n,w)$ contains a quasi-ribbon word component if and only if $\Q{w}$ can be
  obtained by applying the slide up--slide left algorithm to some quasi-ribbon tableau $T$ of shape $\alpha$ and entries
  $1,\ldots,\cwt\alpha$ and such that $\clen\alpha \leq n$.
\end{proposition}

\begin{proof}
  Suppose $\Gamma(\plac_n,w)$ contains a quasi-ribbon word component. Let $u \in \Gamma(\plac_n,w)$ be a quasi-ribbon
  word. Let $\alpha$ be the shape of $\qrtableau{u}$. Note that $\qrtableau{u}$ has at most $n$ rows, so
  $\clen\alpha \leq n$. By \fullref[(2)]{Proposition}{prop:slideupslideleftpq}, applying the slide up--slide left algorithm
  to the unique quasi-ribbon tableau of shape $\alpha$ filled with entries $1,\ldots,\cwt\alpha$ yields the standard
  Young tableau $\Q{u}$. Since $u$ and $w$ are in the same connected component of $\Gamma(\plac_n)$, the standard Young tableaux $\Q{w}$ and
  $\Q{u}$ are equal.

  On the other hand, suppose that $\Q{w}$ can be obtained by applying the slide up--slide left algorithm to some
  quasi-ribbon tableau $T$ of shape $\alpha$ and entries $1,\ldots,\cwt\alpha$ and such that $\clen\alpha \leq n$. Let
  $U$ be the quasi-ribbon tableau of shape $\alpha$ and entries in $\aA_n$ such that $\colreading{U}$ is highest-weight
  in its component of $\Gamma(\hypo_n)$; note that $U$ exists since $\clen\alpha \leq n$. Then by
  \fullref[(2)]{Proposition}{prop:slideupslideleftpq}, $\Q{\colreading{U}} = \Q{w}$ since $U$ has shape $\alpha$. Hence
  $\colreading{U} \in \Gamma(\plac_n,w)$.
\end{proof}

Define a permutation $\sigma$ of $\set{1,\ldots,n}$ to be \defterm{interval-reversing} if there is some composition
$\alpha = (\alpha_1,\ldots,\alpha_k)$ with $\cwt{\alpha} = n$ such that for all $h = 0,\ldots,k$, the permutation
$\sigma$ preserves the interval $\set{\alpha_1+\ldots+\alpha_h+1,\ldots,\alpha_1+\ldots+\alpha_{h+1}}$ and reverses the
order of its elements. (For $h=0$, this interval is $\set{1,\ldots,\alpha_1}$.) Thus, for example,
\[
\begin{pmatrix}
1 & 2 & 3 & 4 & 5 & 6 & 7 & 8 \\
1 & 5 & 4 & 3 & 2 & 8 & 7 & 6 \\
\end{pmatrix}
\]
is an interval-reversing permutation of $\set{1,\ldots,8}$, where the appropriate composition $\alpha$ is
$(1,4,3)$. It is clear that there is only one choice for $\alpha$.

It is well-known that if $w$ is a standard word (and thus a permutation), then the Robinson--Schensted--Knuth correspondence
associates $w$ with $(\P{w},\Q{w})$ and $w^{-1}$ with $(\Q{w},\P{w})$. That is, $\Q{w} = \P{w^{-1}}$ and
$\P{w} = \Q{w^{-1}}$. Thus involutions are associated to pairs $(Q,Q)$, where $Q$ is a standard Young tableau.

\begin{proposition}
  A connected component $\Gamma(\plac_n,w)$ contains a quasi-ribbon word component if and only if the
  Robinson--Schensted--Knuth correspondence associates $(\Q{w},\Q{w})$ with an interval-reversing involution.
\end{proposition}

\begin{proof}
  Suppose $\Gamma(\plac_n,w)$ contains a quasi-ribbon word. By
  \fullref{Proposition}{prop:charplaccomponentscontainingqrt}, $\Q{w}$ can be obtained by applying the slide up--slide
  left algorithm to some quasi-ribbon tableau $T$ of shape $\beta$ and entries $1,\ldots,\cwt\beta$ and such that
  $\clen\beta \leq n$. By \fullref[(1--2)]{Proposition}{prop:slideupslideleftpq},
  $\P{\colreading{T}} = \Q{\colreading{T}} = \Q{w}$. So the Robinson--Schensted--Knuth correspondence associates
  $(\Q{w},\Q{w})$ wih $\colreading{T}$. Let $\alpha = (\alpha_1,\ldots,\alpha_k)$ be such that $\alpha_i$ is the length of the
  $i$-th column of $T$. Then by the definition of $T$ and its column reading, $\colreading{T}$ is an interval-reversing
  permutation where the appropriate composition is $\alpha$.

  On the other hand, suppose the Robinson--Schensted--Knuth correspondence associates $(\Q{w},\Q{w})$ with an
  interval-reversing involution $u$, where the appropriate composition is $\alpha = (\alpha_1,\ldots,\alpha_k)$. Note
  that $\Q{w} = \P{u}$ and $\Q{w} = \Q{u}$. Let $T$ be the quasi-ribbon tabloid $\qrtabloid{u}$. Then $\alpha_i$ is the
  length of the $i$-th column of $T$, which is filled with the image of the interval
  $\set{\alpha_1+\ldots+\alpha_i+1,\ldots,\alpha_1+\ldots+\alpha_{i+1}}$ since $u$ is an interval-reversing
  permutation. Since the elements in each interval are less than the elements in the next, the rows of $T$ are
  increasing from left to right and so $T$ is a quasi-ribbon tableau that is filled with the symbols
  $1,\ldots,|u|$. Note that $u = \colreading{T}$, and so $\Q{w} = \Q{u} = \Q{\colreading{T}}$. So by
  \fullref{Proposition}{prop:slideupslideleftpq}, $\Q{w}$ can be obtained from $T$ by applying the slide up--slide left
  algorithm. Hence $\Gamma(\plac_n,w)$ contains a quasi-ribbon word by
  \fullref{Proposition}{prop:charplaccomponentscontainingqrt}.
\end{proof}

\begin{corollary}
  \label{corol:countisoplaccomponentscontainingqrt}
  Let $w \in \aA_n^*$ and let $\lambda = (\lambda_1,\ldots,\lambda_{\plen\lambda})$ be the shape of $\P{w}$. The number of
  connected components of $\Gamma(\plac_n)$ that are isomorphic to $\Gamma(\plac_n,w)$ and contain quasi-ribbon word
  components of $\Gamma(\hypo_n)$ is
  \[
  \begin{cases}
    {\displaystyle\binom{\lambda_1}{\lambda_1-\lambda_2,\;\;\ldots,\;\;\lambda_{\plen\lambda-1}-\lambda_{\plen\lambda},\;\;\lambda_{\plen\lambda}}} & \text{if $\pwt\lambda - \lambda_1 +1 \leq n$,} \\
    0 & \text{otherwise.}
  \end{cases}
  \]
\end{corollary}

\begin{proof}
  By \fullref{Proposition}{prop:charplaccomponentscontainingqrt}, the number of such components is equal to the number
  of standard Young tableaux of shape $\lambda$ that can be obtained by applying the slide up--slide left algorithm to
  some quasi-ribbon tableau of some shape $\alpha$ and entries $1,\ldots,\cwt\alpha$ and such that $\clen\alpha \leq n$. By
  the definition of a quasi-ribbon diagram, the number of rows $\clen\alpha$ of such a diagram is equal to the total
  number of boxes minus the number of columns plus $1$ (since adjacent pairs of columns overlap in exactly one
  row). Since the slide up--slide left algorithm preserves the number of columns and the total number of boxes, this
  number of rows is $\pwt\lambda - \lambda_1 + 1$. Hence if $\pwt\lambda - \lambda_1 + 1 > n$, there is no standard
  Young tableau that can be obtained in such a way. So suppose henceforth that $\pwt\lambda - \lambda_1 + 1 \leq n$.

  For any composition $\alpha$, there is exactly one quasi-ribbon tableau of shape $\alpha$ filled with symbols
  $1,\ldots,\cwt\alpha$, and each of them yields (under the slide up--slide left algorithm) a different standard Young
  tableau by \fullref{Proposition}{prop:charplaccomponentscontainingqrt}. Furthermore, there is a one-to-one
  correspondence between compositions $\alpha$ and shapes of tabloids: one simply takes a ribbon diagram of shape
  $\alpha$ and applies the `slide up' part of the slide up--slide left algorithm. Thus, to obtain the number of such
  quasi-ribbon tableaux that give the correct shape on applying the slide up--slide left algorithm, it suffices to count
  the number of tabloid shapes that have $\lambda_j$ cells on the $j$-th row, for each $j = 1,\ldots,\plen\lambda$. This
  is equal to the number of tabloid shapes that have $\lambda_1$ cells on the first row and $\lambda_j-\lambda_{j+1}$
  columns ending on the $j$-th row, which in turn is equal to the number of words of length $\lambda_1$ containing
  $\lambda_j-\lambda_{j+1}$ symbols $j$ for each $j < \plen\lambda$ and $\lambda_{\plen\lambda}$ symbols $\plen\lambda$.
\end{proof}

Note that the number of components specified by \fullref{Corollary}{corol:countisoplaccomponentscontainingqrt} is
dependent only on the shape of $\P{w}$, not on $\wt{w}$ or $w$ itself. This is what one would expect in working with a
class of isomorphic components in $\Gamma(\plac_n)$.

\begin{corollary}
  Let $w \in \aA_n^*$, and let $\lambda = (\lambda_1,\ldots,\lambda_m)$ be the shape of $\P{w}$. Suppose
  $\pwt\lambda - \lambda_1 + 1 \leq n$. Then $\Gamma(\plac_n,w)$ contains at least
  \[
    \binom{\lambda_1}{\lambda_1-\lambda_2,\;\;\ldots,\;\;\lambda_{\plen\lambda-1}-\lambda_{\plen\lambda},\;\;\lambda_{\plen\lambda}}
  \]
  connected components of $\Gamma(\hypo_n)$.
\end{corollary}

\begin{proof}
  \fullref{Corollary}{corol:countisoplaccomponentscontainingqrt} gives the number of components isomorphic to
  $\Gamma(\plac_n,w)$ that contain quasi-ribbon word components. Each one of these quasi-ribbon word components is
  isomorphic to a connected component of $\Gamma(\hypo_n)$ inside $\Gamma(\plac_n,w)$ by
  \fullref{Proposition}{prop:isomorphismsrestrict}, and to \emph{distinct} such connected components of
  $\Gamma(\hypo_n)$ by \fullref{Corollary}{corol:qrtcomponentsindifferentplaces}. Hence the number of connected
  components of $\Gamma(\hypo_n)$ inside $\Gamma(\plac_n,w)$ must be at least the number given in
  \fullref{Corollary}{corol:countisoplaccomponentscontainingqrt} in the case where $\pwt\lambda - \lambda_1 + 1 \leq n$.
\end{proof}

\subsection{The Sch\"utzenberger involution}

The \defterm{Sch\"utzenberger involution} is the map $\schinvlit : \aA_n^* \to \aA_n^*$ that sends $a \in \aA_n$ to
$n-a+1$ and is extended to $\aA_n^*$ by $\schinv{(a_1\cdots a_k)} \mapsto \schinv{a_k}\cdots \schinv{a_1}$. That is,
given a word, one obtains its image under $\schinvlit$ by reversing the word and replacing each symbol $a$ by
$n-a+1$. It is well-known that the Sch\"utzenberger involution reverses the order of weights, in the sense that if $u$
has higher weight than $v$, then $\schinv{v}$ has higher weight than $\schinv{u}$.

\begin{proposition}
  In $\Gamma(\hypo_n)$, the Sch\"utzenberger involution maps connected components to connected components. If there is
  an edge from $u$ to $v$ labelled by $i$, then there is an edge from $\schinv{v}$ to $\schinv{u}$ labelled by $n-i$.
\end{proposition}

\begin{proof}
  Clearly the second statement implies the first. So suppose that $\f_i(u) = v$. Then $u$ contains at least one symbol
  $i$ and, every symbol $i$ is to the left of every symbol $i+1$, and $v$ is obtained from $u$ by replacing the
  rightmost symbol $i$ by $i+1$. By the definition of $\schinvlit$, in the word $\schinv{u}$, every symbol $n-i+1$ is to
  the right of every symbol $n-i$. Hence $\e_{n-i}(\schinv{u})$ is defined, is equal to the word obtained from
  $\schinv{u}$ by replacing the leftmost symbol $n-i+1$ by $n-i$, which is $\schinv{v}$. Hence there is an edge from
  $\schinv{v}$ to $\schinv{u}$ labelled by $n-i$.
\end{proof}

\subsection{Characterizing quasi-crystal graphs?}

Stembridge \cite{stembridge_characterization} gives a set of axioms that characterize connected components of crystal
graphs. These axioms specify `local' conditions that the graph must satisfy. It is natural to ask whether there is an
analogous characterization for quasi-crystal graphs:

\begin{question}
\label{qu:characterization}
Is there a local characterization of quasi-crystal graphs?
\end{question}

However, the Stembridge axioms are connected with the underlying representation theory, in the sense that they refer to
whether the arrow labels correspond to orthogonal roots of the algebra. Since the quasi-crystal graphs are defined on a
purely combinatorial level, any characterization of them must also be on a purely combinatorial level.

\section{Applications}
\label{sec:applications}

\subsection{Counting factorizations}

One interpretation of the Littlewood--Richardson rule \cite[ch.~5]{fulton_young} is that the Littlewood--Richardson
coefficients $c_{\lambda\mu}^\nu$ give the number of different factorizations of an element of $\plac$ corresponding to
a Young tableau of shape $\nu$ into elements corresponding to a tableau of shape $\lambda$ and a tableau of shape
$\mu$. In particular, it shows that the number of such factorizations is independent of the content of the tableau. The
quasi-crystal structure yields a similar result for $\hypo$.

\begin{theorem}
  The number of distinct factorizations of an element of the hypoplactic monoid corresponding to a quasi-ribbon tableau
  of shape $\gamma$ into elements that correspond to tableau of shape $\alpha$ and $\beta$ is dependent only of $\gamma$,
  $\alpha$, and $\beta$, and not on the content of the element.
\end{theorem}

\begin{proof}
  Let $\gamma,\alpha,\beta$ be compositions. Let $w \in \aA_n^*$ be a quasi-ribbon word such that $\gamma$ is the
  shape of $\qrtableau{w}$. Let
  \begin{align*}
  S^w_{\alpha,\beta} = \gsetsplit[\big]{(u,v)}{\;&\text{$u,v \in \aA_n^*$ are quasi-ribbon words,} \\
    &\text{$\qrtableau{u}$ has shape $\alpha$,} \\
    &\text{$\qrtableau{v}$ has shape $\beta$,} \\
    &w \hypocong uv}.
  \end{align*}
  So $S^w_{\alpha,\beta}$ is a complete list of factorizations of $w$ into elements whose corresponding tableaux have
  shapes $\alpha$ and $\beta$.  Let $i \in \set{1,\ldots,n-1}$ and suppose $\e_i(w)$ is defined. Pick some pair
  $(u,v) \in S^w_{\alpha,\beta}$. Then $w \hypocong uv$ and so $\e_i(uv)$ is defined, and $\e_i(w) = \e_i(uv)$. By
  \fullref{Lemma}{lem:efdefinedpreservedbyiso}, either $\e_i(uv) = u\e_i(v)$, or $\e_i(uv) = \e_i(u)v$. In the former
  case, $(u,\e_i(v)) \in S_{\e_i(w)}$, and in the latter case $(\e_i(u),v) \in S_{\e_i(w)}$, since $\e_i$ preserves
  being a quasi-ribbon word and the shape of the corresponding quasi-ribbon tableau by
  \fullref{Corollary}{corol:qrwsameshapecomponent}. So $\e_i$ induces an injective map from
  $S^w_{\alpha,\beta}$ to $S^{\e_i(w)}_{\alpha,\beta}$. It follows that $\f_i$ induces the inverse map from
  $S^{\e_i(w)}_{\alpha,\beta}$ to $S^w_{\alpha,\beta}$. Hence
  $\abs{S^w_{\alpha,\beta}} = \abs{S^{\e_i(w)}_{\alpha,\beta}}$. Similarly, if $\f_i(w)$ is defined, then
  $\abs{S^w_{\alpha,\beta}} = \abs{S^{\f_i(w)}_{\alpha,\beta}}$. Since all the quasi-ribbon words whose tableaux have
  shape $\gamma$ lie in the same connected component of $\Gamma(\hypo_n)$, it follows that $|S^w_{\alpha,\beta}|$ is
  dependent only on $\gamma$, not on $w$.
\end{proof}

\subsection{Conjugacy}

There are several possibile generalizations of conjugacy from groups to monoids. One possible definition, introduced
by Otto~\cite{otto_conjugacy}, is \defterm{$o$-conjugacy}, defined on a monoid $M$ by
\begin{equation}
\label{eq:oconjdef}
x \sim_o y \iff (\exists g,h \in M)(xg = gy \land hx = yh).
\end{equation}
The relation $\sim_o$ is an equivalence relation. Another approach is \defterm{primary conjugacy} or
\defterm{$p$-conjugacy}, defined on a monoid $M$ by
\[
x \sim_p y \iff (\exists u,v \in M)(x = uv \land y = vu).
\]
However, $\sim_p$ is reflexive and symmetric, but not transitive; hence it is sensible to follow Kudryavtseva \&
Mazorchuk \cite{kudryavtseva_three,kudryavtseva_conjugation} in working with its transitive closure $\sim_p^*$. It is
easy to show that ${\sim_p} \subseteq {\sim_p^*} \subseteq {\sim_o}$. In some circumstances, equality holds. For
instance, ${\sim_p^*} = {\sim_o}$ in $\plac$, and two elements of $\plac$ are related by $\sim_p^*$ and $\sim_o$ if and
only if they have the same weight \cite[Theorem~17]{choffrut_lexicographic}. Since $\hypo$ is a quotient of $\plac$, the
same result holds in $\hypo$, but the machinery of quasi-crystals gives a very quick proof that two elements with the
same weight are $\sim_o$-related:

\begin{proposition}
  Let $u,v \in \hypo_n$ be such that $\wt{u} = \wt{v}$. Then $u \sim_o v$.
\end{proposition}

\begin{proof}
  Let $g = n(n-1)\cdots 21$. Then $ug$ is highest-weight by \fullref{Proposition}{prop:charhighestweight}, because it
  has an $i$-inversion for all $i \in \set{1,\ldots,n-1}$, and thus no $\e_i$ is defined. Similarly $gv$ is
  highest-weight. Furthermore, $\wt{ug} = \wt{u} + \wt{g} = \wt{v} + \wt{g} = \wt{gv}$. Hence $ug \hypocong
  gv$. Similarly, $gu \hypocong vg$ and so $u \sim_o v$.
\end{proof}

\subsection{Satisfying an identity}

Another application of the quasi-crystal structure is to prove that the hypoplactic monoid satisfies an identity. It is
known that $\plac_1$, $\plac_2$, and $\plac_3$ satisfy identities, but whether $\plac_n$ satisfies an identity (perhaps
dependent on $n$) is an important open problem \cite{kubat_identities}.

\begin{theorem}
  The hypoplactic monoid satisfies the identity $xyxy = yxyx$.
\end{theorem}

\begin{proof}
  Let $x,y \in \aA_n^*$. The idea is to raise the products $xyxy$ and $yxyx$ to highest-weight using the same sequence
  of operators $\e_i$, deduce that the corresponding highest-weight words are equal, and so conclude that
  $xyxy \hypocong yxyx$. More formally, the proof proceeds by reverse induction on the weight of $xyxy$.

  The base case of the induction is when $xyxy$ is highest-weight. Thus $\e_i(xyxy)$ is undefined for all
  $i \in \set{1,\ldots,n-1}$. So $xyxy$ must have an $i$-inversion for all $i \in
  \set{1,\ldots,\max(xyxy)-1}$.
  The symbols $i+1$ and $i$ may each lie in $x$ or $y$, but in any case, there is an $i$-inversion in $yxyx$. Hence
  $\e_i(yxyx)$ is undefined for all $i \in \set{1,\ldots,n-1}$. So $yxyx$ is also highest-weight. Clearly
  $\wt{xyxy} = \wt{yxyx}$, and so $\qrtableau{xyxy} = \qrtableau{yxyx}$ by
  \fullref{Corollary}{corol:qrthighestweightsameweightequal}. Hence $xyxy \hypocong yxyx$.

  For the induction step, suppose $xyxy$ is not highest-weight, and that $x'y'x'y' \hypocong y'x'y'x'$ for all
  $x',y' \in \aA_n^*$ such that $x'y'x'y'$ has higher weight than $xyxy$. Then $\e_i(xyxy)$ is defined for some
  $i \in \set{1,\ldots,n-1}$. Neither $x$ nor $y$ contains a symbol $i$, since otherwise there would be an $i$-inversion
  in $xyxy$. So
  $\ecount_i(xyxy) = \ecount_i(yxyx) = \abs{xyxy}_{i+1} = \abs{yxyx}_{i+1} = 2\abs{x}_{i+1} + 2\abs{y}_{i+1} =
  2\ecount_i(x) + 2\ecount_i(y)$,
  and $\ecount_i(xyxy)$ applications of $i$ change every symbol $i+1$ in $xyxy$ to $i$. That is,
  $\e_i^{\ecount_i(xyxy)}(xyxy)$ and $\e_i^{\ecount_i(yxyx)}(yxyx)$ are both defined and are equal to $x'y'x'y'$ and
  $y'x'y'x'$ respectively, where $x' = \e_i^{\ecount_i(x)}(x)$ and $y' = \e_i^{\ecount_i(y)}(y)$. Since $x'y'x'y'$ has
  higher weight than $xyxy$, it follows by the induction hypothesis that $x'y'x'y' \hypocong y'x'y'x'$. Hence
  \[
  xyxy = \f_i^{\ecount_i(xyxy)}(x'y'x'y') \hypocong \f_i^{\ecount_i(xyxy)}(y'x'y'x') = yxyx.
  \]
  This completes the induction step and thus the proof.
\end{proof}

\bibliography{combinatorics,combinatorialsemigroups,presentations,rewriting,semigroups,c_publications,\jobname}
\bibliographystyle{alphaabbrv}

\end{document}